\documentclass{Tran-l}

\usepackage{amssymb}
\usepackage{amsmath}

\usepackage{diagrams}
\newtheorem{theorem}{Theorem}[section]

\newtheorem{corollary}[theorem]{Corollary}

\newtheorem{definition}[theorem]{Definition}
\newtheorem{example}[theorem]{Example}

\newtheorem{lemma}[theorem]{Lemma}

\newtheorem{proposition}[theorem]{Proposition}
\newtheorem{remark}[theorem]{Remark}

\newarrow {Partial} ----{harpoon}
\newarrow{Into} C--->
\newarrow{Onto} ----{>>}

\begin{document}

\title{Hilsum-Skandalis maps as Frobenius adjunctions with application to geometric morphisms}

\author{Christopher Townsend}
\maketitle
\begin{abstract}
Hilsum-Skandalis maps, from differential geometry, are studied in the context of a cartesian category. It is shown that Hilsum-Skandalis maps can be represented as stably Frobenius adjunctions. This leads to a new and more general proof that Hilsum-Skandalis maps represent a universal way of inverting essential equivalences between internal groupoids. 

To prove the representation theorem, a new characterisation of the connected components adjunction of any internal groupoid is given. The characterisation is that the adjunction is covered by a stable Frobenius adjunction that is a slice and whose right adjoint is monadic. 

Geometric morphisms can be represented as stably Frobenius adjunctions. As applications of the study  we show how it is easy to recover properties of geometric morphisms, seeing them as aspects of properties of stably Frobenius adjunctions. 
\end{abstract}

\section{Introduction}
We provide a categorical account of Hilsum-Skandalis maps and establish that key results about them can be proved in the general context of a cartesian category $ \mathcal{C}$, provided the domain and codomain groupoids are restricted to those that have connected components adjunctions that are stably Frobenius. The restriction is not a strong one as most groupoids used in application satisfy this property. One key result is that Hilsum-Skandalis maps form the morphisms of a category of fractions, inverting essential euivalences between internal groupoids. This is shown by providing a new result which is that Hilsum-Skandalis maps can be represented as stably Frobenius adjunctions over $\mathcal{C}$. 

Along the way, in order to ease the discussion of categories of objects with groupoid actions, we provide a new categorical description of the connected components adjunction for internal groupoids which should be of general interest. The result is that an adjunction $\Sigma_{\mathcal{D}} \dashv \mathcal{D}^* : \mathcal{D} \pile{\rTo \\ \lTo}\mathcal{C} $ is equivalent to a connected components adjunction if and only if it is covered by a slice of $\mathcal{C}$, with a covering adjunction that is stably Frobenius and whose right adjoint is monadic.

A number of applications are given which show how the techniques developed can be applied to recover known aspects of the theory of geometric morphisms. It is shown how the pullback of bounded geometric morphisms can be described by using the product of localic groupoids and a proof is given of pullback stable  hyperconnected-localic factorization.

\emph{All our categories are cartesian (i.e. they have finite limits).}

\section{Summary of contents}
Our first few sections are categorical preliminaries covering properties of categories of adjunctions over a base category, Frobenius reciprocity, monadicity and effective descent. Then there is a section on groupoids which provides new characterisations of the connected components adjunction, in terms of these categorical concepts (adjunctions over a base, monadicity, Frobenius reciprocity and effective descent). 

We then define Hilsum-Skandalis maps and show that stably Frobenius adjunctions give rise to them. In fact by restricting to what we have called stably Frobenius groupoids, which are then introduced, we see that all Hilsum-Skandalis maps arise in this way. This is the main result of the paper: Hilsum-Skandalis maps are essentially the same thing as stably Frobenius adjunctions. Along the way there is some discussion about internal essential equivalences between internal groupoids, showing that they can be used to characterise Hilsum-Skandalis maps in the usual manner. We then focus on applications of the main result and show relative to an arbitrary cartesian category that:

(i) Hilsum-Skandalis maps universally invert essential equivalences,

(ii) stably Frobenius groupoids internal to a category of $\mathbb{G}$-objects are (up to a natural equivalence) the same thing as Hilsum-Skandalis maps to $\mathbb{G}$ (this is achieved with a new categorically flavoured construction of groupoid semidirect product) and

(iii) connected component adjunctions of stably Frobenius groupoids are pullback stable.

Finally the paper focuses on applications to geometric morphisms, recalling that geometric morphisms can be represented as certain stably Frobenius adjunctions. The categorical characterisation of connected component adjunctions is applied to characterise bounded geometric morphisms, from which basic results about localic and bounded geometric morphisms follow. The pullback of connected component adjunctions is used to describe the pullback of bounded geometric morphisms. An account of the hyperconnected-localic factorisation is given.

\section{Preliminaries}
\subsection{Notation for pullback}\label{prelim}

We will assume that all our categories are cartesian and start with some notation. If $f : Y \rTo X$ is a morphism of a category $\mathcal{C}$ then we use $Y_f$ as notation for $f$ when considered as an object of the slice category $\mathcal{C}/X$. Observe that $(\mathcal{C}/X)/Y_f$ is isomorphic to $\mathcal{C}/Y$. We use $Y_X$ as notation for the object $\pi_1 : X \times Y  \rTo X $. For any morphism $f:Y\rTo X$ of a cartesian category $\mathcal{C}$ there is a pullback adjunction $\Sigma_f \dashv f^*: \mathcal{C}/Y \rTo \mathcal{C}/X$; if $Z_g$ is an object of $\mathcal{C}/Y$ then $\Sigma_f(Z_g)$ is defined to be $Z_{fg}$. In the case $X=1$ we use $\Sigma_Y \dashv Y^*$ for the pullback adjunction; $\Sigma_Y$ reflects isomorphisms and creates coequalizers.

\subsection{Adjunctions over a base category}\label{adj}
We shall frequently be discussing results relative to a base category $\mathcal{C}$. If we have two other categories $\mathcal{D}_i,i=1,2,$ each equiped with an adjunction $\Sigma_{\mathcal{D}_i} \dashv \mathcal{D}^*_i : \mathcal{D}_i \pile{\rTo \\ \lTo}\mathcal{C} $ back to $\mathcal{C}$ (so $\Sigma_{\mathcal{D}_i}$ goes from $\mathcal{D}_i$ to $\mathcal{C}$) then an adjunction $L \dashv R : \mathcal{D}_1 \pile{\rTo \\ \lTo} \mathcal{D}_2$ is said to be \emph{over $\mathcal{C}$} provided there is a natural isomorphism $\tau:\Sigma_{\mathcal{D}_2} L  \rTo^{\cong}  \Sigma_{\mathcal{D}_1} $. The morphisms of the category of adjunctions over $\mathcal{C}$ consist of natural transformations on left adjoints that commute with the natural isomorphisms $\tau$ in the obvious manner; i.e. if $\alpha: L_1 \rTo L_2$ then $\tau_1:  \Sigma_{\mathcal{D}_2} L_1\rTo \Sigma_{\mathcal{D}_1} $ factors as $ \Sigma_{\mathcal{D}_2} L_1 \rTo^{\Sigma_{\mathcal{D}_2}\alpha} \Sigma_{\mathcal{D}_2}L_2 \rTo^{\tau_2}  \Sigma_{\mathcal{D}_1} $.

If $X$ is an object of $\mathcal{C}$ and we assert that an adjunction $L \dashv R : \mathcal{C}/X \pile{\rTo \\ \lTo} \mathcal{D}$ is over $\mathcal{C}$ then we mean that there is a natural isomorphism $\alpha: \Sigma_{\mathcal{D}}L \rTo^{\cong} \Sigma_X$. Notice that for the case $X=1$ having such an adjunction means that there is a splitting of the given adjunction $\Sigma_{\mathcal{D}} \dashv \mathcal{D}^*$ back to $\mathcal{C}$.

Given an adjunction $L\dashv R:\mathcal{D}\pile{\rTo \\ \lTo} \mathcal{C}$, then for any object $X$ of $\mathcal{C}$ there is a \emph{sliced adjunction} $L_X \dashv R_X : \mathcal{D}/RX \pile{\rTo \\ \lTo} \mathcal{C}/X$; $L_X(g)=$`the adjoint transpose of $g$' and $R_X(Y_f)=RY_{Rf}$. 

It is a trivial from the definition of adjoint transpose to see that any  adjunction $L \dashv R$ (between cartesian categories) factors through its slice at $L1$:
\begin{eqnarray*}
\mathcal{D} \pile{\rTo^{\Sigma_{\eta_1}} \\ \lTo_{\eta_1^*}} \mathcal{D}/RL1  \pile{ \rTo^{L_{L1}} \\ \lTo_{R_{L1}}}   \mathcal{C}/L1 \pile{\rTo^{\Sigma_{L1}} \\ \lTo_{L1^*}} \mathcal{C}
\end{eqnarray*}
and similar observations can be exploited to relate cetain adjunctions over $\mathcal{C}$ to adjunctions that are splittings:
\begin{lemma}\label{firstlemma}
Given an adjunction $\Sigma_{\mathcal{D}} \dashv \mathcal{D}^* : \mathcal{D} \pile{\rTo \\ \lTo}\mathcal{C} $ and an object $X$ of $\mathcal{C}$ then the category of adjunctions $L \dashv R : \mathcal{C}/X \pile{\rTo \\ \lTo} \mathcal{D}$ over $\mathcal{C}$ is equivalent to the category of adjunctions $L'\dashv R' : \mathcal{C}/X \pile{\rTo \\ \lTo} \mathcal{D}/\mathcal{D}^*X$ over $\mathcal{C}/X$. Further, under this equivalence, $L'1 = L1 \rTo{\psi} \mathcal{D}^*X$ where $\psi$ is the adjont transpose, across $\Sigma_{\mathcal{D}} \dashv \mathcal{D}^*$, of the isomorphism $\alpha_1: \Sigma_{\mathcal{D}}L1 \rTo^{\cong} \Sigma_X 1 = X$ that exists by assumption that $L \dashv R$  is over $\mathcal{C}$.
\end{lemma}

\begin{proof}
Send $L  \dashv R$ to $L_{\mathcal{D}^*X}\Sigma_{\Delta_X} \dashv \Delta^*_X R_{\mathcal{D}^*X}$. The adjunction $L_{\mathcal{D}^*X}\Sigma_{\Delta_X} \dashv \Delta^*_X R_{\mathcal{D}^*X}$ is over $\mathcal{C}/X$ because $(\Sigma_X)_X \cong (\Sigma_{\mathcal{D}})_X L_{\mathcal{D}^*X}$ (if $\Sigma_X \cong \Sigma_{\mathcal{D}}L$) and $Id_{\mathcal{C}/X}$ factors as $\Sigma_{\pi_2}\Sigma_{\Delta_X}$.
In the other direction send $L' \dashv R'$ to $\Sigma_{\mathcal{D}^*X}L' \dashv R' (\mathcal{D}^*X)^* $. Observe that
\begin{eqnarray*}
\mathcal{C}/X \times X \cong (\mathcal{C}/X)/X^*X \cong (\mathcal{C}/X) / R \mathcal{D}^*X \text{ \em \em (*)}
\end{eqnarray*}
Note that the adjoint transpose of the morphism $\Delta_X: 1 \rTo X^*X$ of $\mathcal{C}/X$ under the adjunction $\Sigma_X \dashv X^*$ is the identity on $X$ and so the `further' part follows by uniquenss of adjoint transpose, taking the isomorphisms of categories (*) into consideration.
\end{proof}

Being able to relate general adjunctions over $\mathcal{C}$ with sliced domain, to splittings of adjunctions back to $\mathcal{C} $ will ease proofs considerably as the split case is generally easy. We shall see a number of examples of this in what follows.

\section{Frobenius reciprocity} 
An adjunction $L\dashv R:\mathcal{D}\pile{\rTo \\ \lTo} \mathcal{C}$ satisfies \emph{Frobenius reciprocity} (or \emph{is Frobenius}) provided the map
\begin{eqnarray*}
L(W \times RX)\rTo^{(L\pi_{1},L\pi_{2})}LW\times LRX\rTo^{Id_{LW} \times \epsilon_{X} } LW \times X
\end{eqnarray*}
is an isomorphism for all objects $W$ and $X$ of $\mathcal{D}$ and $\mathcal{C}$ respectively, where $\epsilon$ is the counit of the adjunction. Between cartesian closed categories an adjunction satisfies Frobenius reciprocity if and only if the right adjoint preserves exponentials ($R(X^Y) \cong RX^{RY}$), however the condition is also relevant in the general context of cartesian categories, a striking example being the category of locales. Categories of locales (relative to a topos) are not cartesian closed but the Frobenius condition on adjunctions between them can be used to characterise geometric morphsms, \cite{towgeom}. Being Frobenius is two conditions away from being an equivalence:
\begin{lemma}\label{FrobisEq}
A Frobenius adjunction $L\dashv R: \mathcal{D}\pile{\rTo \\ \lTo} \mathcal{C}$ is an equivalence iff $L1 \cong 1$ and $\eta_W$ is a regular monomorphism for each object $W$ of $\mathcal{D}$ then the adjunction is an equivalence.
\end{lemma}
\begin{proof}
If the adjunction is an equivalence the unit is an isomorphism and $L1\cong 1$, so one way round is clear. In the other direction, $LRX \cong L(1 \times RX) \cong L1 \times X \cong X$; from which the counit $\epsilon$ is seen to be an isomorphism. For each object $W$ there is an equalizer diagram $ W \rInto^{\eta_W} RLW \pile{ \rTo^a \\ \rTo_b} W'$ for some $a,b$ and $W'$. Then $La = L b$ since $L \eta_W$ is an isomorphism using the triangular identities and the fact that $\varepsilon$ is an isomorphism. But if $La = L b$ then they have equal adjoint transposes; that is, $ \eta_{W'} a = \eta_{W'} b$ and so $ a = b $ since $ \eta_{W'}$ is a monomorphism. Hence $\eta_W$ is an isomorphism and the adjunction in fact is an equivalence.
\end{proof}
 
If $X$ is an object of $\mathcal{C}$, then $L \dashv R $ is \emph{Frobenius at $X$} provided $L_X \dashv R_X$ satisfies Frobenius reciprocity and is \emph{stably Frobenius} if $L_X \dashv R_X$ satisfies Frobenius reciprocity for every object $X$ of $\mathcal{C}$.

It is easy to see that an adjunction $L\dashv R:\mathcal{D}\pile{\rTo \\ \lTo} \mathcal{C}$ is stably Frobenius if and only if for any pair of morphisms $g:W \rTo RX$ and $f: Y \rTo X$, the canonical map $L( W \times_{RX} RY) \rTo LW \times_X Y$ is an isomorphism, where the codomain is the pullback of $f$ along the adjoint transpose of $g$. For example any morphism $f:X\rTo Y$ of a cartesian category $\mathcal{C}$ the pullback adjunction $\Sigma_f \dashv f^*: \mathcal{C}/X \rTo \mathcal{C}/Y$ is stably Frobenius. The composition of two (stably) Frobenius adjunctions is (stably) Frobenius. 

The next lemma provides an example of the usefulness of being able to relate adjunctions over $\mathcal{C}$ to split adjunctions. We will see towards the end of this paper that Part 2 can be viewed as the categorical essence of the well known fact that if the composite $fg$ of two geometric morphisms is localic, then $g$ is localic.

\begin{lemma}\label{localic}
Let $\Sigma_{\mathcal{D}} \dashv \mathcal{D}^* : \mathcal{D} \pile{\rTo \\ \lTo}\mathcal{C} $ be an adjunction.

1. If $L \dashv R: \mathcal{C} \pile{ \rTo \\ \lTo} \mathcal{D}$ is over $\mathcal{C}$ and Frobenius at $L1$, then $L_1: \mathcal{C}/1 \rTo \mathcal{D}/L1$, is an equivalence.

2. If $X$ is an object of $\mathcal{C}$ and $L \dashv R: \mathcal{C}/X \pile{ \rTo \\ \lTo} \mathcal{D}$ is an adjunction over $\mathcal{C}$ that is Frobenius at $\mathcal{D}^*X$ then $L_1: \mathcal{C}/X \rTo \mathcal{D}/L1$ is an equivalence.
\end{lemma}

\begin{proof}
1. We need to show that the adjunction
\begin{eqnarray*}
\mathcal{C} \pile{\rTo^{\Sigma_{\eta_1}} \\ \lTo_{\eta_1^*}} \mathcal{C}/RL1  \pile{ \rTo^{L_{L1}} \\ \lTo_{R_{L1}}}   \mathcal{D}/L1 
\end{eqnarray*}
is an equivalence. Lemma \ref{FrobisEq} can be applied as the adjunction satisfies Frobenius reciprocity by assumption (and the fact that the composition of two Frobenius adjunctions is Frobenius). Since $L_11 \cong 1$ we are just required to check that the unit of this adjunction is a regular monomorphism. At $X$, an object of $\mathcal{C}$, the unit is $X \rTo^{(!^X,\eta_X)} 1 \times_{RL1} RLX$ where the pullback is of $RL(!^X : X \rTo 1 )$ along $\eta_1: 1 \rTo RL 1$. Now because $L \dashv R$ is over $\mathcal{C}$ the composite adjunction $\Sigma_{\mathcal{D}} L \dashv R \mathcal{D}^*$ is an equivalence and so its unit, $X \rTo^{\eta_X} RLX \rTo^{R \eta^{\mathcal{D}}_{LX}} R \mathcal{D}^* \Sigma_{\mathcal{D}} LX$ is an isomorphism which implies that $\eta_X$ is a split monomorphism (it is split by $(R \eta^{\mathcal{D}}_{LX} \eta_X)^{-1} R \eta^{\mathcal{D}}_{LX}$). This is sufficient to show that $(!^X,\eta_X)$ is a regular monomorphism.

2. By 1. and Lemma \ref{firstlemma} we know that
\begin{eqnarray*}
(L_{\mathcal{D}^*X}\Sigma_{\Delta_X})_1: (\mathcal{C}/X)/1 \rTo (\mathcal{D}/\mathcal{D}^*X)/(L1 \rTo{\psi}\mathcal{D}^*X)
\end{eqnarray*}
is an equivalence. But $(\mathcal{D}/\mathcal{D}^*X)/(L1 \rTo{\psi} \mathcal{D}^*X ) \cong \mathcal{D}/L1$.
\end{proof}
 
For our final introductory lemma we expose a close relationship between morphisms and natural transformations that arises once we restrict to stably Frobenius adjunctions. 
\begin{lemma}\label{unique}
Let $\Sigma_{\mathcal{D}} \dashv \mathcal{D}^* : \mathcal{D} \pile{\rTo \\ \lTo}\mathcal{C} $ an adjunction, $X$ an object of $\mathcal{C}$ and $L_i\dashv R_i:\mathcal{C}/X \pile{\rTo \\ \lTo} \mathcal{D}$ with $i=1,2$, two adjunctions over $\mathcal{C}$, both Frobenius at $\mathcal{D}^*X$. Let $\psi_i : L_i 1 \rTo \mathcal{D}^* X$ be the adjoint transpose of the isomorphism $\tau_1^i : \Sigma_{\mathcal{D}} L_i 1 \rTo^{\cong} \Sigma_X 1$ that exists by assumption that $L_i \dashv R_i$ is over $\mathcal{C}$. Then there is a bijection between morphisms $\mathcal{D}/\mathcal{D}^*X((L_1 1)_{\psi_1}, (L_2 1)_{\psi_2})$ and morphisms $ (L_1 \dashv R_1) \rTo (L_2 \dashv R_2)$ over $\mathcal{C}$.
\end{lemma}
\begin{proof}
Any $\alpha: L_1 \rTo L_2$ gives rise to a map $\alpha_1: L_1 1 \rTo L_2 1$ which is over $\mathcal{D}^*X$ if $\alpha$ is over $\mathcal{C}$. In the other direction any object $Y_f$ of $\mathcal{C}/X$ is isomorphic to the pullback $1 \times_{X^*X} X^*Y$ (i.e. the pullback of $Id_X \times f $ along $\Delta_X$). As $L_i \dashv R_i$ is Frobenius at $\mathcal{D}^*X$ and over $\mathcal{C}$ there are isomorphisms:
\begin{eqnarray*}
L_i(Y_f) \cong L_i(1 \times_{X^*X} X^*Y) \cong L_i(1 \times_{R_i \mathcal{D}^* X} R_i \mathcal{D}^*Y) \cong L_i 1\times_{\mathcal{D}^*X} \mathcal{D}^* Y \text{($\star$)}
\end{eqnarray*}
where the last pullback is $Id \times \mathcal{D}^*f$ along $\psi_i$ (recall that $\Delta_X$ is the unit of $\Sigma_X \dashv X^*$ evaluated at $1$). Using these isomorphisms, any map $a: L_1 1 \rTo L_2 1$ over $\mathcal{D}^*X$ gives rise to a natural transformation $\hat{a}:L_1 \rTo L_2$. To see that this natural transformation is over $\mathcal{C}$ observe that $\tau^i_{Y_f}$ factors as
\begin{eqnarray*}
\Sigma_{\mathcal{D}}L_i(Y_f) \rTo^{\Sigma_{\mathcal{D}}(\star)} \Sigma_{\mathcal{D}}(L_i 1 \times_{\mathcal{D}^*X} \mathcal{D}^*Y) \rTo^{(\Sigma_{\mathcal{D}}\pi_1,\epsilon^{\mathcal{D}}_Y \Sigma_{\mathcal{D}}\pi_2)} \Sigma_{\mathcal{D}} L_i 1 \times_X Y \rTo^{\pi_2} Y
\end{eqnarray*}
where $\epsilon^{\mathcal{D}}$ is the counit of $\Sigma_{\mathcal{D}} \dashv \mathcal{D}^*$. 

That $\hat{a}_1=a$ is clear from construct and that $\alpha_{Y_f}=(\hat{\alpha_1})_{Y_f}$ for each $Y_f$ follows by naturality of $\alpha$ at $!: Y_f \rTo 1$ and the fact that $\alpha$ is over $\mathcal{C}$.
\end{proof}

\section{Monadicity}
This section consists of a short paragraph where we recall some basic facts about monadic functors.

Beck's monadicity theorem proves that a functor $R$ is monadic iff it (i) has a left adjoint, (ii) reflects isomorphisms and (iii) has and preserves coequalizers for any $R$-split pair of arrows (that is, for any parallel pair of arrows $f,g: X \pile{ \rTo \\ \rTo} Y$ such that $Rf,Rg$ is part of a split coequalizer diagram with coequalizer, $\bar{q}: RY \rTo \bar{Q} $, there exists a coequalizer $q:X  \rTo Q$ of $f,g$ such that the canonical map $\bar{Q} \rTo RQ$ is an isomorphism). From this characterisation of monadicity it is clear that a functor $R: \mathcal{C} \rTo \mathcal{D}$ is monadic iff $R_X$ is monadic for every $X$; i.e. monadicity is slice stable. If $R: \mathcal{C} \rTo \mathcal{D} $ is split by $S: \mathcal{D} \rTo \mathcal{C} $ (i.e. we are given a natural isomorphism $SR \cong Id_{\mathcal{C}}$), then $R$ is monadic if it has a left adjoint (essentially because split coequalizers are preserved by any functor; applied here to both $S$ and $R$). If we have two functors $R: \mathcal{C} \rTo \mathcal{D}$ and  $R': \mathcal{D} \rTo \mathcal{D}'$, both with left adjoints, then provided $R'$ reflects isomorphisms, $R$ is monadic whenever $R'R$ is monadic. In particular, monadicity descends along monadic functors; i.e. if both $R'R$ and $R'$ are monadic then so is $R$. On the other hand if $R'$ has a left adjoint and is split and $R$ is monadic then $R'R$ is monadic; so for split $R'$ (with a left adjoint) we have that $R'R$ is monadic iff $R$ is monadic.

\section{Effective descent}\label{edpara}
If $f:X\rTo Y$ is a morphism of $\mathcal{C}$, then $f$ is an \emph{effective descent morphism} if the functor $f^*: \mathcal{C}/Y \rTo \mathcal{C}/X$ is monadic. An effective descent morphism is necessarily a regular epimorphism (the kernel pair of $f$ is $f^*$-split) and the property of being an effective descent morphism is pullback stable by our observation that monadicity is slice stable. Further, from our observations about monadicity, we known that split epimorphisms are effective descent morphisms and that for any pair of composable morphisms $f,g$, if $fg$ and $g$ are effective descent morphisms then so is $f$. We also know that if $g$ is a split epimorphism then $fg$ is an effective descent morphism if $f$ is.

The next lemma shows that effective descent morphisms interact well with Frobenius adjunctions over a base category $\mathcal{C}$ and is another example of an application of Lemma \ref{firstlemma}:
\begin{lemma}\label{pullbackdescent}
Let $\Sigma_{\mathcal{D}} \dashv \mathcal{D}^* : \mathcal{D} \pile{\rTo \\ \lTo}\mathcal{C} $ be an adjunction.

1. Given an adjunction $L\dashv R:\mathcal{C}\pile{\rTo \\ \lTo} \mathcal{D}$ over $\mathcal{C}$ that satisfies Frobenius reciprocity, if $!^W : W \rTo 1$ is an effect descent morphism in $\mathcal{D}$ then $!^{RW}:RW \rTo 1$ is an effective descent morphism of $\mathcal{C}$.

2. Let $X$ be an object of $\mathcal{C}$.  If $L\dashv R:\mathcal{C}/X \pile{\rTo \\ \lTo} \mathcal{D}$ is over $\mathcal{C}$ which is Frobenius at $\mathcal{D}^*X$ and $!^W : W \rTo 1$ is an effect descent morphism in $\mathcal{D}$, then $RW$ is an effective descent morphism of $\mathcal{C}$.
\end{lemma}
\begin{proof}
1. (a) If $\phi:U \rTo V$ is a morphism of $\mathcal{C}$ such that $U \times RW \rTo^{\phi \times Id_{RW}} V \times RW$ is an isomorophism then, by applying $L$ and the Frobenius reciprocity assumption we have that $LU \times W \rTo^{L\phi \times Id_{W}} LV \times W$ is an isomorphism. Since $!^W$ is an effective descent morphism, $L\phi$ is therefore an isomorphism. But $\Sigma_{\mathcal{D}}L \cong Id_{\mathcal{C}}$ by assumption that $L \dashv R$ is over $\mathcal{C}$ and so $\phi$ is an isomorphism. (b) Say $U \pile{\rTo^f \\ \rTo_g} V$ is a pair of morphisms in $\mathcal{C}$ such that there is $U_{RW} \pile{\rTo^{f \times Id} \\ \rTo_{g \times Id} \\ \lTo_l} V_{RW} \pile{\rTo^m \\ \lTo_n} K_k$, a split coequalizer diagram in $\mathcal{C}/RW$. By taking adjoint transpose across $L \dashv R$ and applying Frobenius reciprocity we obtain $LU_{W} \pile{\rTo^{Lf  \times Id}\\ \rTo_{L g \times Id} \\ \lTo} LV_{W} \pile{\rTo^m \\ \lTo} (LK)_{\hat{k}}$, a split coequalizer diagram in $\mathcal{D}/W$. Since $!^W$ is an effective descent morphism, there exists $LV \rTo^q Q$, a coequalizer of $Lf,Lg$ such that the canonical map $LK \rTo Q \times W$ is an isomorphism. By applying $\Sigma_{\mathcal{D}}$, a left adjoint, to the coequalizer diagram that defines $Q$, and recalling that $\Sigma_{\mathcal{D}}L \cong Id_{\mathcal{C}}$ by assumption, we have that $\Sigma_{\mathcal{D}}Q$ is a coequalizer of $f,g$. But, further, $L\Sigma_{\mathcal{D}}Q \cong Q$ as $L$ preserves coequalizers and $R\mathcal{D}^* \cong Id_{\mathcal{C}}$ and so the following series of natural isomorphisms in $N$, an object of $\mathcal{C}$, shows that $\Sigma_{\mathcal{D}} Q \times RW \cong K$. 
\begin{eqnarray*}
\mathcal{C}(K,N) & \cong &\mathcal{C}(K,R \mathcal{D}^*N) \\
& \cong & \mathcal{D}(LK,\mathcal{D}^*N) \\
&\cong & \mathcal{D}(Q \times W, \mathcal{D}^*N) \\
&\cong & \mathcal{D}(L\Sigma_{\mathcal{D}}Q \times W, \mathcal{D}^*N) \\
&\cong & \mathcal{D}(L(\Sigma_{\mathcal{D}}Q \times R W), \mathcal{D}^*N) \\
&\cong & \mathcal{C}(\Sigma_{\mathcal{D}}Q \times R W, R \mathcal{D}^*N) \cong  \mathcal{C}(\Sigma_{\mathcal{D}}Q \times R W, N)\\
\end{eqnarray*}
Establishing (a) and (b) proves 1., by application of Beck's monadicity theorem.

2. Firstly by our observation that monadicity is slice stable we know that $!: W_{\mathcal{D}^*X} \rTo 1$ is an effective descent morphism of $\mathcal{D}/\mathcal{D}^*X$. By the proof of our first lemma (Lemma \ref{firstlemma}) we see that $L_{\mathcal{D}^*X} \Sigma_{\Delta_X} \dashv \Delta_X^* R_{\mathcal{D}^*X}$ is over $\mathcal{C}/X$ and so by 1. we known that $ \Delta_X^* R_{\mathcal{D}^*X}(W_{\mathcal{D}^*X})$ is an effective descent morphism of $\mathcal{C}$. But $W_{\mathcal{D}^*X}=(\mathcal{D}^*X)^*W$ and, Lemma \ref{firstlemma}, $R$ factors as $\Delta_X^* R_{\mathcal{D}^*X}(\mathcal{D}^*X)^*$ and so $RW$ is an effective descent morphism as required. 
\end{proof}

\section{Groupoids}
We now recall some facts about internal groupoids, $\mathbb{G}$ in $\mathcal{C}$, and find that they provide a plentify supply of effective descent morphisms and further that the Frobenius reciprocity and monadicity conditions, in combination, can be used to characterise when an adjunction is a connected components adjunction for an internal groupoid. 

\subsection{Definition} A \emph{groupoid internal to a cartesian category}, $\mathcal{C}$, consists of the data 
\begin{eqnarray*}
\mathbb{G}=(G_1 \pile{\rTo^{d_0} \\ \rTo_{d_1} } G_0, m: G_1 \times_{G_0} G_1 \rTo G_1, e:G_0 \rTo G_1, i: G_1 \rTo G_1)
\end{eqnarray*}
subject to the usual identities; the object of objects is $G_0$ and object of morphisms $G_1$; $d_0$ is the domain map etc. The domain of $m$ consists of pairs $(g_1,g_2)$ such that $d_0g_1=d_1g_2$. 

\begin{example}\label{groupoideg}
(i) If $X$ is an object of $\mathcal{C}$ then there is the trivial groupoid $\mathbb{X}=(X \pile{\rTo^{Id_X} \\ \rTo_{Id_X} } X, ... )$ associated with it. 

(ii) Another easy but key example of a groupoid is $(X \times X \pile{\rTo^{\pi_1} \\ \rTo_{\pi_2} } X, \pi_{13}: X \times X \times X \rTo X \times X , \Delta_X:X  \rTo X \times X , \tau: X \times X \rTo X \times X)$ for any object $X$ of $\mathcal{C}$, where $\tau$ is the twist isomorphism.
\end{example}

A \emph{$\mathbb{G}$-object} consists of $(X_f, a: \Sigma_{d_1} d_0^* X_f \rTo X_f)$ where $f: X \rTo G_0$ and $a$ is a morphism over $G_0$ that satisfies the usual unit and associative identities (the domain of $\Sigma_{d_1} d_0^* X_f$ is $G_1 \times_{G_0} X$). The map $a: G_1 \times_{G_0} X \rTo X$ is the $\mathbb{G}$-object's \emph{structure} map. A \emph{$\mathbb{G}$-homomorphism} between two $\mathbb{G}$-objects $(Y_g,b)$ and $(X_f,a)$ consists of a map $h:Y_g\rTo X_f$ such that $a\Sigma_{d_1}d_0^*(h)=hb$. This defines the category $[\mathbb{G},\mathcal{C}]$ of \emph{$\mathbb{G}$-objects}. For example, for any object $X$ of $\mathcal{C}$, $[ \mathbb{X},\mathcal{C}]$ is isomorphic to the slice category $\mathcal{C}/X$. As the inverse of a groupoid determines an isomorphism from $\mathbb{G}$ to $\mathbb{G}^{op}$, we always have an isomorphism $[\mathbb{G},\mathcal{C}] \cong [\mathbb{G}^{op},\mathcal{C}]$.

The data for a groupoid $\mathbb{G}$ can be used to define a monad $\mathbb{T}_{\mathbb{G}}$ on $\mathcal{C}/G_0$; its functor part is $\Sigma_{d_1} d_0^*$ and the category of $\mathbb{G}$-objects is the same thing as the category of algebras for this monad. The counit of the associated adjunction $\mathbb{T}_{\mathbb{G}} \dashv U_{\mathbb{G}}: \mathcal{C}/G_0 \pile{\rTo \\ \lTo} [ \mathbb{G}, \mathcal{C}]$, at $(X_f,a)$, is $a: G_1 \times_{G_0} X \rTo X$.
\begin{lemma}\label{groupoidadjisfrob}
For any internal groupoid $\mathbb{G}$, $\mathbb{T}_{\mathbb{G}} \dashv U_{\mathbb{G}} : \mathcal{C}/G_0 \pile{\rTo \\ \lTo} [\mathbb{G},\mathcal{C}]$ is stably Frobenius. 
\end{lemma}
\begin{proof}
As $U_{\mathbb{G}}$ is monadic, pullbacks in the category of $\mathbb{G}$-objects are created in $\mathcal{C}/G_0$. If $X$ is an object over $G_0$ and we are given a $\mathbb{G}$-homomorphism $k: (Y_g,b) \rTo (Z_h,c)$ and a morphism $f:X \rTo Z$ over $G_0$ then we need to show that 
\begin{eqnarray*}
G_1 \times_{G_0} (  X \times_Z Y ) &  \rTo & (G_1 \times_{G_0} X) \times_Z Y \\
(g,x,y) & \mapsto & (g, x, b(g,y))
\end{eqnarray*}
is an isomorphism. But this is clear as $(g,x,y) \mapsto (g,x,b(g^{-1},y))$ can be seen to be the inverse (it is well defined as $k$ is a $\mathbb{G}$-homomorphism).
\end{proof}
In the context of cartesian categories we are using element notation ($x$, $y$, $g$ etc) only as short hand to define morphisms.

Given the Lemma, for each $\mathbb{G}$-object $(X_f,a)$, there is a natural isomorphism 
\begin{eqnarray*}
t_{(X_f,a)}:\mathbb{T}_{\mathbb{G}}U_{\mathbb{G}}(X_f,a) \rTo^{(\mathbb{T}_{\mathbb{G}}!^{U_{\mathbb{G}}(X_f,a)},  \epsilon_{(X_f,a)}   )} \mathbb{T}_{\mathbb{G}}1 \times  (X_f,a) 
\end{eqnarray*}
where $\epsilon$ is the conunit of  $\mathbb{T}_{\mathbb{G}} \dashv U_{\mathbb{G}}$. Using this explicit description it is clear that the diagrams 
\begin{diagram}
\mathbb{T}_{\mathbb{G}} U_{\mathbb{G}} \mathbb{T}_{\mathbb{G}} U_{\mathbb{G}}  (X_f,a) & \rTo^{\mathbb{T}_{\mathbb{G}}U_{\mathbb{G}}t} & \mathbb{T}_{\mathbb{G}} U_{\mathbb{G}} [\mathbb{T}_{\mathbb{G}}1 \times(X_f,a) ] & \rTo^t & \mathbb{T}_{\mathbb{G}}1 \times \mathbb{T}_{\mathbb{G}}1 \times (X_f,a)  \\
 & \rdTo_{\mathbb{T}_{\mathbb{G}} U_{\mathbb{G}} \epsilon_{(X_f,a)}} &  \dTo_{\mathbb{T}_{\mathbb{G}}U_{\mathbb{G}} \pi_2} & & \dTo_{\pi_1 \times Id }  \\
 & & \mathbb{T}_{\mathbb{G}} U_{\mathbb{G}} (X_f,a)  & \rTo_t & \mathbb{T}_{\mathbb{G}}1 \times (X_f,a) \\
\end{diagram}
and
\begin{diagram}
\mathbb{T}_{\mathbb{G}} U_{\mathbb{G}} \mathbb{T}_{\mathbb{G}} U_{\mathbb{G}}  (X_f,a) & \rTo^{\mathbb{T}_{\mathbb{G}}U_{\mathbb{G}}t} & \mathbb{T}_{\mathbb{G}} U_{\mathbb{G}} [\mathbb{T}_{\mathbb{G}}1 \times(X_f,a) ] & \rTo^t & \mathbb{T}_{\mathbb{G}}1 \times \mathbb{T}_{\mathbb{G}}1 \times (X_f,a)  \\
 & \rdTo_{\epsilon_{\mathbb{T}_{\mathbb{G}} U_{\mathbb{G}} (X_f,a)}} &   & \rdTo_{\epsilon_{\mathbb{T}_{\mathbb{G}}1 \times (X_f,a) } }& \dTo_{ \pi_2 \times Id }  \\
 & & \mathbb{T}_{\mathbb{G}} U_{\mathbb{G}} (X_f,a)  & \rTo^t & \mathbb{T}_{\mathbb{G}}1 \times (X_f,a) \\
\end{diagram}
both commute by naturality of $\epsilon$ where the horizontal $t$s are all examples of the isomorphism just described.

\subsection{Connected components}\label{connected}

There is functor $\mathbb{G}^*$ from $\mathcal{C}$ to  $[ \mathbb{G} ,\mathcal{C}]$; it sends any object $X$ to $X_{G_0}$ with its trivial action (i.e. $d_1 \times Id_X  :  G_1 \times X \rTo G_0 \times X $). 
If $\mathbb{G}^* : \mathcal{C}   \rTo $$  [ \mathbb{G} , \mathcal{C} ] $ has a left adjoint adjoint then the left adjoint is written $\Sigma_{\mathbb{G}}$; it is the usual connected components functor, but we have chosen to extend the pullback adjunction notation to this case. (This is not unreasonable since in the case that $\mathbb{G}$ is the groupoid internal to $\mathcal{C}/Y$ determined by $f: X \rTo Y$, then $\Sigma_{\mathbb{G}} \dashv \mathbb{G}^*$ is the adjunction $\Sigma_f \dashv f^*$ up to isomorphism.) This left adjoint, if it exists, must send any $(X_f,a)$ to the coequalizer $X \rTo^n \Sigma_{\mathbb{G}}(X_f,a)$ of $a$ and $\pi_2$. The unit of $\Sigma_{\mathbb{G}} \dashv \mathbb{G}^*$ at $(X_f,a)$ is $X \rTo^{(f,n)} G_0 \times \Sigma_{\mathbb{G}}(X_f,a)$. Connected component adjunctions do not always exist.

If $\mathbb{G}$ has a connected components adjunction then  $\Sigma_{\mathbb{G}}\mathbb{T}_{\mathbb{G}}(X_f) \cong X$ for any object $X_f$ of $\mathcal{C}/G_0$; i.e. $\Sigma_{\mathbb{G}}\mathbb{T}_{\mathbb{G}} \cong \Sigma_{G_0}$. To see this, note that the action map on $\mathbb{T}_{\mathbb{G}}X_f$ is $(g_1,(g_2,x)) \mapsto (g_1g_2,x)$ and $\pi_2 : G_1 \times_{G_0} X \rTo X$ is the coequalizer, split using the identity on $\mathbb{G}$, of this action and $(g_1,(g_2,x)) \mapsto (g_2,x)$. From this it is clear that $\Sigma_{\mathbb{G}}\epsilon_{\mathbb{T}_{\mathbb{G}}X_f}$ is canonically isomorphic to $\pi_2 : G_1 \times_{G_0} X \rTo X$. We also note in passing that therefore $\Sigma_{\mathbb{G}}((G_1)_{d_1},m)$ is isomorphic to $G_0$ as $((G_1)_{d_1},m)=\mathbb{T}_{\mathbb{G}}1$.

Now for any $\mathbb{G}$-object $(X_f,a)$, $U_{\mathbb{G}}\epsilon_{(X_f,a)}$ is the action map $a: G_1 \times_{G_0} X \rTo X$ and we have just seen that $\Sigma_{\mathbb{G}} \epsilon_{\mathbb{T}_{\mathbb{G}} U_{\mathbb{G}} (X_f,a)}$  is the projection $\pi_2 : G_1 \times_{G_0} X \rTo X$. Combining these observations with the diagrams noted after Lemma \ref{groupoidadjisfrob} we see that if $\mathbb{G}$ has a connected components adjunction then for any $\mathbb{G}$-object $(X_f,a)$ there are canonical horizontal isomorphisms such that both diagrams in 
\begin{diagram}
G_1 \times_{G_0} X & \rTo^{\cong} & \Sigma_{\mathbb{G}} ( \mathbb{T}_{\mathbb{G}}1 \times \mathbb{T}_{\mathbb{G}}1 \times (X_f,a) )\\
\dTo^a \dTo_{\pi_2} & & \dTo^{\Sigma_{\mathbb{G}} ( \pi_1  \times Id)} \dTo_{ \Sigma_{\mathbb{G}} (\pi_2 \times Id)}\\
X &  \rTo^{\cong} &  \Sigma_{\mathbb{G}} ( \mathbb{T}_{\mathbb{G}}1 \times  (X_f,a) )\\
\end{diagram}
commute. The connected components functor in conjunction with the monad associated with a groupoid can be used to recover information about any $\mathbb{G}$-object. These observations will be exploited later. 

The following proposition is new and provides characterisations of when an adjunction is a connected components adjunction. 
\begin{proposition}\label{second2}
Let $\Sigma_{\mathcal{D}} \dashv \mathcal{D}^* :\mathcal{D}\pile{\rTo \\ \lTo} \mathcal{C}$ be an ajunction. The following are equivalent:

1. There exists an internal groupoid $\mathbb{G}$ in $\mathcal{C}$ and an equivalence $\Theta: \mathcal{D} \rTo^{\simeq} $$ [ \mathbb{G},\mathcal{C}]$ such that $\Theta \mathcal{D}^* \cong \mathbb{G}^*$.

2. There exists an object $G_0$ and a stably Frobenius adjunction $\mathbb{T} \dashv U: \mathcal{C}/G_0 \pile{\rTo \\ \lTo } \mathcal{D}$ which is over $\mathcal{C}$ and is such that $U$ is monadic.

3. There exists an object $G_0$ and an adjunction $\mathbb{T} \dashv U: \mathcal{C}/G_0 \pile{\rTo \\ \lTo } \mathcal{D}$, Frobenius at $\mathcal{D}^*G_0$, which is over $\mathcal{C}$ and is such that $U$ is monadic.

4. There exists an object $W$ of $\mathcal{D}$ such that $!^W :W\rTo 1$ is an effective descent morphism and $(\Sigma_{\mathcal{D}})_W : \mathcal{D}/W \rTo \mathcal{C}/\Sigma_{\mathcal{D}}W$ is an equivalence of categories.

\end{proposition}

\begin{proof}
1. implies 2. is essentially Lemma \ref{groupoidadjisfrob}.  

2. implies 3. is clear as 3. is a weakening of 2., requiring less of the adjunction $\mathbb{T} \dashv U$.

For 3. implies 4. apply part 2. of Lemma \ref{localic}; take $W = \mathbb{T}1$. Because $U$ is monadic $!^W : W \rTo 1 $ is an effective descent morphism.

For 4. implies 1., define $G_0= \Sigma_{\mathcal{D}}W$, $G_1 = \Sigma_{\mathcal{D}}(W \times W)$, $d_0 = \Sigma_{\mathcal{D}}(\pi_2)$ and $d_1= \Sigma_{\mathcal{D}}(\pi_1)$. Define $e: G_0 \rTo G_1$ as $\Sigma_{\mathcal{D}}(\Delta_W)$. For the multiplication, as $\pi_{13} : W \times W \times W \rTo W$ is the product, relative to $\mathcal{D}/W$, of $\pi_2: W \times  W \rTo W$ and $\pi_1: W \times  W \rTo W$ and as $(\Sigma_{\mathcal{D}})_W$ preserves products (it is an equivalence by assumption) we know that the canonical map $\Sigma_{\mathcal{D}}(\pi_{12},\pi_{23}):\Sigma_{\mathcal{D}}(W \times W \times W) \rTo \Sigma_{\mathcal{D}}(W \times W) \times_{\Sigma_{\mathcal{D}}(W)}   \Sigma_{\mathcal{D}}(W \times W)$ is an isomorphism. Define $m: G_1 \times_{G_0} G_1 \rTo G_1$ to be $\Sigma_{\mathcal{D}}(\pi_{13})\Sigma_{\mathcal{D}}(\pi_{12},\pi_{23})^{-1}$. Define $i: G_1 \rTo G_1$ as $\Sigma_{\mathcal{D}}(\tau)$.

Given that $(\Sigma_{\mathcal{D}})_W$ preserves products and the various pullbacks involved are products in $\mathcal{D}/W$ it is easy to see that a groupoid $\mathbb{G}$ in $\mathcal{C}$ has been defined because of our earlier example (\ref{groupoideg})(ii) which shows that $(W \times W \pile{\rTo \\ \rTo} W, \pi_{13}, ... )$ is always a groupoid.  

Since $!: W \rTo 1$ is an effective descent morphism we can identify $\mathcal{D}$ with the category of algebras of the monad on $\mathcal{D}/W$ induced by the adjunction $\Sigma_W \dashv W^*$. Now say $h: V \rTo W$ is an object of  $\mathcal{D}/W$ and write $X_f$ for its image under $(\Sigma_{\mathcal{D}})_W$. As $W \times V  \rTo^{h \pi_2} W $ is the product of $\pi_2 : W \times W \rTo W $ and $V_h$ in $\mathcal{D}/W$ and  $(\Sigma_{\mathcal{D}})_W$ preserves products we know that $\Sigma_{\mathcal{D}}(V \times W) \cong \Sigma_{\mathcal{D}}(W \times W) \times_{\Sigma_{\mathcal{D}}W} \Sigma_{\mathcal{D}}V = G_1 \times_{G_0} X$, the pullback of $f$ along $d_0$. Further, $W^*\Sigma_W V_h$ is the projection $\pi_1 : W \times V  \rTo W$ which can be written $W_W \times V_h \rTo^{\pi_1} W_W \rTo^{!^{W_W}} 1_{\mathcal{D}/W}$ and therefore maps to $\Sigma_{d_1} d_0^* X_f$ under $(\Sigma_{\mathcal{D}})_W$ as $!^{W_W}$ is $\pi_1: W \times W \rTo W$. It can then be verified that a map $a: W \times V \rTo V$ over $W$ is an algebra of the monad induced by $\Sigma_W \dashv W^*$ if and only if its image under $\Sigma_{\mathcal{D}}$ is a $\mathbb{G}$-object and similarly for homomorphisms. 
\end{proof}
The characterisations of this proposition will be used extensively in what follows essentially as they give us results about $\mathbb{G}$-objects without having to argue in detail about various pullbacks. As an initial example, because $W$ is taken to be $\mathbb{T}1$ in the proof of 2. implies 4.:
\begin{lemma}\label{W is ed}
For any groupoid $\mathbb{G}$ with a connected components adjunction, $! : \mathbb{T}_{\mathbb{G}}1 \rTo 1$ is an effective descent morphism (i.e. $! : ((G_1)_{d_1},m) \rTo 1 $ is an effective descent morphism).
\end{lemma}

As another initial example of the usefulness of the proposition:
\begin{lemma}\label{objectslice}
Let $\mathbb{G}$ be a groupoid in $\mathcal{C}$ with a connected components functor. If $(X_f,a)$ is a $\mathbb{G}$-object then there exists a groupoid $\mathbb{X}_{f,a}$ in $\mathcal{C}$ and an equivalence $[ \mathbb{X}_{f,a},\mathcal{C}] \simeq [\mathbb{G}, \mathcal{C}]/(X_f,a)$ over $\mathcal{C}$. 
\end{lemma}
\begin{proof}
Consider 2 of the Proposition; both the conditions on the adjunction are stable under slicing.
\end{proof}
The groupoid that is obtained in this process is the familiar one: $(G_1 \times_{G_0} X \pile{\rTo^{\pi_2} \\ \rTo_{a} } X, ...)$. The equivalence established can be summarised as follows: if $(Y_g,b)$ is a $\mathbb{X}_{f,a}$-object then the domain of its structure map is the pullback of $g: Y \rTo X$ along $\pi_2: G_1 \times_{G_0} X  \rTo X$; i.e. $G_1 \times_{G_0} Y$ and so the structure map $b$ of $(Y_g,b)$ can also be used to provide a structure map for $Y_{fg}$, turning it into a $\mathbb{G}$-object in such a way that $g$ is $\mathbb{G}$-homomorphism.

\section{Hilsum-Skandalis maps}
Given two internal groupoids $\mathbb{H}$ and $\mathbb{G}$, a Hilsum-Skandalis map from $\mathbb{H}$ to $\mathbb{G}$ will be defined as a particular $\mathbb{H} \times \mathbb{G}$-object. It is easy to verify that the data for a $\mathbb{H} \times \mathbb{G}$-object can be considered to be a triple $( P \rTo{(p,q)} H_0 \times G_0, c:H_1 \times_{H_0} P  \rTo P, d: G_1 \times_{G_0} P \rTo P)$ such that 

(i) $(P_p,c)$ is an $\mathbb{H}$-object, 

(ii) $(P_q,d)$ is a $\mathbb{G}$-object, 

(iii) $p$ is $d$ invariant; i.e. $p \pi_2 = p d$

(iv) $q$ is $c$ invariant; i.e. $q \pi_2 = q c$

(v) $c$ and $d$ commute; i.e. $d (Id_{G_1} \times c) = c (Id_{H_1} \times d)(\tau \times Id_P)$ where $\tau : G_1 \times H_1 \rTo H_1 \times G_1 $ is the twist isomorphism. 

Similarly a morphism can be seen to be a $\mathbb{H} \times \mathbb{G}$-homomorphism if and only if  it is both a $\mathbb{H}$-homomorphism and a $\mathbb{G}$-homomorphism.
\begin{definition}
Given two internal groupoids $\mathbb{H}$ and $\mathbb{G}$ in $\mathcal{C}$ a \emph{Hilsum-Skandalis map from $\mathbb{H}$ to $\mathbb{G}$} is a $\mathbb{H} \times \mathbb{G}$-object $\frak{P}=( P \rTo{(p,q)} H_0 \times G_0, c:H_1 \times_{H_0} P  \rTo P, d: G_1 \times_{G_0} P \rTo P)$ such that 
 
(i) $(d,\pi_2): G_1 \times_{G_0} P \rTo P \times_{H_0} P$ is an isomorphism, and 

(ii) $p: P \rTo H_0$ is an effective descent morphism.

A \emph{morphism between Hilsum-Skandalis maps} $\frak{P} \rTo \frak{P}'$ is a map $\theta: P \rTo P'$ that is both a $\mathbb{H}$-homomorphism and a $\mathbb{G}$-homomorphism; i.e. it is just a morphism of $[\mathbb{H} \times \mathbb{G}, \mathcal{C}]$.
\end{definition}
Note that the inverse of $(d,\pi_2)$, if it exists, must be of the form $(\psi,\pi_2)$ for some $\psi: P \times_{H_0} P \rTo G_1$. 

If $X$ is an object of $\mathcal{C}$ then a \emph{principal $\mathbb{G}$-bundle over $X$} is a Hilsum-Skandalis morphism from $\mathbb{X}$ to $\mathbb{G}$. This definition recovers the usual meaning, with `effective descent morphism' playing the role of open surjection. So, intuitively, Hilsum-Skandalis maps are generalisations of principal bundles; for principal bundles the `domain' is determined by an object and the `codomain' is a groupoid, but for Hilsum-Skandalis maps both the domain and codomain are groupoids. Consult \cite{HSMr} for background material on Hilsum-Skandalis maps. The definition is usually phrased in terms of topological groupoids, where the requirement (ii) on $p$ is for an open surjection. In the context of cartesian categories therefore we are not strictly just generalising the topological situation since there is no notion of `open'. However, as we shall see, all the main results still go through and it is easy to specialise to the open case if required. There are two other minor differences to the standard definition: \\
(i) because $\mathbb{G}$ is homeomorphic to $\mathbb{G}^{op}$ we are not distinguishing the two actions in terms of their handedness.\\ 
(ii) we are not isolating a map up to an equivalence relation determined by morphisms between Hilsum-Skandalis maps. We will keep track of these 2-cells, though, as is well known and we will now show, they are all isomorphisms:
\begin{proposition}\label{morphHS}
Let $(P_{(p,q)},c,d)$ and $(P'_{(p',q')},c',d')$ be two Hilsum-Skandalis maps (from $\mathbb{H}$ to $\mathbb{G}$). (a) If $\theta:(P_q,d) \rTo (P'_{q'},d') $	is a $\mathbb{G}$-homomorphism then $\theta$ is an isomorphism in $[\mathbb{G}, \mathcal{C}]$. (b) If, further, $\theta$ is a morphism between Hilsum-Skandalis maps, i.e. it is also an $\mathbb{H}$-homomorphism, then it is an isomorphism in  $[\mathbb{H} \times \mathbb{G}, \mathcal{C}]$.
\end{proposition}
\begin{proof}
(a) It is sufficient to prove that $\theta$ is an isomorphism in $\mathcal{C}$ ($[ \mathbb{G}, \mathcal{C}]$ is a category of algebras over $\mathcal{C}/ G_0$). Because $p: P \rTo H_0$ is an effective descent morphism it is sufficient to prove $ \theta \times Id_P : P \times_{H_0} P \rTo P' \times_{H_0} P  $ is an isomorphism because $p^*$ reflects isomorphisms. Its inverse is given by
\begin{eqnarray*}
P' \times_{H_0} P \rTo^{(\pi_1, \theta \pi_2, \pi_2)} P' \times_{H_0} P' \times_{H_0} P \rTo^{\psi' \times Id_P} G_1 \times_{G_0} P \rTo^{(d, \pi_2)} P \times_{H_0} P \text{.}
\end{eqnarray*}
To see this use the fact that $\psi$ must factor as $\psi' (\theta \times \theta)$ as $\theta$ is a $\mathbb{G}$-homomorphism and $\pi_1: P' \times_{H_0} P' \rTo P'$ factors as $d'(\psi',\pi_2)$.

(b) follows from (a).
\end{proof}

The next Proposition provides a class of examples of Hilsum-Skandalis maps. We will see later that with a further restriction on the groupoids, all Hilsum-Skandalis maps arise in this way.
\begin{proposition}\label{HSexample}
Let $L \dashv R : [ \mathbb{H},\mathcal{C}] \pile{ \rTo \\ \lTo }   [ \mathbb{G},\mathcal{C}]$ be a stably Frobenius adjunction over $\mathcal{C}$, where both $\mathbb{G}$ and $\mathbb{H}$ have connected component adjunctions. If $(P_p, c: H_1 \times_{H_0} P \rTo P) $ is defined as $R \mathbb{T}_{\mathbb{G}}1$ and $(Q_q, d: G_1 \times_{G_0} Q \rTo Q) $ as $L \mathbb{T}_{\mathbb{H}}1$, then there exists an isomorphism $\psi: P \rTo^{\cong} Q $ and $(P_{(p,q\psi)},c,\psi^{-1}d(Id_{G_0} \times \psi))$, is a Hilsum-Skandalis map from $\mathbb{H}$ to $\mathbb{G}$.
\end{proposition}
\begin{proof}
By Lemma \ref{W is ed} we know that $! : \mathbb{T}_{\mathbb{G}}1 \rTo 1$ is an effective descent morphism. Then by part 2. of  Lemma \ref{pullbackdescent} applied to $L \mathbb{T}_{\mathbb{H}} \dashv U_{\mathbb{H}} R$ we know that $p: P \rTo H_0$ is an effective descent morphism. 

Since $L \dashv R$ is Frobenius and over $\mathcal{C}$ there are isomorphisms 
\begin{eqnarray*}
\Sigma_{\mathbb{H}} [(Y_g,b) \times R (X_f,a)] \rTo^{\cong} \Sigma_{\mathbb{G}} [ (X_f,a) \times L(Y_g,b)]
\end{eqnarray*}
natural in $\mathbb{G}$-objects $(X_f,a)$ and $\mathbb{H}$-objects $(Y_g,b)$. By naturality:
\begin{diagram}
\Sigma_{\mathbb{H}}[\mathbb{T}_{\mathbb{H}}1 \times \mathbb{T}_{\mathbb{H}}1 \times R \mathbb{T}_{\mathbb{G}}1 \times R \mathbb{T}_{\mathbb{G}}1] & \rTo^{\bar{\chi}} & \Sigma_{\mathbb{G}}[\mathbb{T}_{\mathbb{G}}1 \times \mathbb{T}_{\mathbb{G}}1 \times L ( \mathbb{T}_{\mathbb{H}}1 \times  \mathbb{T}_{\mathbb{H}}1)] \\
\dTo^{\Sigma_{\mathbb{H}}( \pi_i \times Id \times Id) } & & \dTo_{\Sigma_{\mathbb{G}} ( Id \times Id \times L \pi_i )}\\
\Sigma_{\mathbb{H}}[\mathbb{T}_{\mathbb{H}}1 \times  R \mathbb{T}_{\mathbb{G}}1 \times R \mathbb{T}_{\mathbb{G}}1 ] & \rTo^{\bar{\theta}} & \Sigma_{\mathbb{G}}[\mathbb{T}_{\mathbb{G}}1 \times \mathbb{T}_{\mathbb{G}}1 \times L ( \mathbb{T}_{\mathbb{H}}1)] \\
\dTo^{\Sigma_{\mathbb{H}}(Id \times  \pi_j) } & & \dTo_{\Sigma_{\mathbb{G}} ( \pi_j \times Id )}\\
\Sigma_{\mathbb{H}}[\mathbb{T}_{\mathbb{H}}1 \times   R \mathbb{T}_{\mathbb{G}}1 ] & \rTo^{\bar{\psi}} & \Sigma_{\mathbb{G}}[ \mathbb{T}_{\mathbb{G}}1 \times L ( \mathbb{T}_{\mathbb{H}}1)] \\
\end{diagram} 
and
\begin{diagram}
\Sigma_{\mathbb{H}}[\mathbb{T}_{\mathbb{H}}1 \times \mathbb{T}_{\mathbb{H}}1 \times R \mathbb{T}_{\mathbb{G}}1 \times R \mathbb{T}_{\mathbb{G}}1] & \rTo^{\bar{\chi}} & \Sigma_{\mathbb{G}}[\mathbb{T}_{\mathbb{G}}1 \times \mathbb{T}_{\mathbb{G}}1 \times L ( \mathbb{T}_{\mathbb{H}}1 \times  \mathbb{T}_{\mathbb{H}}1)] \\
\dTo^{\Sigma_{\mathbb{H}}( Id \times Id \times \pi_i) } & & \dTo_{\Sigma_{\mathbb{G}} ( \pi_i \times L( Id \times Id ) )}\\
\Sigma_{\mathbb{H}}[\mathbb{T}_{\mathbb{H}}1 \times  \mathbb{T}_{\mathbb{H}}1 \times R \mathbb{T}_{\mathbb{G}}1 ] & \rTo^{\bar{\rho}} & \Sigma_{\mathbb{G}}[\mathbb{T}_{\mathbb{G}}1  \times L ( \mathbb{T}_{\mathbb{H}}1 \times \mathbb{T}_{\mathbb{H}}1)] \\
\dTo^{\Sigma_{\mathbb{H}}(\pi_j \times Id ) } & & \dTo_{\Sigma_{\mathbb{G}} ( Id \times L \pi_j  )}\\
\Sigma_{\mathbb{H}}[\mathbb{T}_{\mathbb{H}}1  \times R \mathbb{T}_{\mathbb{G}}1 ] & \rTo^{\bar{\psi}} & \Sigma_{\mathbb{G}}[ \mathbb{T}_{\mathbb{G}}1 \times L ( \mathbb{T}_{\mathbb{H}}1)] \\
\end{diagram} 
commute for $i,j=1,2$ in both diagrams. 
But, applying the diagram before Proposition \ref{second2}, we then have that both squares in
\begin{diagram}
P \times_{H_0} P & \rTo^{\theta} & G_1 \ \times_{G_0} Q \\
\dTo^{\pi_1} \dTo_{\pi_2} & & \dTo^d \dTo_{\pi_2} \\
P & \rTo^{\psi}  & Q \\
\end{diagram} 
commute by considering $j=1,2$ in the bottom square of the first diagram and noting the natural isomorphisms 
$\Sigma_{H_0}U_{\mathbb{H}} \cong \Sigma_{\mathbb{H}} \mathbb{T}_{\mathbb{H}} U_{\mathbb{H}} \cong \Sigma_{\mathbb{H}} [ \mathbb{T}_{\mathbb{H}}1 \times (\_)]$. Then $G_1 \times_{G_0} P \rTo^{(\star, \pi_2)} P \times_{H_0} P$ has an inverse, where $\star$ is the action $d$ translated into an action on $P$ via $\psi$ (i.e. $\star=\psi^{-1}d (Id_{G_1} \times \psi)$) and $q \psi$ must be $\star$ invariant as the diagram shows that $p\psi^{-1}$ is $d$-invariant.

Write $(R_r,e)$ for $L (\mathbb{T}_{\mathbb{H}} \times \mathbb{T}_{\mathbb{H}}1 )$, and $p_1,p_2: R \rTo Q$ for the $\mathbb{G}$-homomorphisms $L\pi_1$ and $L\pi_2$. Then by the bottom square of the second diagram we have commuting diagrams 
\begin{diagram}
H_1 \times_{H_0} P & \rTo^{\rho} & R \\
\dTo^{c} \dTo_{\pi_2} & & \dTo^{p_1} \dTo_{p_2} \\
P & \rTo^{\psi}  & Q \\
\end{diagram} 
from which it is clear that $q\psi$ is $c$-invariant.

To complete this proof we must prove that the two actions $d$ and $c$ commute. Firstly note that the additional commuting squares above can be exploited to confirm that, $\omega$, defined as  
\begin{eqnarray*}
H_1 \times_{H_0} (G_1 \times_{G_0} Q) \rTo^{Id \times \theta^{-1}} H_1 \times_{H_0} ( P \times_{H_0} P) \rTo^{\chi} G_1 \times_{G_0} R \rTo^{Id \times \rho^{-1}} G_1 \times_{G_0} ( H_1 \times_{H_0} P)
\end{eqnarray*} 
is $\tau \times \psi $. To see this observe: $(Id \times \pi_2) \omega = (Id \times \psi^{-1})(Id \times p_2)\chi (Id \times \theta^{-1}) = (Id \times \psi^{-1})\theta \pi_{23} (Id \times \theta^{-1}) = (Id \times \psi^{-1}) \pi_{23}$ and $\pi_{23} \omega = \rho^{-1} \pi_2 \chi (Id \times \theta^{-1}) =  \rho^{-1} \rho (Id \times \pi_2)( Id \times \theta^{-1}) =( Id \times \psi^{-1}) (Id \times \pi_2)$. Combining it is clear that $\omega = \tau \times \psi $. 

Finally, as $p_1$ is a $\mathbb{G}$-hommorphsm, $p_1 e = d (Id \times p_1)$ and so $c (Id \times \psi^{-1}) (Id \times d) = (\psi^{-1} p_1 \rho) (Id \times \pi_1) (Id \times \theta^{-1}) = \psi^{-1} p_1 e \chi (Id \times \theta^{-1}) = \psi^{-1} d (Id \times p_1) \chi (Id \times \theta^{-1}) = \psi^{-1} d (Id \times \psi) (Id \times c) (Id \times \rho^{-1}) \chi (Id \times \theta^{-1}) =  \psi^{-1} d (Id \times \psi) (Id \times c) \omega$.
\end{proof}

In what follows we will use the notation $\frak{P}_{L,R}$ for the Hilsum-Skandalis map corresponding to an adjunction $L \dashv R$.

\section{Internal functors}

An internal functor $\mathbb{F}:\mathbb{H} \rTo \mathbb{G}$ gives rise to a functor $\mathbb{F}^*:[ \mathbb{G} , \mathcal{C} ] \rTo   $$ [ \mathbb{H} , \mathcal{C} ]$ at the level of $\mathbb{G}$-objects. Explicitly the functor $\mathbb{F}^*$ sends $(X_f,a)$ to $(\pi_1: H_0 \times_{G_0} X \rTo H_0, (d_1^{\mathbb{H}}\pi_1,a(F_1 \times Id_X)): H_1 \times_{G_0} X \rTo H_0 \times_{G_0} X)$, where $H_0 \times_{G_0} X$ is the pullback of $f$ along $F_0: H_0 \rTo G_0$. The functor $\mathbb{F}^*$ does not in general have a left adjoint. To see this consider for an internal groupoid $\mathbb{G}$ the unique internal functor $!^{\mathbb{G}} : \mathbb{G} \rTo 1$; then $(!^{\mathbb{G}})^*:\mathcal{C} \rTo $$ [ \mathbb{G} , \mathcal{C}]$ is $\mathbb{G}^*$, and so it has a left adjoint if and only if $\mathbb{G}$ has a connected components adjunction and not every internal groupoid has a connected components adjunction.

However the following lemma provides a class of functors for which the left adjoint does exist and the resulting adjunction is stably Frobenius:

\begin{lemma}\label{leftadjointfor(Id,F)}
Let $\mathbb{F}:\mathbb{H} \rTo \mathbb{G}$ be an internal functor. Then $(Id_{\mathbb{H}},\mathbb{F})^* : [ \mathbb{H} \times \mathbb{G}, \mathcal{C}] \rTo $$ [ \mathbb{H} , \mathcal{C}]$ has a left adjoint $\Sigma_{(Id_{\mathbb{H}}, \mathbb{F})}$ such that 
$\Sigma_{(Id_{\mathbb{H}}, \mathbb{F})} \dashv (Id_{\mathbb{H}},\mathbb{F})^*$ is stably Frobenius.
\end{lemma}
\begin{proof}
Let $Z=(Z_{(a,b)},*_{\mathbb{H}},*_{\mathbb{G}})$ be an $\mathbb{H} \times \mathbb{G}$-object. Then observe that $(Id_{\mathbb{H}},\mathbb{F})^*Z$ is given by $E_{ae}$ where $E \rTo^e Z$ is the equalizer of $F_0a: Z \rTo H_0 \rTo G_0$ and $b: Z \rTo G_0$; the action is given by the factorization through $e$ of $\star:(h,i) \mapsto hF_1hei$ where we are using concatenation for both $*_{\mathbb{H}}$ and $*_{\mathbb{G}}$. To ease the notation we drop the $e$ in what follows.
Given an $\mathbb{H}$-object $Y=(Y_p,*_Y)$, consider the pullback of $d_1^{\mathbb{G}}$ along $ F_0p$, i.e. $Y \times_{G_0} G_1 = \{ (y,g_0) \vert F_0py = d_1^{\mathbb{G}} g_0 \}$. Define $\Sigma_{(Id_{\mathbb{H}}, \mathbb{F})}Y$ to be $(Y \times_{G_0} G_1)_{p \times d_0^{\mathbb{G}}}$ with action $(h,g)*(y,g_0)=(hy,F_1hg_0g^{-1})$; it can be checked that this is an $\mathbb{H} \times \mathbb{G}$-object and clearly a functor is defined.
We now check that $\Sigma_{(Id_{\mathbb{H}}, \mathbb{F})} \dashv (Id_{\mathbb{H}},\mathbb{F})^*$. Say $\Psi: Y \times_{G_0} G_1 \rTo Z$ is an $\mathbb{H} \times \mathbb{G}$-homomorphism. Then pre-composition with $(Id_Y,e^{\mathbb{G}}F_0p)$ creates a morphism that composes equally with $F_0a$ and $b$ as $\Psi$ is a morphism over $H_0 \times G_0$. So $\Psi(Id_Y,e^{\mathbb{G}}F_0p): Y \rTo Z $ factors through $E$ and the resulting map, written $\tilde{\Psi}$, is an $\mathbb{H}$-homomorphism because
\begin{eqnarray*}
h \star (\Psi ( Id_Y,e^{\mathbb{G}} F_0 p)y) & = & hF_1 h\Psi(y, e^{\mathbb{G}}F_0py) \\
& = & \Psi [(h,F_1h)*(y,e^{\mathbb{G}}F_0 py)]\\
& = & \Psi (hy,(F_1h)(e^{\mathbb{G}}F_0py)(F_1h)^{-1})\\
& = & \Psi (h,e^{\mathbb{G}}F_0p(hy))\\
& = & \Psi(Id_Y, e^{\mathbb{G}}F_0p)(hy)  \\
\end{eqnarray*}
In the other direction, say we are given an $\mathbb{H}$-homomorphism $\psi: Y \rTo E$, then define $\bar{\psi}: Y \times_{G_0} G_1 \rTo Z$ by $(y,g_0) \mapsto g_0^{-1}\psi y$. Note that $b \psi y = F_0 py$ as $\psi$ is over $H_0$; so $g_0^{-1}\psi y$ makes sense in this definition. Because $\psi$ composes equally with $F_0a$ and $b$, $\bar{\psi}$ is over $H_0 \times G_0$. But $\bar{\psi}(g(y,g_0))=\bar{\psi}(y,g_0g^{-1})=gg_0^{-1}\psi y = g\bar{\psi} (y,g_0)$ and $\bar{\psi}(h(y,g_0))=\bar{\psi}(hy,F_1hg_0)=g_0^{-1}(F_1h)^{-1}\psi(hy)=g_0^{-1}(F_1h)^{-1}[h \star \psi(y)]=g_0^{-1}(F_1h)^{-1}F_1h\psi y = hg_0^{-1}\psi y$ (recall that $*_{\mathbb{H}}$ and $*_{\mathbb{G}}$ commute) and so $\bar{\psi}$ is an $\mathbb{H} \times \mathbb{G}$-homomorphism. It is easy to see that $\widetilde{( \_ )}$ and $\bar{( \_ )}$ are natural and are bijections and so $\Sigma_{(Id_{\mathbb{H}}, \mathbb{F})} \dashv (Id_{\mathbb{H}},\mathbb{F})^*$.
For the stably Frobenius claim, note that the adjoint transpose of the identity on $E$ (i.e. on $(Id_{\mathbb{H}},\mathbb{F})^*Z$) is the map $E \times_{G_0} G_1 \rTo Z$ given by $(z,g)\mapsto g^{-1}z $; and so this gives an explicit description of the counit of the adjunction $\Sigma_{(Id_{\mathbb{H}}, \mathbb{F})} \dashv (Id_{\mathbb{H}},\mathbb{F})^*$. So to prove the stably Frobenius claim we must verify that for any  $\mathbb{H} \times \mathbb{G}$-homomorphisms $\Psi: Z \rTo Z'$ and $\Phi: Y \times_{G_0} G_1 \rTo Z'$ that 
\begin{eqnarray*}
(Y \times_{E'} E) \times_{G_0} G_1 & \rTo & (Y \times_{G_0} G_1 ) \times_{Z'} Z \\
 (y,z,g_0) & \mapsto & (y,g_0,g_0^{-1}z)
\end{eqnarray*}
is an isomorphism (where $E' = (Id_{\mathbb{H}},\mathbb{F})^*Z'$). This can easily be established by verifying that $(y,g_0,z) \mapsto (y,g_0z,g_0)$ is well defined (i.e. factors through $(Y \times_{E'} E) \times_{G_0} G_1$). The verification is straightforward given the following two observations: (i) if $(y,g_0,z)$ is an `element' of $(Y \times_{G_0} G_1 ) \times_{Z'} Z$ then $bz$ is $d^{\mathbb{G}}_0g_0$ and so $bg_0z=d_1^{\mathbb{G}}g_0=F_0py=F_0az=F_0ag_0z$, the last equality because $a$ is $\mathbb{G}$-equivariant, and so $g_0z$ factors through $E$, and (ii) $\Psi g_0 z = g_0 \Psi z = g_0 \Phi (y,g_0)= \Phi (y,g_0g_0^{-1})= \Phi (y, e^{\mathbb{G}}F_0py)= \tilde{\Phi}y$ and so $(y,g_0z)$ factors through $(Y \times_{E'} E)$.
\end{proof}

\begin{definition}
An internal functor $\mathbb{F}: \mathbb{H} \rTo \mathbb{G}$ between two internal groupoids is (i) \emph{fully faithful} provided the canonical morphism from $H_1$ to the pullback of $(d^{\mathbb{G}}_0,d^{\mathbb{G}}_1):G_1 \rTo G_0 \times G_0$ along $H_0 \times H_0 \rTo^{F_0 \times F_0} G_0 \times G_0$ is an isomorphism, (ii) \emph{essentially surjective} if $H_0 \times_{G_0} G_1 \rTo^{d_0^{\mathbb{G}}\pi_2} G_0$ is an effective descent morphism (where the domain $H_0 \times_{G_0} G_1$ is the pullback of $d_1^{\mathbb{G}}$ along $F_0:H_0 \rTo G_0$); and, (iii) \emph{an essential equivalence} if it is both fully faithful and essentially surjective. 
\end{definition}

\begin{example} A morphism $f:Y \rTo X$, treated as an internal functor from $\mathbb{Y}$ to $\mathbb{X}$, is an essential equivalence if and only if it is an isomorphism. Any effective descent morphism is the coequalizer of its kernel pair; but the kernel pair must be the identity ($X \pile{\rTo \\ \rTo}X$) if the internal functor is fully faithful.	
\end{example}

Hilsum-Skandalis maps can be characterised in terms of essential equivalences:
\begin{proposition}\label{HSasesseq}
A $\mathbb{H}\times\mathbb{G}$-object $\frak{P}=(P \rTo^{(p,q)} H_0 \times G_0 ,c:H_1 \times_{H_0} P  \rTo P, d: G_1 \times_{G_0} P \rTo P)$ is a Hilsum-Skandalis map if and only if the internal projection functor $\frak{p}: \mathbb{P} \rTo \mathbb{H}$ is an essential equivalence, where $\mathbb{P}$ is the internal groupoid corresponding to $\frak{P}$ (Lemma \ref{objectslice}).
\end{proposition}

\begin{proof}
The internal projection functor is given by $p : P \rTo H_0$ (on objects) and $\pi_1: H_1 \times_{H_0} G_1 \times_{G_0} P \rTo H_1$ (on morphisms). The pullback of $H_1 \rTo^{(d_0^{\mathbb{H}},d_1^{\mathbb{H}})} H_0 \times H_0$ along $p \times p$ is $H_1 \times_{H_0 \times H_0} (P \times P) = \{ ( h,x_1,x_2) \vert px_1=d^{\mathbb{H}}_0h,px_2=d^{\mathbb{H}}_1h\}$ which is isomorphic to $H_1 \times_{H_0} P \times_{H_0} P  = \{ (h,y,z) \vert py=pz=d^{\mathbb{H}}_0h \}$ using $(h,x_1,x_2) \mapsto (h,x_1,h^{-1}x_2)$ and $(h,y,z) \mapsto (h,y,hz)$. Under this isomorphism the canonical map from the morphism object of $\mathbb{P}$ (i.e. from $H_1 \times_{H_O} G_1 \times_{G_0}P$) to $H_1 \times_{H_0} P \times_{H_0} P$ is $Id_{H_1} \times (\pi_2,e):H_1 \times_{H_0}  G_1 \times_{G_0} P \rTo H_1 \times_{H_0} P \times_{H_0} P $ and this last is an isomorphism if and only if $G_1 \times_{G_0} P \rTo^{(\pi_2,e)} P \times_{H_0} P$ is an isomorphism as $H_1 \times_{H_0} (\_)$ reflects isomorphisms (since $d^{\mathbb{H}}_0$ is a split epimorphism and so is an effective descent morphism). So the internal projection functor is fully faithful if and only if condition (i) in the definition of Hilsum-Skandalis morphism is true.

The projection functor will be an essential surjection if and only if $P \times_{H_0} H_1 \rTo^{\pi_2} H_1 \rTo^{d_0^{\mathbb{H}}}H_0$ is an effective descent morphism where the domain is the pullback of $p$ along $d_1^{\mathbb{H}}$. But as $(P_p,c)$ is an $\mathbb{H}$-object, this morphism factors as $p\tilde{c}$ where $\tilde{c}(x,h)=h^{-1}x$. Since $\tilde{c}$ is a split epimorphism we know (see the introductory paragraph of the section on Effective Descent) that $p\tilde{c}$ is an effective descent morphism if and only if $p$ is.
\end{proof}
Given a $\mathbb{H}\times\mathbb{G}$-object $\frak{P}$ there is also an internal projection functor $\frak{q}: \mathbb{P} \rTo \mathbb{G}$. So, for any Hilsum-Skadalis map $\frak{P}$, we have a span $\mathbb{H} \lTo^{\frak{p}} \mathbb{P} \rTo^{\frak{q}} \mathbb{G}$ where $\frak{p}$ is an essential equivalence.
\section{Stably Frobenius Groupoids}

We will also be focusing on certain stably Frobenius adjunctions as a matter of definition:
\begin{definition}
A groupoid $\mathbb{G}$ internal to a cartesian category $\mathcal{C}$ is \emph{stably Frobenius} provided the functor $\mathbb{G}^* : \mathcal{C} \rTo $$ [ \mathbb{G}, \mathcal{C} ]$ has a left adjoint and the resulting adjunction is stably Frobenius.
\end{definition}
If $\mathcal{C}$ has coequalizers then $\Sigma_{\mathbb{G}}$ always exists and all groupoids are stably Frobenius provided that coequalizers are pullback stable. Not all groupoids with connected component adjunctions are stably Frobenius: consider the groupoid determined by the kernel pair of a regular epimorphism that is not pullback stable. Open and proper groupoids in the category of locales provide examples of groupoids that are stably Frobenius in a cartesian category that does not have pullback stable regular epimorphisms in general. 
We suggest that restricting to stably Frobenius groupoids is, in an intuitive sense, the right thing to. It covers all the key cases and does not leave us figuring out which coequalizers are required of the ambient cartesian category $\mathcal{C}$. Having the connected components functor provides all the coequalizers we need, requiring the resulting connected components adjunction to be stably Frobenius allows for a number of nice results to go through, as we show with the next few propositions and the key results on Hilsum-Skandalis maps to follow. The restriction to stably Frobenius groupoids is not a serious barrier to applications; most localic groupoids are stably Frobenius and similarly for the topological case. 

It can be checked that if $\mathbb{G}$ has a connected components adjunction, then it is stably Frobenius if and only for any $\mathbb{G}$-object $(X_f,a)$, the coequalizer that determines $\Sigma_{\mathbb{G}}(X_f,a)$ is pullback stable in the strong sense that for any $f: Y \rTo Z$ the pullback functor $f^*$ preserves the coequalizer diagram that determines $\Sigma_{\mathbb{G}}(X_f,a)$, if that diagram is over $Z$.

We now provide a specialisation of part of Proposition \ref{second2}. This starts our process of demonstrating how stably Frobenius groupoids are convenient by showing that once we restrict to them proving that an adjunction is a connected components adjunction becomes easier.
\begin{proposition}\label{bounded}
Let $\Sigma_{\mathcal{D}} \dashv \mathcal{D}^* :\mathcal{D}\pile{\rTo \\ \lTo} \mathcal{C}$ be a stably Frobenius ajunction. The following are equivalent:

1. There exists an internal groupoid $\mathbb{G}$ in $\mathcal{C}$ and an equivalence $\Theta: \mathcal{D} \rTo^{\simeq} $$ [ \mathbb{G},\mathcal{C}]$ such that $\Theta \mathcal{D}^* \cong \mathbb{G}^*$.

4'. There exists an object $W$ of $\mathcal{D}$ such that $!^W: W \rTo 1$ is an effective descent morphism and for every morphism $f : Y \rTo W$ the morphism $(f,\eta_Y): Y \rTo W \times_{\mathcal{D}^* \Sigma_{\mathcal{D}}W} \mathcal{D}^* \Sigma_{\mathcal{D}}Y$ is a regular monomorphism.
\end{proposition}
This result is first observed in \cite{towgroth}.
\begin{proof}
The functor $(\Sigma_{\mathcal{D}})_W$ factors as $\Sigma_{\eta_W} : \mathcal{D} / W \rTo \mathcal{D} /   \mathcal{D}^* \Sigma_{\mathcal{D}} W$ followed by $(\Sigma{\mathcal{D}})_{\Sigma_{\mathcal{D}}W}$ (i.e. the left adjoint of $\Sigma_{\mathcal{D}} \dashv \mathcal{D}^*$ sliced at $\Sigma_{\mathcal{D}}W$). It has been observed already that $\Sigma_{\eta_W} \dashv \eta_W^*$ satisfies Frobenius reciprocity (as it is the pullback adjunction of a morphism in a cartesian category) and $(\Sigma_{\mathcal{D}})_{\Sigma_{\mathcal{D}}W} \dashv (\mathcal{D}^*)_{\Sigma_{\mathcal{D}}W}$ satisfies Frobenius reciprocity by assumption. Frobenius reciprocity is closed under composition of adjunctions and so $(\Sigma_{\mathcal{D}})_W : \mathcal{D}/ W \rTo \mathcal{C}/ \Sigma_{\mathcal{D}}W$ is the left adjoint of an adjunction that satisfies Frobenius reciprocity. It can be seen that $ (f , \eta_Y)$ is the unit of this adjunction at $Y_f$. Hence $(\Sigma_{\mathcal{D}})_W$ is an equivalence by application of the Lemma \ref{FrobisEq} because $(\Sigma_{\mathcal{D}})_W(1) \cong 1$ and $ (f , \eta_Y)$ is a regular monomorphism for each $Y_f$. Conversely it is clear that if $(\Sigma_{\mathcal{D}})_W$ is an equivalence then the unit is a regular monomorphism.
\end{proof}

In contrast to what is true if we only assume connected components:
\begin{proposition}\label{product}
Let $\mathbb{G}$ and $\mathbb{H}$ be two internal groupoids in  $\mathcal{C}$.

(i) If $\mathbb{H}$ is stably Frobenius then the functor $\pi_2^* :  [ \mathbb{G} , \mathcal{C}] \rTo $$[ \mathbb{H} \times \mathbb{G} , \mathcal{C}] $, corresponding to the projection functor $\pi_2 : \mathbb{H} \times \mathbb{G} \rTo \mathbb{G}$, has a left adjoint, $\Sigma_{\pi_2}$, such that the resulting adjunction is stably Frobenius.

(ii) If $\mathbb{F}:\mathbb{H} \rTo \mathbb{G}$ is an internal functor and $\mathbb{H}$ is stably Frobenius then $\mathbb{F}^* : [ \mathbb{G}, \mathcal{C}] \rTo $$ [ \mathbb{H} , \mathcal{C}] $ has a left adjoint $\Sigma_{\mathbb{F}}$ such that 
$\Sigma_{\mathbb{F}} \dashv \mathbb{F}^*$ is stably Frobenius and over $\mathcal{C}$.

(iii) If $\mathbb{G}$  and $\mathbb{H}$ are both stably Frobenius then $\mathbb{G} \times \mathbb{H}$ is stably Frobenius.

\end{proposition}
\begin{proof}
(i) Given a $\mathbb{H} \times \mathbb{G}$-object $(P_{(p,q)},c,d)$ notice that $\Sigma_{\mathbb{H}}(P_p,c)$ can be turned into a $\mathbb{G}$-object as it is the codomain of a pullback stable regular epimorphism by assumption that $\mathbb{H}$ is stably Frobenius (and it is over $G_0$ as $q$ is $c$-invariant).  Define $\Sigma_{\pi_2}(P_{(p,q)},c,d)$ to be this $\mathbb{G}$-object, extending to homomorphisms in the obvious manner. The remainder of the proof of (i) is then routine from construction.

(ii) Apply (i) and Lemma \ref{leftadjointfor(Id,F)}, noting that $\mathbb{F}^*$ factors as $(Id_{\mathbb{H}},\mathbb{F})^*\pi_2^*$. The claim of `over $\mathcal{C}$' is clear from the definition of $\mathbb{F}^*$.

(iii) Verify that $(\mathbb{H} \times \mathbb{G})^* \cong \pi_2^* \mathbb{G}^*$ and then apply (i) and the assumption that $\mathbb{G}$ is stably Frobenius.
\end{proof}

Internal essential equivalences between stably Frobenius groupoids give rise to equivalences at the level of categories of $\mathbb{G}$-objects:

\begin{proposition}\label{esseqiseq}
Given an internal functor $\mathbb{F}: \mathbb{H} \rTo \mathbb{G}$ between two internal groupoids with $\mathbb{H}$ stably Frobenius, the following are equivalent:\\
(a) $\mathbb{F}$ is an essential equivalence.\\
(b) $\mathbb{F}$ is fully faithful and $d_0^{\mathbb{G}}\pi_2 : H_0 \times_{G_0} G_1 \rTo G_0$ is a regular epimorphism.\\
(c) $\mathbb{F}^* : [ \mathbb{G} , \mathcal{C} ] \rTo   $$ [ \mathbb{H} , \mathcal{C} ] $ is an equivalence.
\end{proposition}
\begin{proof}
(a) implies (b) is trivial as effective descent morphisms are regular epimorphisms. 

For (b) implies (c), by (ii) of the last Proposition there is a stably Frobenius adjunction $\Sigma_{\mathbb{F}} \dashv \mathbb{F}^*$ and we apply Lemma \ref{FrobisEq}. 

The underlying object of $\Sigma_{(Id_{\mathbb{H}},\mathbb{F})}(1)$ is $H_0 \times_{G_0} G_1 \rTo^{Id \times d_0^{\mathbb{G}} } H_0 \times G_0$ (i.e. $\{(x^{\mathbb{H}},g_0) \vert d_1^{\mathbb{G}}g_0 = F_0(x^{\mathbb{H}}) \}$ with $(x^{\mathbb{H}},g_0)\mapsto (x^{\mathbb{H}},d_0^{\mathbb{G}}g_0)$) and the $\mathbb{H}$-action is $(h,(x,g_0)) \mapsto (d_1^{\mathbb{H}}h,F_1hg_0)$ (see the proof of Lemma \ref{leftadjointfor(Id,F)}; the action on the terminal object $1$ of $[ \mathbb{H} ,\mathcal{C} ]$ is $d_1^{\mathbb{H}}$ for any internal groupoid $\mathbb{H}$). The domain of the underlying object of $\Sigma_{\mathbb{F}}1$ is the coequalizer of $d^{\mathbb{H}}_1 \times m^{\mathbb{G}}(F_1 \pi_1, \pi_2), d^{\mathbb{H}}_0 \times Id_{G_1}: H_1 \times_{G_0} G_1 \pile{ \rTo \\ \rTo } H_0 \times_{G_0} G_1$. The unique $\mathbb{G}$-homomorphism $\Sigma_{\mathbb{F}}1 \rTo 1$ is given by the second factor of the factorization of $d^{\mathbb{G}}_0\pi_2:H_0 \times_{G_0} G_1 \rTo G_0$ through this coequalizer. Any regular epimorphism is the coequalizer of its kernal pair, so we can show that if $d^{\mathbb{G}}_0\pi_2$ is a regular epimorphism then $\Sigma_{\mathbb{F}}1\cong 1$ by showing that the kernal pair of $d^{\mathbb{G}}_0\pi_2$ factors through (and is factored through by) the coequalizer that determines $\Sigma_{\mathbb{F}}1$. Certainly $d^{\mathbb{H}}_1 \times m^{\mathbb{G}}(F_1 \pi_1, \pi_2), d^{\mathbb{H}}_0 \times Id_{G_1}$ factors through $\pi_1 , \pi_2 : (H_0 \times_{G_0} G_1) \times_{G_0} (H_0 \times_{G_0} G_1) $ as both maps are over $G_0$. But we also know that $(d^{\mathbb{H}}_0,d^{\mathbb{H}}_1,F_1):H_1 \rTo (H_0 \times H_0) \times_{(G_0 \times G_0)} G_1$ has an inverse, $\Phi$ say, by the assumption that $\mathbb{F}$ is fully faithful. Send $[(x_1^{\mathbb{H}},g_1),(x_2^{\mathbb{H}},g_2)]$, from the kernal pair of $d^{\mathbb{H}}_0\pi_2$, to $(\Phi(x^\mathbb{H}_1,x^\mathbb{H}_2,g_2 g_1^{-1}),g_1)$ to verify that the kernal pair factors through $H_1 \times_{G_0} G_1$ as required. Therefore $\Sigma_{\mathbb{F}}1 \cong 1$.

For any $\mathbb{H}$-object $Y= (p:Y \rTo H_0,b: H_1 \times_{H_0} Y \rTo Y)$ the underlying object over $G_0$ determined by the $\mathbb{G}$-object $\Sigma_{\mathbb{F}}Y$ is the codomain of the coequalizer $n: Y \times_{G_0} G_1 \rTo \Sigma_{\mathbb{H}}(Y \times_{G_0} G_1)$. That the unit $\eta_Y$ is 
\begin{eqnarray*}
Y \rTo^{(Id_Y,e^{\mathbb{G}}F_0p)}Y \times_{G_0} G_1 \rTo^{(p \pi_1, n)} H_0 \times_{G_0} \Sigma_{\mathbb{H}}(Y \times_{G_0} G_1)
\end{eqnarray*}
follows from our construction of $\Sigma_{(Id_{\mathbb{H}},\mathbb{F})} \dashv (Id_{\mathbb{H}},\mathbb{F})^*$ and of $\Sigma_{\pi_2} \dashv \pi_2^*$. 
 
To complete (b) imples (c), we show that each unit $\eta_Y : Y \rTo \mathbb{F}^*\Sigma_{\mathbb{F}}Y$ is a regular monomorphism if $\mathbb{F}$ is fully faithful. In fact it is simpler to prove the stronger result that $\eta_Y$ is an isomorphism. (The proof of this part has been taken from 5.13 of \cite{MoerClassTop}.) 

The coequalizer $q$ is pullback stable by assumption that $\mathbb{H}$ is stably Frobenius. In particular by pulling back along $F_0 : H_0 \rTo G_0$ we see that (the underlying object over $H_0$ of) $\mathbb{F}^*\Sigma_{\mathbb{F}}Y$ is given by the coequalizer of 
\begin{eqnarray*}
H_0 \times_{G_0} ( H_1 \times_{H_0} Y \times_{G_0} G_1) \pile{ \rTo \\ \rTo} H_0   \times_{G_0} ( Y \times_{G_0} G_1) \text{ \em \em (*)}
\end{eqnarray*}
where the top arrow is $(x,h,y,g) \mapsto (x,y,g)$ and the bottom arrow is $(x,h,y,g) \mapsto (x,hy,(F_1h)g)$. Now $H_1 \times_{H_0} Y = \{ (h,y) \vert py \rTo^h d^{\mathbb{H}}_1h \}$ maps to $H_0 \times_{G_0} ( Y \times_{G_0} G_1) = \{ (x,y,g) \vert F_0 x \rTo^g F_0py \}$ via $(h,y) \mapsto ( phy , y, (F_1h)^{-1})$, and this is an isomorphism as $\mathbb{F}$ is fully faithful. Also $H_1 \times_{H_0} H_1 \times_{H_0} Y =\{ (h_1,h_2,y) \vert py \rTo^{h_2} d^{\mathbb{H}}_1h_2 \rTo^{h_1} d^{\mathbb{H}}_1h_1 \}$ maps to $H_0 \times_{G_0} ( H_1 \times_{H_0} Y \times_{G_0} G_1) = \{ (x,h,y,g) \vert F_0x \rTo^g F_0 py,py \rTo^h d_1^{\mathbb{H}}h \}$ via $(h_1,h_2,y) \mapsto (ph_1h_2y,h_2^{-1},h_2y,(F_1h_1)^{-1})$ and this too is an isomorphism as $\mathbb{F}$ is fully faithful. But using these two isomorphisms it can be checked that the pair (*) is isomorphic to the pair
\begin{eqnarray*}
H_1 \times_{H_0} ( H_1 \times_{H_0} Y ) \pile{\rTo^{Id_{H_1} \times b } \\ \rTo_{m^{\mathbb{H}} \times Id_Y}  } H_1  \times_{H_0} Y \text{ \em \em (**)}
\end{eqnarray*}
and so their coequalizers are isomorphic. Since the coequalizer of the second pair (**) is $b: H_1 \times_{H_0} Y \rTo Y$ (it is split by $(e^{\mathbb{H}}p,Id_Y)$) we have that there is an isomorphism from $Y$ to $\mathbb{F}^*\Sigma_{\mathbb{F}}Y$; it is clear from construction that this isomorphism is $\eta_Y$.

For (c) implies (a), by considering the unit $\eta_Y$ at $Y=\mathbb{T}_{\mathbb{H}}1$ we see that if the unit is an isomorphism then $\mathbb{F}$ is fully faithful. Equivalences preserve effective descent morphisms and so $!: \Sigma_{\mathbb{F}} \mathbb{T}_{\mathbb{H}}1 \rTo 1$ is an effective descent morphism by Lemma \ref{W is ed}. Then, by part 2. of Lemma \ref{localic} applied to $\mathbb{T}_{\mathbb{G}} \dashv U_{\mathbb{G}}$, we see that $U_{\mathbb{G}}\Sigma_{\mathbb{F}} \mathbb{T}_{\mathbb{H}}1 \rTo^{U_{\mathbb{G}}!} 1$ is an effective descent morphism. This completes the proof as this last map is $d^{\mathbb{G}}_0\pi_2:H_0 \times_{G_0} G_1 \rTo G_0$ (recall: $\Sigma_{H_0}U_{\mathbb{H}} \cong \Sigma_{\mathbb{H}} \mathbb{T}_{\mathbb{H}} U_{\mathbb{H}} \cong \Sigma_{\mathbb{H}} [ \mathbb{T}_{\mathbb{H}}1 \times (\_)]$ - quotienting out $H_1$ returns $H_0$ and the quotient is pullback stable).
\end{proof}

The weakening of the definition of essential equivalence in this last proposition (part (b)), which is available as we are restricting to stably Frobenius groupoids, works similarly for Hilsum-Skandalis maps. We show this in the following result, but note that it is not used elsewhere in the paper.
\begin{proposition}
Let $\mathbb{H}$ and $\mathbb{G}$ be two groupoids internal to a category $\mathcal{C}$ with $\mathbb{G}$ stably Frobenius. If $(P \rTo^{(p,q)} H_0 \times G_0 ,c:H_1 \times_{H_0} P  \rTo P, d: G_1 \times_{G_0} P \rTo P)$ is a $\mathbb{H}\times\mathbb{G}$-object then it is a Hilsum-Skandalis map if and only if

(i) $(d,\pi_2): G_1 \times_{G_0} P \rTo P \times_{H_0} P$ is an isomorphism, and 

(ii) $p: P \rTo H_0$ is a regular epimorphism.

\end{proposition}
\begin{proof}
As an effective descent morphism is a regular epirmorphism, to complete this proof we must just show that $p$ is an effective descent morphism if conditions (i) and (ii) hold. Firstly, if $p$ is a regular epimorphism then it is the coequalizer of its kernel pair. But this kernel pair is isomorphic to $d,\pi_2 : G_1 \times_{G_0} P \pile{ \rTo \\ \rTo}  P$ and their coequalizer is $G_1 \times_{G_0} P \rTo^n \Sigma_{\mathbb{G}} (P_q,c)$ which, by assumption that $\mathbb{G}$ is stably Frobenius, is pullback stable. Therefore $p$ is a pullback stable regular epimorphism. Note that $H_0 \cong \Sigma_{\mathbb{G}}(P_q,d)$ and under this isomorphism the unit of $\Sigma_{\mathbb{G}} \dashv \mathbb{G}^*$ at $(P_q,d)$ is $P \rTo^{(q,p)} G_0 \times H_0$.

It is well known that $p^*$ reflects isomorphisms for any pullback stable regular epimorphism $p$. (For completeness we recall the proof: if $\theta : A_a \rTo B_b$, a morphism of $\mathcal{C}/H_0$, is such that $p^*\theta$ is an isomorphism, then $\pi_2: P \times_{H_0} A \rTo A$ and $\pi'_2: P \times_{H_0} B \rTo B$ are both regular epimorphisms by the assumption that $p$ is a pullback stable regular epimorphism. In particular $\pi_2$ and $\pi'_2$ are coequalizers of their respective kernel pairs. But their kernel pairs can be seen to be isomorphic as $p^*\theta$ is an isomorphism and so $A \cong B$. Proving additionally that this isomorphism is $\theta$ is clear from construction.)

As $P_{(p,q)}$ is (the underlying object of) a $\mathbb{H}\times\mathbb{G}$-object $(p,q): (P_p,e) \rTo \mathbb{G}^*H_0$ is a $\mathbb{G}$-homomorphism. Now because $\Sigma_{\mathbb{G}}(P_q,e)$ is given by the coequalizer of $e,\pi_2: G_1 \times_{G_0} P \pile{\rTo \\ \rTo} P$ and condition (i) holds by assumption, $\Sigma_{\mathbb{G}}(P_q,e)$ is given by the coequalizer of the kernel pair $p$, which is $p$ itself as $p$ is a regular epimorpishm. From this it is clear that the adjoint transpose of $(p,q)$ is (isomorphic to) the identity on $H_0$.

Now, say $f,g: A_a \pile{ \rTo \\ \rTo} B_b$ are a pair of morphisms of $\mathcal{C}/H_0$ that are $p^*$-split with $r: P \times_{H_0} B \rTo Q$ the coequalizer of $Id_P \times f$ and $Id_P \times g$. Certainly $q\pi_1: P \times_{H_0} B \rTo G_0$ composes equally with $Id_P \times f$ and $Id_P \times g$ and so factors via $r$; say as $tr$, where $t:Q \rTo G_0$. The objects $(P \times_{H_0} A)_{q\pi_1}$ and $(P \times_{H_0} B)_{q\pi_1}$ of $\mathcal{C}$ are the underlying objects of the $\mathbb{G}$-objects $(P_q,d) \times_{\mathbb{G}^*H_0} \mathbb{G}^*A$ and $(P_q,d) \times_{\mathbb{G}^*H_0} \mathbb{G}^*B$ (where the pullback in $[ \mathbb{G}, \mathcal{C}]$ is along $P \rTo^{(p,q)} H_0 \times G_0$). But $r$ is part of a split coequalizer diagram, preserved by any functor, in particular $G_1 \times_{G_0} (\_) $, and so $(Q_t,k)$ is a $\mathbb{G}$-object with an action $k$, since $Id_P \times f$ and $Id_P \times g$ are both $\mathbb{G}$-homomorphisms; $r$ is their coequalizer in $[ \mathbb{G}, \mathcal{C}] $. Apply $\Sigma_{\mathbb{G}}$, a left adjoint, to this coequalizer, use Frobenius reciprocity of $\Sigma_{\mathbb{G}} \dashv \mathbb{G}^*$ at $H_0$ and recall our earlier observation that the adjoint transpose of $(p,q): (P_p,e) \rTo \mathbb{G}^*H_0$ is the identity on $H_0$ as $(p,q)$ is isomorphic to the unit of the adjunction. This shows that $\Sigma_{\mathbb{G}}(Q_t,k)$ is a coequalizer of $f,g$. It can be checked that this coequalizer is over $H_0$ as $b\pi_2: P \times_{H_0} B \rTo H_0$ is $\mathbb{G}$-invariant (and $Id_{G_1} \times r$ is an epimorphism) so the unique factorization of $b\pi_2$ through $r$ also factors through $n: Q \rTo \Sigma{\mathbb{G}}(Q_t,k)$. It also follows that $Q\cong P \times_{H_0} \Sigma_{\mathbb{G}} (Q_t,k)$ as the coequalizers determined by $\Sigma_{\mathbb{G}}$ are pullback stable in the strong sense (recall, by considering kernel pairs, that for any map $s$, if $s$ is the coequalizer in a coqualizer diagram and $s$ is a pullback stable regular epimorphism then the coequalizer diagram itself is pullback stable). This completes a verification of the conditions of Beck's theorem for $p^*$ and so $p$ is of effective descent.
\end{proof}

\section{Hilsum-Skandalis maps as stably Frobenius adjunctions}
We can now prove our first main result which is that between stably Frobenius groupoids, all Hilsum-Skandalis maps arise from stably Frobenius adjunctions.

\begin{theorem}
If $ \frak{P} : \mathbb{H}  \rTo \mathbb{G}$ is a Hilsum-Skandalis map, with $\mathbb{H}$ and $\mathbb{G}$ stably Frobenius, then there exists a stably Frobenius adjunction $ L^{\frak{P}} \dashv R^{\frak{P}} : [ \mathbb{H},\mathcal{C}] \pile{ \rTo \\ \lTo }   [ \mathbb{G},\mathcal{C}]$ over $\mathcal{C}$ such that $\frak{P} \cong \frak{P}_{L^{\frak{P}},R^{\frak{P}}}$. 
\end{theorem}
\begin{proof}
By Lemma \ref{HSasesseq} the projection functor $\frak{p}:\mathbb{P} \rTo \mathbb{H}$ is an essential equivalence, where $\mathbb{P}$ is the internal groupoid corresponding to $\frak{P}$ (use Lemma \ref{objectslice}; $[\mathbb{P}, \mathcal{C}] \rTo^{\simeq}  $$[ \mathbb{H} \times \mathbb{G}, \mathcal{C} ] / \frak{P} $ over $\mathcal{C}$). Now $\Sigma_{\frak{P}} \dashv \frak{P}^*$ is stably Frobenius (it is a pullback adjunction) so we can conclude that $\mathbb{P}$ is stably Frobenius, because $\mathbb{H} \times \mathbb{G}$ is stably Frobenius (by (iii) of Proposition \ref{product}). As $\mathbb{P}$ is stably Frobenius we can apply Proposition \ref{esseqiseq} to conclude that $[ \mathbb{H}, \mathcal{C} ] \rTo^{\frak{p}^*}$$ [\mathbb{P}, \mathcal{C}]$ is an equivalence. We therefore have an equivalence $\phi: [ \mathbb{H}, \mathcal{C} ] \rTo{\simeq} $$ [ \mathbb{H} \times \mathbb{G}, \mathcal{C} ] / \frak{P}$. Given an $\mathbb{H}$-object $(Y_g,b)$, its image is the projection $\pi_1 : P \times_{H_0} Y \rTo P$ (where the action on $P \times_{H_0} Y$ is determined by $(g,h)(x,y)=(ghx,hy)$).

Define $L^{\frak{P}} \dashv R^{\frak{P}}$ as 
\begin{eqnarray*}
[ \mathbb{H}, \mathcal{C} ] \simeq [ \mathbb{H} \times \mathbb{G}, \mathcal{C} ] / \frak{P} \pile{ \rTo^{\Sigma_{\frak{P}}} \\ \lTo_{\frak{P}^*}} [ \mathbb{H} \times \mathbb{G}, \mathcal{C} ] \pile{\rTo^{\Sigma_{\pi_2}} \\ \lTo_{\pi_2^*}} [ \mathbb{G}, \mathcal{C}]   \text{  \ \ \ \     (*)}
\end{eqnarray*}
where the final adjunction is from (i) of Proposition \ref{product}. This adjunction is stably Frobenius and over $\mathcal{C}$.

Say $(P \rTo^{(p,q)} H_0 \times G_0 ,c:H_1 \times_{H_0} P  \rTo P, d: G_1 \times_{G_0} P \rTo P)$ is the underlying data of $\frak{P}$. The image of the $\mathbb{H}$-object $(P_p,c)$ under the equivalence in the display (*) is the morphism $P \times_{H_0} P \rTo^{\pi_1} P$. The image of $\mathbb{T}_{\mathbb{G}}1$ under $\frak{P}^*\pi_2^*$ is the pullback of $d^{\mathbb{G}}_1 \times Id_{H_0} : G_1 \times H_0 \rTo G_0 \times H_0 $ along $P \rTo^{(q,p)} G_0 \times H_0$ which is $\pi_2 : G_1 \times_{G_0} P \rTo P$. But $G_1 \times_{G_0} P \rTo P \times_{H_0}P$ given by $(g_0,x) \mapsto (x,g_0^{-1}x)$ is an $\mathbb{H} \times \mathbb{G}$-isomorphism as $\frak{P}$ is a Hilsum-Skandalis map and so $(P_p,c) \cong R^{\frak{P}} \mathbb{T}_{\mathbb{G}}1$.

Further, the image of $\mathbb{T}_{\mathbb{H}}1$ under the equivalence in the display (*) is the $\mathbb{H} \times \mathbb{G}$-homomorphism $\pi_1: P \times_{H_0} H_1 \rTo P$. But $\Sigma_{\pi_2}$ quotients out the $\mathbb{H}$-action and when applied to (the $\mathbb{H} \times \mathbb{G}$-object whose underlying object is) $P \times_{H_0} H_1$ returns $(P_q,d)$ as the quotienting map $P \times_{H_0} H_1 \rTo P$ can be seen to be $(x,h_0) \mapsto h_0^{-1}x$ which can be split using the identity of $\mathbb{H}$. Hence $(P_q,d) \cong L^{\frak{P}} \mathbb{T}_{\mathbb{H}}1$ and so  $\frak{P} \cong \frak{P}_{L^{\frak{P}},R^{\frak{P}}}$. 
\end{proof}

The reader may find it odd that we have not mentioned profunctors. Whilst parts of the proof could usefully reference constructions from the well established theory of profunctors, notably the construction of $L^{\frak{P}} \dashv R^{\frak{P}}$ from $\frak{P}$, I was not able to easily adapt the theory of profunctors to get the result.
\begin{corollary}\label{HSCor}
If $\mathbb{H}$ and $\mathbb{G}$ are two stably Frobenius groupids, then there exists an equivalence  $ \theta : [ \mathbb{H},\mathcal{C}] \rTo   $$ [ \mathbb{G},\mathcal{C}]$ over $\mathcal{C}$ if and only if there is a span $\mathbb{H} \lTo^{\frak{p}} \mathbb{P} \rTo^{\frak{q}} \mathbb{G}$ such that both $\frak{p}$ and $\frak{q}$ are essential equivalences. 
\end{corollary}
\begin{proof}
If we have such a span then by apply Proposition \ref{esseqiseq} twice to construct the required equivalence. In the other direction if $\theta$ is an equivalence then $\theta \dashv \theta^{-1}$ and $\theta^{-1} \dashv \theta$. Apply the Theorem twice and notice that we get the same span both times. 
\end{proof}
Two internal groupoids, internal to $\mathcal{C}$, are said to be \emph{Morita equivalent} if their categories of $\mathbb{G}$-objects are equivalent over $\mathcal{C}$. The corollary characterizes Morita equivalence for stably Frobenius groupoids.  

We now extend the Theorem to morphisms and so make good our claim that Hilsum-Skandalis maps between stably Frobenius groupoids can be represented by stably Frobenius adjunctions.
\begin{theorem}\label{HS}
If $ \mathbb{H} $ and $\mathbb{G}$ are two internal stably Frobenius groupoids, then the category of Hilsum-Skandalis maps from $ \mathbb{H} $ to $\mathbb{G}$ is equivalent to the category of stably Frobenius adjunctions $L \dashv R : [ \mathbb{H}, \mathcal{C} ] \pile{\rTo \\ \lTo} [ \mathbb{G}, \mathcal{C}]$ over $\mathcal{C}$.
\end{theorem}
\begin{proof}
We first verify that the assignment of Hilsum-Skandalis maps from adjunctions over $\mathcal{C}$ (i.e. $(L \dashv R) \mapsto \frak{P}_{L,R}$) is functorial. Say we have two adjunctions, $L_i \dashv R_i$, $i=1,2$ over $\mathcal{C}$. Then any natural transformation between the left adjoints (say $\alpha$ from $L_1$ to $L_2$) gives rise to a natural transformation between the right adjoints; say $\beta: R_2 \rTo R_1$. Explicitly, $\beta$ is $R_2 \rTo^{\eta^1} R_1 L_1 R_2 \rTo^{R_1 \alpha_{R_2}} R_1 L_2 R_2 \rTo^{R_1 \epsilon^2} R_1$ where $\eta^i$ and $\epsilon^i$ are the unit and counit respectively of $L_i \dashv R_i$ for $i=1,2$; from this definition we always have $\epsilon^1 L_1 \beta = \epsilon^2 \alpha_{R_2}$. Then we note that we have $\alpha_{\mathbb{T}_{\mathbb{H}}1} : L_1 \mathbb{T}_{\mathbb{H}}1 \rTo L_2 \mathbb{T}_{\mathbb{H}}1$ which is a $\mathbb{G}$-homomorphism. In fact it must be an isomorphism by application of part (a) of Proposition \ref{morphHS}. We also have $\beta_{\mathbb{T}_{\mathbb{G}}1} : R_2 \mathbb{T}_{\mathbb{G}}1 \rTo R_1 \mathbb{T}_{\mathbb{G}}1$ which is an $\mathbb{H}$-homomorphism. This looks like it is going in the wrong direction, but if we can show that the canonical isomorphism $\psi_2: \Sigma_{H_0}U_{\mathbb{H}}R_2\mathbb{T}_{\mathbb{G}}1 \rTo \Sigma_{G_0}U_{\mathbb{G}}L_2 \mathbb{T}_{\mathbb{H}}1$ factors as
\begin{eqnarray*}
\Sigma_{H_0}U_{\mathbb{H}}R_2\mathbb{T}_{\mathbb{G}}1 \rTo^{\Sigma_{H_0}U_{\mathbb{H}} \beta_{\mathbb{T}_{\mathbb{G}}1}} \Sigma_{H_0}U_{\mathbb{H}} R_1\mathbb{T}_{\mathbb{G}}1  \rTo^{\psi_1}  \Sigma_{G_0}U_{\mathbb{G}}L_1 \mathbb{T}_{\mathbb{H}}1   \rTo^{ \Sigma_{G_0}U_{\mathbb{G}}\alpha_{\mathbb{T}_{\mathbb{H}}1}}
\Sigma_{G_0}U_{\mathbb{G}}L_2 \mathbb{T}_{\mathbb{H}}1
\end{eqnarray*}
then we will be done as then $\beta^{-1}_{\mathbb{T}_{\mathbb{G}}1}$ can be defined, is a $\mathbb{H}$-homomorphism and corresponds to $\alpha_{\mathbb{T}_{\mathbb{H}}1}$ under the canonical isomorphisms. Now, up to canonical isomorphisms (determined by $\mathbb{G}$ and $\mathbb{H}$ and not by the adjunctions being over $\mathcal{C}$) $\psi_i$, for $i=1,2$, is
\begin{eqnarray*}
\Sigma_{\mathbb{H}}( \mathbb{T}_{\mathbb{H}}1 \times R_i \mathbb{T}_{\mathbb{G}}1) \rTo^{\tau^{-1}_i} \Sigma_{\mathbb{G}}L_i ( \mathbb{T}_{\mathbb{H}}1 \times R_i \mathbb{T}_{\mathbb{G}}1) \rTo^{\Sigma_{\mathbb{G}}( \epsilon_{\mathbb{T}_{\mathbb{G}}1}^iL_i \pi_2, L_i \pi_1)}   \Sigma_{\mathbb{G}} ( \mathbb{T}_{\mathbb{G}}1 \times L_i \mathbb{T}_{\mathbb{H}}1)
\end{eqnarray*}
where $\tau_i$ are the natural isomorphisms that exist as the adjunctions are over $\mathcal{C}$. Under the canonical isomorphism determined by $\mathbb{G}$, $\alpha_{\mathbb{T}_{\mathbb{H}}1}$ corresponds to $\Sigma_{\mathbb{G}}(Id_{\mathbb{T}_{\mathbb{G}}1} \times \alpha_{\mathbb{T}_{\mathbb{H}}1})$ and similarly $\beta_{\mathbb{T}_{\mathbb{G}}1}$ corresponds to $\Sigma_{\mathbb{H}} ( Id_{\mathbb{T}_{\mathbb{H}}1} \times \beta_{\mathbb{T}_{\mathbb{G}}1})$. To prove the factorisation we check that $\Sigma_{\mathbb{G}}( \epsilon_{\mathbb{T}_{\mathbb{G}}1}^2L_2 \pi_2, L_2 \pi_1)\tau^{-1}_2$
\begin{eqnarray*}
&  = & \Sigma_{\mathbb{G}}( \epsilon_{\mathbb{T}_{\mathbb{G}}1}^2L_2 \pi_2, L_2 \pi_1) \Sigma_{\mathbb{G}}
\alpha_{\mathbb{T}_{\mathbb{H}}1 \times R_2 \mathbb{T}_{\mathbb{G}1}}\tau^{-1}_1 \text{ (as $\alpha$ is over $\mathcal{C}$)}\\
& = & \Sigma_{\mathbb{G}}(\epsilon^2_{\mathbb{T}_{\mathbb{G}}1} \times Id_{L_2 \mathbb{T}_{\mathbb{H}}1}) \Sigma_{\mathbb{G}}(\alpha_{R_2 \mathbb{T}_{\mathbb{G}}1} \times \alpha_{\mathbb{T}_{\mathbb{H}}1}) \Sigma_{\mathbb{G}}(L_1 \pi_2,L_1 \pi_1) \tau^{-1}_1 \\
& = & \Sigma_{\mathbb{G}}(Id_{\mathbb{T}_{\mathbb{G}}1} \times \alpha_{\mathbb{T}_{\mathbb{H}}1}) \Sigma_{\mathbb{G}}(\epsilon^2_{\mathbb{T}_{\mathbb{G}}1} \times Id_{L_1 \mathbb{T}_{\mathbb{H}}1}) \Sigma_{\mathbb{G}}(\alpha_{R_2 \mathbb{T}_{\mathbb{G}}1} \times Id_{L_1 \mathbb{T}_{\mathbb{H}}1}) \Sigma_{\mathbb{G}}(L_1 \pi_2,L_1 \pi_1) \tau^{-1}_1\\
& = & \Sigma_{\mathbb{G}}(Id_{\mathbb{T}_{\mathbb{G}}1} \times \alpha_{\mathbb{T}_{\mathbb{H}}1}) \Sigma_{\mathbb{G}}(\epsilon^1_{\mathbb{T}_{\mathbb{G}}1} \times Id_{L_1 \mathbb{T}_{\mathbb{H}}1}) \Sigma_{\mathbb{G}}(L_1 \beta_{\mathbb{T}_{\mathbb{G}}1} \times Id_{L_1 \mathbb{T}_{\mathbb{H}}1}) \Sigma_{\mathbb{G}}(L_1 \pi_2,L_1 \pi_1) \tau^{-1}_1  \\
& = & \Sigma_{\mathbb{G}}(Id_{\mathbb{T}_{\mathbb{G}}1} \times \alpha_{\mathbb{T}_{\mathbb{H}}1}) \Sigma_{\mathbb{G}}(\epsilon^1_{\mathbb{T}_{\mathbb{G}}1} \times Id_{L_1 \mathbb{T}_{\mathbb{H}}1})  \Sigma_{\mathbb{G}}(L_1 \pi_2,L_1 \pi_1) \Sigma_{\mathbb{G}}L_1 ( Id_{\mathbb{T}_{\mathbb{H}}1} \times \beta_{\mathbb{T}_{\mathbb{G}}1} ) \tau^{-1}_1 \\
& = & \Sigma_{\mathbb{G}}(Id_{\mathbb{T}_{\mathbb{G}}1} \times \alpha_{\mathbb{T}_{\mathbb{H}}1}) \Sigma_{\mathbb{G}}( \epsilon_{\mathbb{T}_{\mathbb{G}}1}^1L_1 \pi_2, L_1 \pi_1 )\Sigma_{\mathbb{G}}L_1 ( Id_{\mathbb{T}_{\mathbb{H}}1} \times \beta_{\mathbb{T}_{\mathbb{G}}1} ) \tau^{-1}_1 \\
& = & \Sigma_{\mathbb{G}}(Id_{\mathbb{T}_{\mathbb{G}}1} \times \alpha_{\mathbb{T}_{\mathbb{H}}1}) \Sigma_{\mathbb{G}}( \epsilon_{\mathbb{T}_{\mathbb{G}}1}^1L_1 \pi_2, L_1 \pi_1 )\tau^{-1}_1 \Sigma_{\mathbb{H}} ( Id_{\mathbb{T}_{\mathbb{H}}1} \times \beta_{\mathbb{T}_{\mathbb{G}}1} ) \\
\end{eqnarray*}
where the second line is by naturality of $\alpha$ at $\pi_1: \mathbb{T}_{\mathbb{H}}1 \times R_2 \mathbb{T}_{\mathbb{G}} 1  \rTo \mathbb{T}_{\mathbb{H}}1 $ and $\pi_2: \mathbb{T}_{\mathbb{H}}1 \times R_2 \mathbb{T}_{\mathbb{G}} 1  \rTo R_2 \mathbb{T}_{\mathbb{G}} 1 $, the fourth line is from the definition of $\beta$ and the fifth line is naturality of $\beta$. The last line is naturality of $\tau_1$; we have been suppressing the objects in the notation for the $\tau$s. So we can conclude that $(L \dashv R) \mapsto \frak{P}_{L,R}$ is functorial.

In the other direction to make $\frak{P} \mapsto (L_{\frak{P}} \dashv R_{\frak{P}})$ functorial say we have an $\mathbb{H} \times \mathbb{G}$-homomorphism $a: \frak{P}_1 \rTo \frak{P}_2$; in other words we have a morphism $a: P_1 \rTo P_2$ that is both an $\mathbb{H}$-homomorphism and a $\mathbb{G}$-homomorphism. Writing $L_i \dashv R_i$ for $L_{\frak{P}_i} \dashv R_{\frak{P}_i}$, by the last theorem $P_i \cong L_i \mathbb{T}_{\mathbb{H}}1$ and so by Lemma \ref{unique} applied to $L_i\mathbb{T}_{\mathbb{H}} \dashv U_{\mathbb{H}} R_i$ there is a natural transformation $\hat{a}: L_1 \mathbb{T}_{\mathbb{H}} \rTo L_2 \mathbb{T}_{\mathbb{H}}$ over $\mathcal{C}$ that is uniquely determined by $\hat{a}_1 =a'$ where $a'\cong a$ via $P_i \cong L_i \mathbb{T}_{\mathbb{H}}1$. Explicitly, for every object $Y_g$ of $\mathcal{C}/H_0$, $\hat{a}_{Y_g}$ is defined using $a' \times Id_Y :L_1 \mathbb{T}_{\mathbb{H}}1 \times_{\mathbb{H}^*H_0} \mathbb{H}^*Y \rTo  L_2 \mathbb{T}_{\mathbb{H}}1 \times_{\mathbb{H}^*H_0} \mathbb{H}^*Y$ and $L_i \mathbb{T}_{\mathbb{H}} Y_g \cong L_i \mathbb{T}_{\mathbb{H}}1 \times_{\mathbb{H}^*H_0} \mathbb{H}^*Y$.
But every $\mathbb{H}$-object $(Y_g,b)$ is a canonical coequalizer of a pair of arrows
\begin{eqnarray*}
\mathbb{T}_{\mathbb{H}}U_{\mathbb{H}}\mathbb{T}_{\mathbb{H}}U_{\mathbb{H}}(Y_g,b) \pile{\rTo^{\mathbb{T}_{\mathbb{H}}b} \\ \rTo_{\epsilon_{\mathbb{T}_{\mathbb{H}}U_{\mathbb{H}}(Y_g,b)}}} \mathbb{T}_{\mathbb{H}}U_{\mathbb{H}}(Y_g,b) \text{.}
\end{eqnarray*}
where $\epsilon$ is the counit of $\mathbb{T}_{\mathbb{H}} \dashv U_{\mathbb{H}}$. Therefore as $L_i$ preserves coequalizers, $\hat{a}$ can be seen to define a natural transformation $\bar{a}: L_1 \rTo L_2$ over $\mathcal{C}$, uniquely determined by $\bar{a}_{(Y_g,b)}b=b \hat{a}_{Y_g}$ for all $\mathbb{H}$-objects $(Y_g,b)$, provided we can check that (a) $\hat{a}_{Y_g} (L_1 \mathbb{T}_{\mathbb{H}}b)= (L_2 \mathbb{T}_{\mathbb{H}}b)\hat{a}_{U_{\mathbb{H}}\mathbb{T}_{\mathbb{H}}Y_g}$ and (b) $\hat{a}_{Y_g} (L_1 \epsilon_{\mathbb{T}_{\mathbb{H}}U_{\mathbb{H}}(Y_g,b)})= (L_2 \epsilon_{\mathbb{T}_{\mathbb{H}}U_{\mathbb{H}}(Y_g,b)})\hat{a}_{U_{\mathbb{H}}\mathbb{T}_{\mathbb{H}}Y_g}$. That (a) is true is clear from naturality of $\hat{a}$. As for (b), recall the diagrams after Lemma \ref{groupoidadjisfrob} which show how $\epsilon_{\mathbb{T}_{\mathbb{H}}U_{\mathbb{H}}(Y_g,b)}$ can be identified with $\pi_2 \times Id : \mathbb{T}_{\mathbb{H}}1 \times  \mathbb{T}_{\mathbb{H}}1  \times (Y_g,b) \rTo \mathbb{T}_{\mathbb{H}}1  \times (Y_g,b)$. Hence $L\epsilon_{\mathbb{T}_{\mathbb{H}}U_{\mathbb{H}}(Y_g,b)}$ can be identified with 
\begin{eqnarray*}
\Sigma_{\mathbb{G}}( \mathbb{T}_{\mathbb{G}} 1 \times L [\mathbb{T}_{\mathbb{H}}1 \times  \mathbb{T}_{\mathbb{H}}1  \times (Y_g,b)]) \rTo^{\Sigma_{\mathbb{G}}(Id \times L[\pi_2 \times Id])} \Sigma_{\mathbb{G}}( \mathbb{T}_{\mathbb{G}} 1 \times L [\mathbb{T}_{\mathbb{H}}1  \times (Y_g,b)]) 
\end{eqnarray*}
which (see the proof of Proposition \ref{HSexample}) is isomorphic to:
\begin{eqnarray*}
\Sigma_{\mathbb{H}}([\mathbb{T}_{\mathbb{H}}1 \times  \mathbb{T}_{\mathbb{H}}1  \times (Y_g,b)] \times R \mathbb{T}_{\mathbb{G}}1 ) \rTo^{\Sigma_{\mathbb{H}}( [\pi_2 \times Id] \times Id )  } \Sigma_{\mathbb{H}}( [ \mathbb{T}_{\mathbb{H}}1  \times (Y_g,b)] \times  R \mathbb{T}_{\mathbb{G}} 1) 
\end{eqnarray*}
But this last is the $\mathbb{H}$-action on $(Y_g,b) \times R\mathbb{T}_{\mathbb{G}}1$ and so (b) holds because $a$ is also an $\mathbb{H}$-homomorphism (recalling how $\hat{a}$ is defined in terms of $a \times Id$). Hence $\frak{P} \mapsto (L_{\frak{P}} \dashv R_{\frak{P}})$ is functorial.

Now by applying (b) to $(Y_g,b)= 1$ we see that $a'L_1\epsilon_{\mathbb{T}_{\mathbb{H}}1} = L_2 \epsilon_{\mathbb{T}_{\mathbb{H}}1} \hat{a}_{U_{\mathbb{H}}\mathbb{T}_{\mathbb{H}}1}$ and so $a'=\bar{a}_{\mathbb{T}_{\mathbb{H}}1}$, from which it is clear that $(L \dashv R) \mapsto \frak{P}_{L,R}$ is full. To show faithfulness notice that for $\alpha: L_1 \rTo L_2$ over $\mathcal{C}$, the natural transformation $\alpha_{\mathbb{T}_{\mathbb{H}}}: L_1 \mathbb{T}_{\mathbb{H}} \rTo L_2 \mathbb{T}_{\mathbb{H}}$ is equal to $\widehat{\alpha_{\mathbb{T}_{\mathbb{H}}1}}$ because they both agree at $1$ (Lemma \ref{unique}). By naturality of $\alpha$ at $b$ we then have $\alpha_{(Y_g,b)}L_1 \mathbb{T}_{\mathbb{H}}b = b (\widehat{\alpha_{\mathbb{T}_{\mathbb{H}}1}})_{(Y_g,b)}$ and so $\alpha$ is uniquely determined by $\alpha_{\mathbb{T}_{\mathbb{H}}1}$.
\end{proof}

The next Proposition is an application of the characterisation of morphisms between Hilsum-Skandlis maps as natural transformations (neccesarily isomorphisms) between stably Frobenius adjunctions. Recall that if $\mathbb{G}$ is an internal groupoid, in $\mathcal{C}$, and $I$ is an object of $\mathcal{C}$, then a groupoid $\mathcal{C}(I,\mathbb{G})$ can be defined by setting $\mathcal{C}(I,\mathbb{G})_0=\mathcal{C}(I,G_0)$ and $\mathcal{C}(I,\mathbb{G})_1=\mathcal{C}(I,G_1)$. The groupoid structure of $\mathcal{C}(I,\mathbb{G})$ is inherited from $\mathbb{G}$, for example, the multiplication is given by $(y_1,y_2) \mapsto m(y_1,y_2)$ where $m$ is the multiplication of $\mathbb{G}$.
\begin{proposition}\label{recoverG}
Let $\mathbb{G}$ be a stably Frobenius groupoid in $\mathcal{C}$ and $I$ an object of $\mathcal{C}$. Then $\mathcal{C}(I,\mathbb{G})$ is equivalent to the full subcategory of the category of stably Frobenius adjunctions from $\mathcal{C}/I$ to $[\mathbb{G},\mathcal{C}]$ over $\mathcal{C}$ consisting of adjunctions of the form $\mathbb{T}_{\mathbb{G}}\Sigma_x \dashv x^* U_{\mathbb{G}}$ for each $x: I \rTo G_0$. 
\end{proposition}
In other words $\mathcal{C}(I,\mathbb{G})$ embeds in the groupoid of `stage $I$ points' of $[\mathbb{G},\mathcal{C}]$. Do not expect the embedding to be full as there will also be the points arising from every $\mathbb{G}'$ Morita equivalent to $\mathbb{G}$ in this groupoid (see C5.2.4 of \cite{Elephant} for a reference to a concrete example).
\begin{proof}
Write $p_x$ for the adjunction $\mathbb{T}_{\mathbb{G}}\Sigma_x \dashv x^* U_{\mathbb{G}}$. The Hilsum-Skandalis map associated with $p_x$ is the principal $\mathbb{G}$-bundle $G_1 \times_{G_0} I \rTo I$ (i.e. $\pi_2 : d_0^*x \rTo I$) where the $\mathbb{G}$-action is given by $m \times Id_I$. Define a functor $F$ by $x \mapsto p_x$ on objects and send any morphism $y: I \rTo G_1$ to 
\begin{eqnarray*}
G_1 \times_{G_0} I  \rTo G_1 \times_{G_0} I \\
(g,i) \mapsto  (gy(i)^{-1},i)\\
\end{eqnarray*}
where the domain is a pullback of $I \rTo^y G_1 \rTo^{d_0} G_0$ and the codomain a pullback of $I \rTo^y G_1 \rTo^{d_1} G_0$. This defines a $\mathbb{G}$-homomorphism over $I$ and certainly by considering the groupoid identity we see that $F$, defined on morphisms in this way, is faithful.
Now, given two objects $x_1$ and $x_2$ of $\mathcal{C}(I,\mathbb{G})$, a map of principal bundles $\psi: G_1 \times_{G_0} I \rTo G_1 \times_{G_0} I$ must be of the form $(\theta,\pi_2)$ for some $\theta: G_1 \times_{G_0} I \rTo G_1$ as $\psi$ is over $I$. But in fact as $\psi$ is also a $\mathbb{G}$-homomorphism $\theta$ must be $(g,i) \mapsto gy(i)^{-1}$ where $y: I \rTo G_1$ is given by $i \mapsto [\theta(ex_1i,i)]^{-1}$. It is easy to see that $d_0 \theta(ex_1 i,i)=x_2 i$ and $d_1 \theta(ex_1 i,i)=x_1 i$ and so $y$ is a morphism from $x_1$ to $x_2$. Hence $F$ is full and this completes the proof as $F$ is essentially surjective by definition of the codomain category.
\end{proof}

\section{Hilsum-Skandalis maps invert essential equivalences}\label{invert}
In this section we use the representation of Hilsum-Skandalis maps as stably Frobenius adjunctions to give a short proof that Hilsum-Skandalis maps correspond to the maps of a category that universally inverts internal essential equivalences. This result is a known characteristic of Hilsum-Skandalis maps, though has perhaps not been observed at this level of generality before. The result is also notable as it does not require the machinery of a calculus of fractions. Let $SFGpd_{\mathcal{C}}$ be the 2-category whose objects are stably Frobenius groupoids internal to $\mathcal{C}$, whose morphisms are internal functors and whose 2-cells are internal natural transformations. Let $T_{\mathcal{C}}$ be the 2-category whose objects are categories of the form $[\mathbb{G},\mathcal{C}]$ with stably Frobenius $\mathbb{G}$, whose morphisms are stably Frobenius adjunctions over $\mathcal{C}$ and whose 2-cells are natural transformatons between the left adjoints that commute with connected components functors in the manner as described in Section \ref{adj}. 

Before the statement and proof of the theorem we need a lemma:
\begin{lemma}
If $h: (Y_g,b) \rTo (X_f,a) $ is a $\mathbb{G}$-homomorphism then $(h: Y \rTo X,Id_{G_1} \times h: G_1 \times_{G_0} Y \rTo G_1 \times_{G_0} X)$ determines an internal functor $\frak{h}: \mathbb{X}_{f,a} \rTo \mathbb{Y}_{g,b}$ (i.e Lemma \ref{objectslice} extends to morphisms). The functor $\frak{h}$ is an essential equivalence iff $h: Y \rTo X$ is an effective descent morphism.
\end{lemma}
\begin{proof}
That an internal functor is defined is routine from definitions. For the effective descent claim, note that the relevant map is $Y \times_X (G_1 \times_{G_0} X ) \rTo^{\pi_2} G_1 \times_{G_0} X \rTo^{\pi_2} X$ where the domain is the pullback of $a$ along $h$. But this map factors as $hc$ where $c:Y \times_X (G_1 \times_{G_0} X ) \rTo Y$ is $(y,(g,x)) \mapsto g^{-1}y$ since $h$ is a $\mathbb{G}$-homomorphism. As $c$ is a split epimorphism, using the groupoid identity, we can complete by recalling the remarks of Section \ref{edpara} which outlined how a composition of two morphisms is an effective descent morphism if and only if the second factor is an effective descent morphism, provided the first factor is a split epimorphism.
\end{proof}
\begin{theorem}
There is a pseudo-functor $[ \_ , \mathcal{C}] : SFGpd_{\mathcal{C}} \rTo T_{\mathcal{C}}$ which takes essential equivalences to equivalences and has the property that for any other pseudo-functor $\mathcal{N}: SFGpd_{\mathcal{C}} \rTo N$ which takes essential equivalences to equivalences, there is a pseudo-functor $\bar{\mathcal{N}} : T_{\mathcal{C}} \rTo N$, unique up to equivalence, such that $\bar{\mathcal{N}} [ \_ , \mathcal{C}] \simeq \mathcal{N}$.
\end{theorem}
\begin{proof}
Part (ii) of Lemma \ref{product}, which can be applied as we are restricting to stably Frobenius groupoids, shows how to define $[ \_ , \mathcal{C}]$ on morphisms; send an internal functor $\mathbb{F}: \mathbb{H} \rTo \mathbb{G}$ to $\Sigma_{\mathbb{F}} \dashv \mathbb{F}^*$. If $\alpha: \mathbb{F}^a \rTo \mathbb{F}^b$ is an internal natural transformation (where $\mathbb{F}^a, \mathbb{F}^b:   \mathbb{H} \pile{ \rTo \\ \rTo} \mathbb{G}$) then for any $\mathbb{G}$-object $(X_f,a)$ consider the map $(F_0^a)^* X_f \rTo  (F_0^b)^* X_f$  defined by $(\pi_1,a (\alpha \times Id_X))$; this determines a natural transformation from  $(\mathbb{F}^a)^*$ to $(\mathbb{F}^b)^*$ which can be seen to be a 2-cell of $T_{\mathcal{C}}$. Proposition \ref{esseqiseq} shows that essential equivalences are sent to equivalences.

Let us now say we are given a pseudofunctor $\mathcal{N}: SFGpd_{\mathcal{C}} \rTo N$ which takes essential equivalences to equivalences. Define $\bar{\mathcal{N}}: T_{\mathcal{C}} \rTo N$ on objects by $\bar{\mathcal{N}} [ \mathbb{G}, \mathcal{C}] = \mathcal{N} \mathbb{G}$; this is well defined up to equivalence by Corrolary \ref{HSCor}. Given a morphism $L \dashv R$ of $T_{\mathcal{C}}$, from  $[\mathbb{H},\mathcal{C}]$ to  $[\mathbb{G},\mathcal{C}]$, we know there is a Hilsum-Skandalis map $\frak{P}_{L,R}$ from $\mathbb{H}$ to $\mathbb{G}$. But then there is a span $\mathbb{H} \lTo^{\frak{p}} \mathbb{P}_{L,R} \rTo^{\frak{q}} \mathbb{G}$, with $\frak{p}$ an essential equivalence and so we can define $\bar{\mathcal{N}}(L \dashv R) = \mathcal{N}\frak{q}(\mathcal{N}\frak{p})^{-1}$. If $\alpha : (L_1 \dashv R_1) \rTo  (L_2 \dashv R_2)$ is a 2-cell of $T_{\mathcal{C}}$ then we know (Theorem \ref{HS}) that there is a morphism of Hilsum-Skandalis maps, $\frak{P}_{L_1,R_1} \rTo \frak{P}_{L_2,R_2}$. So, by the lemma, $\alpha$ gives rise to an internal functor $t^{\alpha}:\mathbb{P}_{L_1,R_1} \rTo \mathbb{P}_{L_2,R_2}$ which clearly must commute with the spans associated with the $L_i \dashv R_i$s in the obvious manner (i.e. $\frak{p}_2t^{\alpha}=\frak{p}_1$ and $\frak{q}_2t^{\alpha}=\frak{q}_1$). Note then that because $\mathcal{N}\frak{p}_1$ and $\mathcal{N}\frak{p}_2$ are both equivalences (in $N$) so is $\mathcal{N}t^{\alpha}$. As $\mathcal{N}$ is a pseudofunctor we therefore have a 2-cell $\mathcal{N}\frak{q}_1(\mathcal{N}\frak{p}_1)^{-1} \rTo \mathcal{N}\frak{q}_2\mathcal{N}t^{\alpha} (\mathcal{N}\frak{p}_2 t^{\alpha})^{-1} \rTo  \mathcal{N}\frak{q}_2\mathcal{N}t^{\alpha} (\mathcal{N}t^{\alpha})^{-1}(\mathcal{N}\frak{p}_2)^{-1} \rTo \mathcal{N}\frak{q}_2(\mathcal{N}\frak{p}_2)^{-1} $. This defines $\bar{\mathcal{N}}$ on 2-cells. Because of our representation of Hilsum-Skandalis as stably Frobenius adjunctions, the definition of $\bar{\mathcal{N}}$ on morphisms and 2-cells just given are unique up to equivalence given the requirement of $\mathcal{N}$ to factor as $\bar{\mathcal{N}} [ \_ , \mathcal{C}]$. 

All that remains is to check that $\bar{\mathcal{N}}$ is pseudo-functorial; i.e. if there are two stably Frobenius adjunctions 
\begin{eqnarray*}
[\mathbb{H},\mathcal{C}] \pile{ \rTo^{L_1} \\ \lTo_{R_1}} [\mathbb{G},\mathcal{C}] \pile{ \rTo^{L_2} \\ \lTo_{R_2}} [\mathbb{K},\mathcal{C}] 
\end{eqnarray*} 
over $\mathcal{C}$ then $\mathcal{N}\frak{q}_2 (\mathcal{N}\frak{p}_2)^{-1}\mathcal{N}\frak{q}_1 (\mathcal{N}\frak{p}_1)^{-1}$ is isomorphic (in $N$) to $\mathcal{N}\frak{q}_{12} (\mathcal{N}\frak{p}_{12})^{-1}$ where $\frak{p}_{12},\frak{q}_{12}$ is the span associated with the composite adjunction $L_2L_1 \dashv R_1 R_2$. If $\frak{P}_i = (P_i, ...)$ are the Hilsum-Skandalis maps associated with $L_i \dashv R_i$, for $i=1,2$, then $P_1 \times_{G_0} P_2$ can be seen to be a $\mathbb{H} \times \mathbb{G} \times \mathbb{K}$-object; the action is $[(h,g,k),(x_1,x_2)] \mapsto (hgx_1,kgx_2)$. We write this  $\mathbb{H} \times \mathbb{G} \times \mathbb{K}$-object as $\frak{P}_{123}$ with $\mathbb{P}_{123}$ the associated internal groupoid. As a $\mathbb{G}$-object, $\frak{P}_{123}$ is $L_1 \mathbb{T}_{\mathbb{H}}1 \times R_2 \mathbb{T}_{\mathbb{K}}1$ and because $\Sigma_{\mathbb{G}}(L_1 \mathbb{T}_{\mathbb{H}}1 \times R_2 \mathbb{T}_{\mathbb{K}}1 )\cong \Sigma_{\mathbb{G}}L_1 (\mathbb{T}_{\mathbb{H}}1 \times R_1 R_2 \mathbb{T}_{\mathbb{K}}1) \cong \Sigma_{\mathbb{H}} (\mathbb{T}_{\mathbb{H}}1 \times R_1 R_2 \mathbb{T}_{\mathbb{K}}1)$ the underlying object of $\Sigma_{\mathbb{G}}$ applied to $\frak{P}_{123}$ is $\frak{P}_{13}$, i.e. the $\mathbb{H} \times \mathbb{K}$-object associated with $L_2 L_1 \dashv R_1 R_2$. Further it is easy to see that $\Sigma_{\mathbb{G}}$ determines a functor $\frak{u}: \mathbb{P}_{123} \rTo \mathbb{P}_{13}$ such that $\mathbb{P}_{123} \rTo^{\pi_1} \mathbb{P}_1 \rTo^{\frak{p}_1} \mathbb{H}$ factors as $\frak{p}_{13}\frak{u}$ and $\mathbb{P}_{123} \rTo^{\pi_2} \mathbb{P}_2 \rTo^{\frak{q}_2} \mathbb{K}$ factors as $\frak{q}_{13}\frak{u}$. Let us say that $\pi_1  : \mathbb{P}_{123} \rTo \mathbb{P}_{1} $ is an essential equivalence; then $\mathcal{N}\pi_1$ is an equivalence in $N$ and so is $\mathcal{N}\frak{u}$. But then certainly $\bar{\mathcal{N}}$ is pseudo-functorial; to see this consider  
\begin{diagram}
 & & & & \mathcal{N}\mathbb{P}_{123} & & & & \\ 
  & & & \ruTo^{(\mathcal{N}\pi_1)^{-1}} &        &  \rdTo^{\mathcal{N}\pi_2}& & &     \\ 
 & &\mathcal{N}\mathbb{P}_1 & & & &\mathcal{N}\mathbb{P}_2  & & \\ 
  & \ruTo^{(\mathcal{N}\frak{p}_1)^{-1}}& & \rdTo^{\mathcal{N}\frak{q}_1} &        &  \ruTo^{(\mathcal{N}\frak{p}_2)^{-1}}& & \rdTo^{\mathcal{N}\frak{q}_2}&     \\ 
\mathcal{N}\mathbb{H} & & & & \mathcal{N}\mathbb{G} & & & & \mathcal{N}\mathbb{K} \\
\end{diagram}
and the factorization, up to isomorphsm, of the identity on $ \mathcal{N}\mathbb{P}_{13}$ as $\mathcal{N}\frak{u} (\mathcal{N}\frak{u}^{-1})$ to see that  $\mathcal{N}\frak{q}_2 (\mathcal{N}\frak{p}_2)^{-1}\mathcal{N}\frak{q}_1 (\mathcal{N}\frak{p}_1)^{-1}$ is isomorphic to $\mathcal{N}\frak{q}_{12} (\mathcal{N}\frak{p}_{12})^{-1}$. 

So, to complete, we must check that the projection functor $\pi_1 : \mathbb{P}_{123} \rTo \mathbb{P}_{1} $ is an essential equivalence. On objects $\pi_1$ is the first projection $P_1 \times_{G_0} P_2 \rTo P_1$ which is an effective descent morphism because it is the pullback of $P_2 \rTo G_0$ (which is an effective descent morphism as $\frak{P}_2$ is a Hilsum-Skandlis map). Therefore $\pi_1$ is essentially surjective by the lemma.
To prove that $\pi_1$ is fully faithful we must show that $H_1 \times_{H_0} G_1 \times_{G_0} K_1 \times_{K_0} (P_1 \times_{G_0} P_2) \rTo $$ [ ( P_1 \times_{G_0} P_2) \times (P_1 \times_{G_0} P_2) ] \times_{P_1 \times P_1} (H_1 \times_{H_0} G_1 \times_{G_0} P_1 )$ given by  
\begin{eqnarray*}
(h,g,k,(x_1,x_2)) \mapsto  ([(hgx_1,kgx_2),(x_1,x_2)],(h,g,x_1))
\end{eqnarray*}
is an isomorphsm. But as $\frak{P}_2$ is a Hilsum-Skandalis map we have that $K_1 \times_{K_0} P_2 \rTo^{(*,\pi_2)} P_2 \times_{K_0} P_2$ is an isomorphism and so the inverse is given by
\begin{eqnarray*}
([(hgx,y),(x,z)],(h,g,x)) \mapsto (h,g,\pi_1(*,\pi_2)^{-1}(y,gz),(x,z)) \text{ .}
\end{eqnarray*}
\end{proof}

\section{Composing monadic functors}
In order to give a categorical construction of semi-direct product using the techniques developed, which is the topic of the next section, we need a result about when monadic functors are closed under composition. It is not true in general that the composition of two monadic functors is monadic; however, 
\begin{lemma}\label{hannah}
{\em{[Hannah's lemma.]}} Given a diagram of adjunctions
\begin{diagram}
\mathcal{C} & \pile{\rTo^{L_1} \\ \lTo_{R_1}} & \mathcal{C}_1\\
\dTo^{L}  \uTo_{R} &        & \dTo^{L_2} \uTo_{R_2} \\
\bar{\mathcal{C}} & \pile{\rTo^{\bar{L}} \\ \lTo_{\bar{R}}} & \mathcal{C}_2 \\
\end{diagram}
that commutes up to natural isomorphism (i.e. we are given a natural isomorphism $R_1 R_2 \rTo^{\cong} R \bar{R}$) and that satisfies Beck-Chevalley (i.e. the canonical natural transformation $L R_1 \rTo \bar{R} L_2$ is an isomorphism), if $R_1$ and $R_2$ are both monadic then so is $R_1 R_2$.
\end{lemma}
\begin{proof}
First note that we can assume that $\mathcal{C}_2$ is $\mathcal{C}_1^{\mathbb{T}_2}$ where $\mathbb{T}_2$ is the monad induced by $L_2 \dashv R_2$; $R_2$ is the forgetful functor. We apply Beck's theorem and as the property of reflecting isomorphisms is closed under composition what remains is a check that if $a,b: A \pile{\rTo \\ \rTo}B $ is a pair of arrows in $\mathcal{C}_2$ that is $R_1 R_2$-split then the pair has a coequalizer that is mapped under $R_1 R_2$ to the split, up to canonical isomorphism. As $R_1$ is monadic there exists a coequalizer 
\begin{eqnarray*}
R_2 A \pile{ \rTo^{R_2 a } \\ \rTo_{R_2 b} } R_2 B \rTo^q Q  \text{\ \ \ \ (*)}
\end{eqnarray*}
in $\mathcal{C}_1$ such that $R_1 Q \cong Q_0$ where $Q_0$ is the split coequalizer of $R_1R_2a, R_1 R_2 b$. Now consider $R_1 R_2 L_2 $ applied to (*). By the assumption that the diagram commutes up to isomorphism and that Beck-Chevalley holds we get,
\begin{eqnarray*}
R L R_1  R_2 A \pile{ \rTo^{R  L R_1  R_2 a } \\ \rTo_{ R  L R_1  R_2 b} } R  L R_1 R_2 B \rTo^q R  L R_1  Q  \text{\ \ \ \ (*)}
\end{eqnarray*}
which is part of a split coequalizer diagram as $a,b$ are $R_1R_2$-split (and split coequalizer diagrams are preserved by functors; here $RL$). It follows that as $R_1$ creates coequalizers for such split diagrams and reflects isomorphisms, that $R_2 L_2 q$ is the coequalizer of $R_2 L_2 R_2 a, R_2 L_2 R_2 b$ and in particular, by naturality of the counit $\epsilon^2$ of $L_2 \dashv R_2$, that $R_2 L_2 R_2 B \rTo^{R_2  \epsilon^2_B} R_2 B \rTo^q Q$ factors through $R_2 L_2 q$, say via $c: R_2L_2 Q \rTo Q$. From the  definition of $c$ it is readily checked that $(Q,c)$ satisfies the unit condition required to be a $\mathbb{T}_2$-algebra. A similar argument to the one just deployed to show that $ R_2 L_2 q$ is a coequalizer also shows that $R_2 L_2 R_2 L_2 q$ is the coequalizer of $R_2 L_2 R_2 L_2 R_2 a, R_2 L_2 R_2 L_2 R_2 b$ and so $R_2 L_2 R_2 L_2 q$ is an epimorphism. From this observation it is easy to check that $(Q,c)$  satisfies the associative condition required to be a $\mathbb{T}_2$-algebra and so it is a  $\mathbb{T}_2$-algebra and from this it is clear that $q$ is the coequalizer of $a,b$.
\end{proof}
As a consequence:
\begin{lemma}\label{edcompose}
If $X$ and $Y$ are two objects of a category $\mathcal{C}$ and both $!^X: X \rTo 1$ and $\pi_2: Y \times X \rTo X$  are effective descent morphisms then $!^{Y \times X} : Y \times X \rTo 1$ is an effective descent morphism.
\end{lemma}
\begin{proof}
Consider the diagram of adjunctions:
\begin{diagram}
\mathcal{C}/ Y \times X  & \pile{\rTo^{\Sigma_{\pi_2}} \\ \lTo_{\pi_2^*}} & \mathcal{C}/X\\
\dTo^{\Sigma_{\pi_1}}  \uTo_{\pi_1^*} &        & \dTo^{\Sigma_X} \uTo_{X^*} \\
\mathcal{C}/Y & \pile{\rTo^{\Sigma_Y} \\ \lTo_{Y^*}} & \mathcal{C} \\
\end{diagram}
which clearly commutes and satisfies Beck-Chevalley.
\end{proof}
\section{Semi-direct products}
In this section we exploit the result just given on the composition of effective descent morphisms to give a categorically flavoured description of semi-direct product. The result provides a non-trivial application of our proposition that characterises connected component adjunctions (Propositon \ref{second2}) and is a generalisation of Lemma \ref{objectslice} from trivial internal groupoids (i.e. from slices of $[ \mathbb{G}, \mathcal{C}]$) to all stably Frobenius internal groupoids: 

\begin{proposition}\label{connect}
We are given $\mathbb{G}$, a groupoid in $\mathcal{C}$ with a connected components adjunction, and $\frak{K}$, a stably Frobenius groupoid in $[ \mathbb{G}, \mathcal{C}]$. Then there exists a groupoid $\mathbb{G} \ltimes \frak{K}$ internal to $\mathcal{C}$ and an equivalence $\Theta: [ \mathbb{G} \ltimes \frak{K},\mathcal{C}] $$\rTo{\simeq}$$ [ \frak{K}, [\mathbb{G}, \mathcal{C}]] $ such that $\Sigma_{\mathbb{G}} \Sigma_{\frak{K}} \Theta \cong \Sigma_{\mathbb{G} \ltimes \frak{K}}$. Further, if  $\mathbb{G}$ is stably Frobenius then so is  $\mathbb{G} \ltimes \frak{K}$. 
\end{proposition}
\begin{proof}
The proof can be completed by showing that condition 4. of Proposition \ref{second2} is closed under composition. 
We are given stably Frobenius adjunctions $L_0 \dashv R_0 : \mathcal{D}_0 \pile{\rTo \\ \lTo} \mathcal{D}$ and $L \dashv R : \mathcal{D} \pile{\rTo \\ \lTo} \mathcal{C}$ where there are objects $W_0$ and $W$ of $\mathcal{D}_0$ and $\mathcal{D}$ respectively (with both $W_0 \rTo 1$ and $W \rTo 1$ of effective desecent) such that $L_W: \mathcal{D}/W \rTo \mathcal{C}/LW$ and $(L_0)_{W_0}: \mathcal{D}_0/W_0 \simeq \mathcal{D}/L_0 W_0 $ are both equivalences. Since $\frak{K}$ is stably Frobenius we can assume that $L_0 \dashv R_0$ is stably Frobenius and so $\Sigma_{L_0 W_0} \dashv (L_0 W_0)^*$ factors up to isomorphism as
\begin{eqnarray*}
\mathcal{D}/L_0 W_0 \pile{ \rTo^{[(L_0)_{W_0}]^{-1}} \\ \lTo_{(L_0)_{W_0}}} \mathcal{D}_0/W_0 \pile{\rTo^{\Sigma_{W_0}} \\ \lTo_{W_0^*}}  \mathcal{D}_0 \pile{\rTo^{L_0} \\ \lTo_{R_0}} \mathcal{D} \text{.}
\end{eqnarray*}
Then, by taking the slice at $W$, we can conclude that $\mathcal{D}_0/ W_0 \times R_0W $ is equivalent to $\mathcal{D}/L_0 W_0 \times W$. Now $\pi_1^* : \mathcal{D}/L_0 W_0 \rTo \mathcal{D}/L_0 W_0 \times W $ is monadic as it is the slice at $L_0 W_0$ of the monadic functor $W^*: \mathcal{D} \rTo \mathcal{D}/W$; hence $\pi_1: W_0 \times R_0 W \rTo W_0$ is an effective descent morphism. By Lemma \ref{edcompose} the morphism $!: W_0 \times R_0 W \rTo 1$ is an effective descent morphism. 
But this essentially completes the proof because $\mathcal{D}/L_0 W_0 \times W$ is a slice of $\mathcal{D}/W$ and $\mathcal{D}/W$ is (equivalent to) the slice $\mathcal{C}/LW$. 

The `Further' part is clear as stably Frobenius adjunctions are stable under composition. 
\end{proof}

Let us be explicit about the structure of $\mathbb{G} \ltimes \frak{K}$. Say $\frak{K} = ( (K^1_{f_1},a_1) \pile{ \rTo^{d^{\frak{K}}_0}  \\ \rTo_{d^{\frak{K}}_1} }  (K^0_{f_0},a_0), ... ) $. In the proof of the last Proposition we have $\mathcal{D}= [ \mathbb{G}, \mathcal{C} ]$, $ W= \mathbb{T}_{\mathbb{G}}1 $ and $ LW=\Sigma_{\mathbb{G}} W$ is $G_0$. Similarly $W_0$ is $ \mathbb{T}_{\frak{K}}1 $, $L_0=\Sigma_{\frak{K}}$, $L_0 W_0$ is $(K^0_{f_0},a_0)$ and $L_0(W_0 \times W_0)$ is $(K^1_{f_1},a_1)$. The object of objects of $\mathbb{G} \ltimes \frak{K}$, constructed in the proof of the Proposition, is $\Sigma_{\mathbb{G}}$ applied to the image of $1$ under the equivalences $\mathcal{D}_0/W_0 \times R_0 W \rTo \mathcal{D}/L_0 W_0 \times W$ which is $\Sigma_{\mathbb{G}}(L_0 W_0 \times \mathbb{T}_{\mathbb{G}}1)$ which from earlier identities we know is $\Sigma_{G_0}U_{\mathbb{G}}(L_0 W_0)=K^0$. Now Proposition \ref{second2} shows us that $d^{\frak{K}}_0=L_0\pi_2$ and $d^{\frak{K}}_1=L_0\pi_1$ where $L_0\pi_i:L_0(W_0 \times W_0) \rTo L_0 W_0$. The domain and codomain maps of $\mathbb{G} \ltimes \frak{K}$ are $\Sigma_{\mathbb{G}}L_0 \pi_i : \Sigma_{\mathbb{G}}L_0 ( [W_0 \times R_0 W] \times [W_0 \times R_0 W] ) \rTo   \Sigma_{\mathbb{G}}L_0 ( [W_0 \times R_0 W] )$ for $i=2,1$ respectively, which by application of the  Frobenius reciprocity assumption on $L_0 \dashv R_0$ are  $\Sigma_{\mathbb{G}}( \pi_i \times L_0 \pi_i ) : \Sigma_{\mathbb{G}}( [\mathbb{T}_{\mathbb{G}}1  \times \mathbb{T}_{\mathbb{G}}1 ] \times L_0 [W_0 \times W_0]  ) \rTo   \Sigma_{\mathbb{G}}(\mathbb{T}_{\mathbb{G}}1 \times L_0  W_0 )$. But, see the diagram before Proposition \ref{second2}, the image of the two projections $\pi_2$ and $\pi_1$ onto $\mathbb{T}_{\mathbb{G}}1$ under the connected components functor of $\mathbb{G}$ can be identified with the projection and $\mathbb{G}$-action respectively and so the domain and codomain maps of $\mathbb{G} \ltimes \frak{K}$ can be seen to be $G_1 \times_{G_0} K^1 \rTo^{\pi_2} K^1 \rTo^{d^{\frak{K}}_0} K^0$ and $G_1 \times_{G_0} K^1 \rTo^{a_1} K^1 \rTo^{d^{\frak{K}}_1} K^0$.

As for the multiplication we first pause for a lemma:
\begin{lemma}
If $\mathbb{G}$ is a groupoid with a connected components adjunction and $(X_f,a)$ a $\mathbb{G}$-object, then under the canonical isomorphisms the morphisms $\Sigma_{\mathbb{G}}(\pi_{ij} \times Id) : \Sigma_{\mathbb{G}}( \mathbb{T}_{\mathbb{G}}1 \times  \mathbb{T}_{\mathbb{G}}1 \times  \mathbb{T}_{\mathbb{G}}1 \times (X_f,a)) \rTo \Sigma_{\mathbb{G}}( \mathbb{T}_{\mathbb{G}}1 \times  \mathbb{T}_{\mathbb{G}}1 \times (X_f,a))$ are 
\begin{eqnarray*}
G_1 \times_{G_0} G_1 \times_{G_0} X & \rTo &G_1 \times_{G_0} X \\
(g_1,g_2,x) & \mapsto  & (g_1,g_2x)\\
(g_1,g_2,x) & \mapsto  & (g_1g_2,x) \\
(g_1,g_2,x) & \mapsto  & (g_2,x)\\
\end{eqnarray*} 
for $ij=12,13,23$ respectively.
\end{lemma}
\begin{proof}
Apply $\mathbb{T}_{\mathbb{G}}1 \times ( \_)$ to the diagrams just before the start of subsection \ref{connected}. For any $\mathbb{G}$-homomorphism $h$, $Id_{\mathbb{T}_{\mathbb{G}}1} \times h$ is canonically isomorphic to $\mathbb{T}_{\mathbb{G}}U_{\mathbb{G}}h$ by naturality of $\epsilon$. Applying to $h=\mathbb{T}_{\mathbb{G}} U_{\mathbb{G}}\epsilon_{(X_f,a)}$ and $h=\epsilon_{\mathbb{T}_{\mathbb{G}} U_{\mathbb{G}}(X_f,a)}$ we see that for $ij= 12, 13$ respectively, $\Sigma_{\mathbb{G}}(\pi_{ij} \times Id)$ is canonically isomorphic to $U_{\mathbb{G}} \mathbb{T}_{\mathbb{G}} U_{\mathbb{G}}\epsilon_{(X_f,a)}$ and $U_{\mathbb{G}}\epsilon_{\mathbb{T}_{\mathbb{G}}U_{\mathbb{G}}(X_f,a)}$ respectively from which these cases are clear given our knowledge of $\epsilon$. 

For the $ij=23$ case apply the second diagram to the $\mathbb{G}$-object $\mathbb{T}_{\mathbb{G}}1 \times (X_f,a)$ and note that $\epsilon_{\mathbb{T}_{\mathbb{G}}U_{\mathbb{G}}(\mathbb{T}_{\mathbb{G}}1 \times (X_f,a))}$ is canonically isomorphic to $\epsilon_{\mathbb{T}_{\mathbb{G}}U_{\mathbb{G}}\mathbb{T}_{\mathbb{G}}U_{\mathbb{G}}(X_f,a)}$. We have commented in the paragraphs leading up to Proposition \ref{second2} that $\Sigma_{\mathbb{G}}$ applied to the counit at the free algebra returns the relevant projection.
\end{proof}
Returning to multiplication, the 4. implies 1. part of Proposition \ref{second2}  shows that it is
\begin{eqnarray*}
L_0(W_0 \times W_0) \times_{L_0 W_0} L_0(W_0 \times W_0) \rTo^{(L_0\pi_{12},L_0 \pi_{23})^{-1}} L_0(W_0 \times W_0 \times W_0) \rTo^{L_0 \pi_{13}} L_0 (W_0 \times W_0)
\end{eqnarray*}
for $\frak{K}$ and similarly for $\mathbb{G}$ and $\mathbb{G} \ltimes \frak{K}$. So, again exploiting the fact that $L_0 \dashv R_0$ satisfies Frobenius reciprocity, the multiplication of $\mathbb{G} \ltimes \frak{K}$ is the composition of $\Sigma_{\mathbb{G}}(\pi_{13} \times L_0\pi_{13})$
\begin{eqnarray*}
\Sigma_{\mathbb{G}}( \mathbb{T}_{\mathbb{G}}1 \times  \mathbb{T}_{\mathbb{G}}1 \times  \mathbb{T}_{\mathbb{G}}1 \times L_0(W_0 \times W_0 \times W_0))  \rTo \Sigma_{\mathbb{G}}( \mathbb{T}_{\mathbb{G}}1 \times  \mathbb{T}_{\mathbb{G}}1 \times L_0(W_0 \times W_0))\text{ \ \ \ (a)}
\end{eqnarray*}
and $[ \Sigma_{\mathbb{G}}(\pi_{12},L_0 \pi_{12}),   \Sigma_{\mathbb{G}}(\pi_{23},L_0 \pi_{23})]^{-1}$
\begin{eqnarray*}
& \Sigma_{\mathbb{G}}( \mathbb{T}_{\mathbb{G}}1 \times  \mathbb{T}_{\mathbb{G}}1 \times L_0( W_0 \times W_0)) \times_{ \Sigma_{\mathbb{G}}( \mathbb{T}_{\mathbb{G}}1 \times    L_0 W_0)} \Sigma_{\mathbb{G}}( \mathbb{T}_{\mathbb{G}}1 \times  \mathbb{T}_{\mathbb{G}}1 \times L_0( W_0 \times W_0))  \\ 
& \rTo \Sigma_{\mathbb{G}}( \mathbb{T}_{\mathbb{G}}1 \times  \mathbb{T}_{\mathbb{G}}1   \times  \mathbb{T}_{\mathbb{G}}1   \times    L_0( W_0 \times W_0 \times W_0 ))  \text{ \ \ \ (b)}
\end{eqnarray*}
Apply the lemma  to see that (a) is sending, under the canonical isomorphsms, $(g_1,g_2,[k_1,k_2])$ to $(g_1g_2,k_1k_2)$ and the inverse of (b), i.e. $[ \Sigma_{\mathbb{G}}(\pi_{12},L_0 \pi_{12}),   \Sigma_{\mathbb{G}}(\pi_{23},L_0 \pi_{23})]$, sends $(g_1,g_2,[k_1,k_2])$ to $((g_1,g_2k_1),(g_2,k_2))$ where we are using $[ \_, \_]$ for $(L_0\pi_{12},L_0 \pi_{23})^{-1}$. From this it is clear the multiplcation of $\mathbb{G} \ltimes \frak{K}$ is $((g,k),(g',k')) \mapsto (gg',((g')^{-1}k)k')$ and we have indeed defined semidirect product in the usual manner.

Going in the other direction to the last Proposition, stably Frobenius adjunctions with categories of $\mathbb{H}$-objects as domains give rise to internal groupoids:
\begin{proposition}\label{connectback}
Let $\mathbb{H}$ be a groupoid in $\mathcal{C}$ with a connected components adjunction, $\mathcal{D} \pile{\rTo \\ \lTo} \mathcal{C}$ an adjunction and $L \dashv R : [ \mathbb{H} , \mathcal{C} ] \pile{\rTo \\ \lTo} \mathcal{D}$ a stably Frobenius adjunction over $\mathcal{C}$. Then there exists a stably Frobenius groupoid $\frak{K}_{L,R}$ in $\mathcal{D}$ and an equivalence $\Theta : [ \mathbb{H} , \mathcal{C}] \rTo $$ [ \frak{K}_{L,R} ,  \mathcal{D}  ] $ over $ \mathcal{D}$.
\end{proposition}
\begin{proof}
By Proposition \ref{second2}, using the second condition, we need to find an object $W$ of $\mathcal{D}$ and a stably Frobenius adjunction $\mathcal{D} /W \pile{\rTo \\ \lTo} [ \mathbb{H} , \mathcal{C} ] $ over $\mathcal{D}$ whose right adjoint is monadic. But by Proposition \ref{second2} we know that there is a stably Frobenius adjunction $\mathbb{T}_{\mathbb{H}} \dashv U_{\mathbb{H}} : \mathcal{C}/H_0 \pile{\rTo \\ \lTo} [ \mathbb{H} , \mathcal{C} ] $ over $ \mathcal{C} $ whose right adjoint is monadic, so to progress the proof we exhibit an equivalence $\mathcal{C} / H_0 \simeq \mathcal{D}/ W$. Take $W = L\mathbb{T}_{\mathbb{H}}1$ and notice that Lemma \ref{localic} part 2 can be applied as $L  \mathbb{T}_{\mathbb{H}} \dashv U_{\mathbb{H}} R$ is stably Frobenius because $L \dashv R$ is stably Frobenius.  
\end{proof}
When $\mathcal{D}=[\mathbb{G}, \mathcal{C}]$, for $\mathbb{G}$ stably Frobenius, describing  $\frak{K}_{L,R}$ is easy; it is given by $(L\pi_2,L\pi_1: L (\mathbb{T}_{\mathbb{H}}1 \times \mathbb{T}_{\mathbb{H}}1) \pile{ \rTo \\ \rTo } L \mathbb{T}_{\mathbb{H}}1, ...)$ and $L (\mathbb{T}_{\mathbb{H}}1 \times \mathbb{T}_{\mathbb{H}}1)$ is the $(R_r,e)$ that played a role in the proof of Propositin \ref{HSexample}; so the domain and codomain maps of  $\frak{K}_{L,R}$ are canonically isomorphic to $\pi_2, c: H_1 \times_{H_0} P \pile{\rTo \\ \rTo} P$ where, as before, $(P_p, c: H_1 \times_{H_0} P \rTo P) $ is defined as $R \mathbb{T}_{\mathbb{G}}1$. 
\begin{lemma}\label{GasP}
Given $L \dashv R$ as in the last proposition, with $\mathcal{D}=[\mathbb{G}, \mathcal{C}]$. Then the groupoid $\mathbb{G} \ltimes \frak{K}_{L,R}$ in $\mathcal{C}$ is isomorphic to $\mathbb{P}_{L,R}$, the groupoid corresponding to the Hilsum-Skandalis map $\frak{P}_{L,R}$
\end{lemma}
\begin{proof}
Follows from the explicit descriptions of semidirect product and $\frak{K}_{L,R}$ already provided.
\end{proof}
So, in summary, we see that the stably Frobenius adjunctions corresponding to Hilsum-Skandalis maps are the connected component adjunctions of internal stably Frobenius groupoids. Hilsum-Skandalis maps give rise to stably Frobenius groupoids internal to $[\mathbb{G},\mathcal{C}]$ and vice versa, but these internal groupoids are unique only up to Morita equivalence.

\section{Pullback of connected component adjunctions}
In this section we show that the 2-category $T_{\mathcal{C}}$ defined in the introductory paragraph of Section \ref{invert} has finite pseudo-limits. It clearly has a terminal object as the trivial groupoid $\mathbb{G}=1$ is stably Frobenius. The next proposition shows that $T_{\mathcal{C}}$ has binary pseudo-products, but this is sufficient as we can exploit the results of the previous section to show that a pullback diagram over $[\mathbb{G},\mathcal{C}]$ gives rise to a product diagram in $T_{[\mathbb{G},\mathcal{C}]}$. We will see later that this allows us to conclude that the pullbacks of bounded geometric morphisms over some base topos $\mathcal{S}$ can be calculated in $T_{\mathbf{Loc}_{\mathcal{S}}}$, where $\mathbf{Loc}_{\mathcal{S}}$ is the category of locales over $\mathcal{S}$.

We need a lemma.
\begin{lemma}
If $\mathbb{H}$ and $\mathbb{G}$ are two groupoids in $\mathcal{C}$ then $\mathbb{T}_{\mathbb{H} \times \mathbb{G}}1 \rTo^{(Id_{H_1} \times d_1^{\mathbb{G}},d_1^{\mathbb{H}} \times Id_{G_1})} \pi_1^* \mathbb{T}_{\mathbb{H}}1 \times \pi_2^* \mathbb{T}_{\mathbb{G}}1$ is an $\mathbb{H} \times \mathbb{G}$-isomorphism. If, further, both $\mathbb{H}$ and $\mathbb{G}$ are stably Frobenius then 

(a) for any $\mathbb{H} \times \mathbb{G}$-object $A=(A_{f,g},c,d)$, the map $n_A: A \rTo \Sigma_{\pi_1} A$, which quotients out the $\mathbb{G}$-action, is an $\mathbb{H}$-homomorphism,

(b) $Id \times d_1^{\mathbb{G}}: \mathbb{T}_{\mathbb{H} \times \mathbb{G}}1 \rTo \pi_1^* \mathbb{T}_{\mathbb{H}} 1$ is the adjoint transpose across $\Sigma_{\pi_1} \dashv \pi_1^*$ of the projection $\pi_1 : \mathbb{T}_{\mathbb{H}}1 \times \mathbb{H}^*G_0 \rTo \mathbb{T}_{\mathbb{H}}1$,

(c) for any $\mathbb{H}$-object $B=(B_f,c)$, the morphism 
\begin{eqnarray*}
\Sigma_{\mathbb{H}} \Sigma_{\pi_1}(\mathbb{T}_{\mathbb{H} \times \mathbb{G}}1 \times \pi_1^*B) \rTo^{\Sigma_{\mathbb{H}}(\cong)} \Sigma_{\mathbb{H}}(\Sigma_{\pi_1}\mathbb{T}_{\mathbb{H} \times \mathbb{G}}1 \times B) \rTo^{\Sigma_{\mathbb{H}}(\pi_1 \times Id)} \Sigma_{\mathbb{H}}(\mathbb{T}_{\mathbb{H}}1 \times B)
\end{eqnarray*}
 is canonically isomorphic to $\pi_1: B \times G_0 \rTo B$; and,

(d) the morphism
\begin{eqnarray*}
\Sigma_{\mathbb{H} \times \mathbb{G}}(\mathbb{T}_{\mathbb{H} \times \mathbb{G}}1 \times A) \rTo^{\Sigma_{\mathbb{H} \times \mathbb{G}}( (Id \times d_1^{\mathbb{G}} ) \times Id)} \Sigma_{\mathbb{H} \times \mathbb{G}}(\pi_1^* \mathbb{T}_{\mathbb{H}}1 \times A) \rTo^{\cong} \Sigma_{\mathbb{H}}(\mathbb{T}_{\mathbb{H}}1 \times \Sigma_{\pi_1}A)
\end{eqnarray*}
is canonically isomorphic to $A \rTo^{n_A} \Sigma_{\pi_1}A$. 
\end{lemma}
The functor $\Sigma_{\pi_1} : [ \mathbb{H} \times \mathbb{G}, \mathcal{C}] \rTo $$ [\mathbb{H} , \mathcal{C}]$ used in the statement of the lemma exists by application of Proposition \ref{product}.
\begin{proof}
That the given morphism is an isomorphism is clear from the construction of product in $[\mathbb{H} \times \mathbb{G},\mathcal{C}]$; the pullback of $Id_{H_1} \times d_1^{\mathbb{G}}$ along $d_1^{\mathbb{H}} \times Id_{G_1}$ is $d_1^{\mathbb{H}} \times d_1^{\mathbb{G}}: H_1 \times G_1 \rTo H_0 \times G_0$, i.e. $d_1^{\mathbb{H} \times \mathbb{G}}$.

Part (a) is clear from the construction of the $\mathbb{H}$-action on $\Sigma_{\pi_1}A$ from which the identification of $\Sigma_{\pi_1}\mathbb{T}_{\mathbb{H} \times \mathbb{G}}1$ with $\mathbb{T}_{\mathbb{H}}1 \times \mathbb{H}^*G_0$ is also clear; quotienting out the $\mathbb{G}$-action on $G_1$ returns $G_0$ and from this (b) follows as the quotienting map is $d_1^{\mathbb{G}}$.

For (c), first note that $\Sigma_{\pi_1}(\mathbb{T}_{\mathbb{H} \times \mathbb{G}}1 \times \pi_1^*B)$ is constructed using the quotient $(Id \times d_1^{\mathbb{G}} ) \times (Id \times Id): (H_1 \times G_1 ) \times_{H_0 \times G_0} (B \times G_0) \rTo (H_1 \times G_0) \times_{H_0 \times G_0} (B \times G_0)$. But (c) then follows as

\begin{diagram}
(H_1 \times G_0) \times_{H_0 \times G_0} (B \times G_0) & \rTo^{\cong} & (H_1 \times G_0) \times_{H_0} B & \rTo^{\pi_1 \times Id}  & H_1 \times_{H_0} B\\
\dTo^{\pi_2} & & & & \dTo_{\pi_2}\\
B \times G_0 & & \rTo^{\pi_2} & & B\\
\end{diagram}
clearly commutes where the vertical arrows quotient out the $\mathbb{H}$-action.

The morphism of (d) is $\Sigma_{\mathbb{H}}$ applied to
\begin{eqnarray*}
\Sigma_{\pi_1}(\mathbb{T}_{\mathbb{H} \times \mathbb{G}}1 \times A) \rTo^{\Sigma_{\pi_1}( \tilde{\pi_1} \times Id)} \Sigma_{\pi_1}(\pi_1^* \mathbb{T}_{\mathbb{H}}1 \times A) \rTo^{\cong} \mathbb{T}_{\mathbb{H}}1 \times \Sigma_{\pi_1}A
\end{eqnarray*}
where $\tilde{(\_)}$ is adjoint transpose across $\Sigma_{\pi_1} \dashv \pi_1^*$. As the canonical isomorphism in the definition of Frobenius reciprocity involves a counit this morphism is equal to 
\begin{eqnarray*}
\Sigma_{\pi_1}(\mathbb{T}_{\mathbb{H} \times \mathbb{G}}1 \times A) \rTo^{\Sigma_{\pi_1}(Id \times \eta_A)} & \Sigma_{\pi_1}(\mathbb{T}_{\mathbb{H} \times \mathbb{G}}1 \times \pi_1^* \Sigma_{\pi_1}A) & \rTo^{\cong} \Sigma_{\pi_1}\mathbb{T}_{\mathbb{H} \times \mathbb{G}}1 \times \Sigma_{\pi_1}A \\
  & \rTo^{\pi_1 \times Id} \mathbb{T}_{\mathbb{H}}1 \times \Sigma_{\pi_1}A \text{.}& \\
\end{eqnarray*}
The result then follows by taking $B =\Sigma_{\pi_1}A$ in (c) since the unit $\eta_A$ is the map $A \rTo^{(n_A,g)} \Sigma_{\pi_1}A \times G_0$.
\end{proof}

\begin{proposition}\label{last}
$T_{\mathcal{C}}$ has binary pseudo-products.
\end{proposition}	
\begin{proof}
Let $\mathbb{G}$ and $\mathbb{H}$ be two stably Frobenius groupoids in $\mathcal{C}$. Then, Proposition \ref{product}, $\mathbb{G} \times \mathbb{H}$ is stably Frobenius and the top horizontal and left vertical adjunctions of the diagram
\begin{diagram}
[  \mathbb{H} \times \mathbb{G} , \mathcal{C}] & \pile{\rTo^{\Sigma_{\pi_2}}\\ \lTo_{\pi_2^*}} \ & [ \mathbb{G}, \mathcal{C} ]\\
  \dTo^{\Sigma_{\pi_1}} \uTo_{\pi_1^*}  & &  \dTo^{\Sigma_{\mathbb{G}}} \uTo_{\mathbb{G}^*} \text{ \ \ \ \ (*)}\\
[\mathbb{H},\mathcal{C} ]& \pile{ \rTo^{\Sigma_{\mathbb{H}}} \\ \lTo_{\mathbb{H}^*} } & \mathcal{C} \\    
\end{diagram}
are morphisms of $T_{\mathcal{C}}$. The diagram commutes up to canonical natural isomorphism.

For any stably Frobenius groupoid $\mathbb{K}$ given a diagram 
\begin{diagram}
[\mathbb{K},\mathcal{C}]  & \pile{\rTo^{L_2} \\ \lTo_{R_2} }  & [ \mathbb{G}, \mathcal{C} ]\\
  \dTo^{L_1} \uTo_{R_1}  & &  \dTo^{\Sigma_{\mathbb{G}}} \uTo_{\mathbb{G}^*} \\
[\mathbb{H},\mathcal{C}] & \pile{ \rTo^{\Sigma_{\mathbb{H}}} \\ \lTo_{\mathbb{H}^*} } & \mathcal{C} \\    
\end{diagram}
commuting up to natural isomorphism and over $\mathcal{C}$, we must construct an adjunction $[\mathbb{K},\mathcal{C}]  \pile{ \rTo^{L} \\ \lTo_{R} } [ \mathbb{H} \times \mathbb{G}, \mathcal{C}]$, unique up to natural isomorphism, such that $\Sigma_{\pi_1}L \cong L_1$ and $\Sigma_{\pi_2}L \cong L_2$. Write $\frak{P}_i=((P_i)_{(p_i,q_i)},c_i,d_i)$ for the Hilsum-Skandalis maps corresponding to $L_i \dashv R_i$, for $i=1,2$. Notice that $P_1 \times_{K_0} P_2 \rTo^{q_1 \times q_2} H_0 \times G_0$ can be made into a $\mathbb{H} \times \mathbb{G}$-object ($[(h,g),(x_1,x_2) ]\mapsto (hx_1,gx_2)$)and $p_1 \pi_1 : P_1 \times_{K_0} P_2 \rTo K_0$ into a $\mathbb{K}$-object ($(k,(x_1,x_2)) \mapsto (kx_1,kx_2)$). Given that $p_1 \pi_1 : P_1 \times_{K_0} P_2 \rTo K_0$ is an effective descent morphism (Lemma \ref{edcompose}, applied in $\mathcal{C}/K_0$) it is then easy to check that a Hilsum-Skandalis map, $\frak{P}$, has been defined and so there is a corresponding adjunction $L \dashv R$. If $(Y_g,b)$ is a $\mathbb{K}$-object then $L_1(Y_g,b)$ is obtained by quotienting the $\mathbb{K}$-action of the $\mathbb{K} \times \mathbb{H}$-object $P_1 \times_{K_0} Y$ (the action map is $((k,h),x,y)\mapsto (khx,ky)$). Similarly $L_{\frak{P}}(Y_g,b)$ has underlying object $\Sigma_{\mathbb{K}}((P_1 \times_{K_0} P_2) \times_{K_0} Y)$. Then, since $\Sigma_{\pi_1}$ quotients out the $\mathbb{G}$-action, the process of quotienting commutes (the groupoids are stably Frobenius) and $K_0 \cong \Sigma_{\mathbb{G}}((P_2)_{q_2},d_2)$, it follows that $L_1 \cong \Sigma_{\pi_1} L_{\frak{P}}$. Similarly $L_2 \cong \Sigma_{\pi_2} L_{\frak{P}}$.

To complete we must show that any adjunction $[\mathbb{K},\mathcal{C}]  \pile{ \rTo^{L} \\ \lTo_{R} } [ \mathbb{H} \times \mathbb{G}, \mathcal{C}]$ over $\mathcal{C}$ is isomorphic to $L_{\frak{P}} \dashv R_{\frak{P}}$ where $\frak{P}$ is the Hilsum-Skandalis map whose underlying object is $P_1 \times_{K_0} P_2$, constructed from $\Sigma_{\pi_1} L  \dashv R \pi_1^* $ and $\Sigma_{\pi_2} L \dashv R \pi_2^*$ as in the last paragraph. The lemma shows that $\mathbb{T}_{\mathbb{G} \times \mathbb{H}} 1 \cong \pi_1^* \mathbb{T}_{\mathbb{H}} 1 \times \pi^*_2 \mathbb{T}_{\mathbb{G}} 1$; by applying $R$ we see that the $\mathbb{K}$-object $R\mathbb{T}_{\mathbb{H} \times \mathbb{G}}1$ is isomorphic to $ P_1 \times_{K_0} P_2$.  The result then follows by taking $A=L\mathbb{T}_{\mathbb{K}}1$ in the lemma as it shows that $R(Id_{H_1} \times d_1^{\mathbb{G}}): R \mathbb{T}_{\mathbb{H} \times \mathbb{G}}1 \rTo P_1$ is an $\mathbb{H}$-homomorphism (it corresponds to $n_{L\mathbb{T}_{\mathbb{K}}1}$ under the canonical isomorphisms) and similarly $R(d_1^{\mathbb{H}} \times Id_{G_1}): R \mathbb{T}_{\mathbb{H} \times \mathbb{G}}1 \rTo P_2$ is a $\mathbb{G}$-homomorphism; but then by construction of $P_1 \times_{K_0} P_2$ as an $\mathbb{H} \times \mathbb{G}$-object the morphism $(R(Id_{H_1} \times d_1^{\mathbb{G}}),R(d_1^{\mathbb{H}} \times Id_{G_1}))$ must be an $\mathbb{H} \times \mathbb{G}$-homomorphism and so $\frak{P} \cong \frak{P}_{L,R}$.
\end{proof}

\section{Geometric Morphisms as stably Frobenius adjunctions}
In this section we recall the main result of \cite{towgeom}, which is a representation theorem of geometric morphsms as stably Frobenius adjunctions, and give a correct account of the morphisms in the representation theorem. Familiarity with locale theory and topos theory is assumed; e.g.\cite{Elephant} or \cite{JoyT}. Our toposes are elementary unless stated otherwise. We write $f :\mathcal{F} \rTo \mathcal{E}$ for a geometric morphism between toposes and $\mathbf{Loc}_{\mathcal{E}}$ for the category of locales relative to a topos $\mathcal{E}$. If $X$ is a locale in $\mathcal{F}$ we write $\mathcal{O}_{\mathcal{F}}X$ for its frame of opens. The direct image $f_*$ of any geometric morphisms preserves frames; i.e. $f_*\mathcal{O}_{\mathcal{F}}X= \mathcal{O}_{\mathcal{E}}\Sigma_fX$ for some locale $\Sigma_fX$ in $\mathcal{E}$. This determines a functor $\Sigma_f: \mathbf{Loc}_{\mathcal{F}} \rTo \mathbf{Loc}_{\mathcal{E}}$ which can be shown to have a right adjoint which, abusing notation, is written $f^*: \mathbf{Loc}_{\mathcal{E}} \rTo \mathbf{Loc}_{\mathcal{F}}$. The abuse is reasonable because $f^*$ is the inverse image of $f$ when restricted to discrete locales. We call $\Sigma_f \dashv f^*$ the \emph{pullback adjunction} associated with the geometric morphism $f$. This is reasonable because if $f$ is a geometric morphism between categories of sheaves over two locales, i.e. $f: Sh(X) \rTo Sh(Y)$, then it is uniquely determined by a locale map $f: X \rTo Y$ and under the equivalences $\mathbf{Loc}_{Sh(Z)} \simeq \mathbf{Loc}/Z$, for $Z=X,Y$, the adjunction $\Sigma_f \dashv f^*$ is the pullback adjunction in the usual sense (i.e. $\Sigma_f(Y_g)=Y_{fg}$). 

The following result is clear from \cite{towgeom}:
\begin{proposition}
If $f : \mathcal{F} \rTo \mathcal{E}$ is a geometric morphism between elementary toposes then the adjunction $\Sigma_f \dashv f^*: \mathbf{Loc}_{\mathcal{F}} \pile{\rTo \\ \lTo} \mathbf{Loc}_{\mathcal{E}}$ is stably Frobenius and $f^* \mathbb{S}_{\mathcal{E}} \cong   \mathbb{S}_{\mathcal{F}}$. Further, for every stably Frobenius adjunction $L \dashv R :\mathbf{Loc}_{\mathcal{F}} \pile{\rTo \\ \lTo} \mathbf{Loc}_{\mathcal{E}}$, if $R$ preserves $\mathbb{S}$ then there exists a geometric morphism $f: \mathcal{F} \rTo \mathcal{E}$, unique up natural isomorphism, such that $R \cong f^*$.
\end{proposition}
Here $\mathbb{S}$ is the Sierpi\'{n}ski locale. Its frame of opens is the feee frame on the singleton set $1$. It is an internal distributive lattice in $\mathbf{Loc}$ and when we are asserting that it is preserved by a functor we mean that the implied isomorphism also commutes with the lattice structure. 
\begin{proof}
The proof is in \cite{towgeom}. Two comments are useful. 1. The proof that $\Sigma_f \dashv f^*$ is Frobenius in fact is already clear from Proposition 2(4) of Ch. 6, \S 1 of \cite{JoyT}, and from Proposition 2 of \S 3 of the same Chapter, which establishes $\mathbf{Loc}_{Sh(Z)} \simeq \mathbf{Loc}/Z$, it can be seen that the adjunction $\Sigma_f \dashv f^*$ is stably Frobenius (consider Proposition 4 \S 1). 2. The category of locales is order enriched and $X \times \mathbb{S}$ is the tensor $X \otimes \{0 \leq 1 \}$ for any locale $X$; from this it is clear that a functor between categories of locales preserves the order enrichment provided it preserves the Sierpi\'{n}ski locale.
\end{proof}

To turn this last proposition into a categorical equivalence, one needs to restrict to natural isomorphisms between geometric morphisms and adjunctions. This is because the mapping $f \mapsto (\Sigma_f \dashv f^*)$ is not functorial on natural transformations; if $\alpha: (f_1)_* \rTo (f_2)_*$ is a natural transformation between direct images then the morphism $\alpha_{\Omega_{\mathcal{F}}X}$ is not necessarily a frame homomorphism. But it will be if $\alpha$ is an isomorphism (use the naturality of $\alpha$ to show that  $\alpha_{\Omega_{\mathcal{F}}X}$ is monotone).

However we can account for natural transformations between geometric morphisms in this representation theorem if required. To see this recall that such natural transformations can be represented as $\mathbb{S}$-homotopies. If $f,g: \mathcal{F} \pile{\rTo \\ \rTo} \mathcal{E}$ are two geometric morphisms then a \emph{$\mathbb{S}$-homotopy from $f$ to $g$} consists of a geometric morphism $H: Sh_{\mathcal{F}}\mathbb{S}_{\mathcal{F}} \rTo \mathcal{E}$ such that $f=H0_{\mathbb{S}_{\mathcal{F}}}$ and $g=H1_{\mathbb{S}_{\mathcal{F}}}$ where $0_{\mathbb{S}_{\mathcal{F}}}$ and $1_{\mathbb{S}_{\mathcal{F}}}$ are the geometric morphisms corresponding to the locale maps that are the bottom and top points of the Sierpi\'{n}ski locale $\mathbb{S}_{\mathcal{F}}$. The topos $Sh_{\mathcal{F}}\mathbb{S}_{\mathcal{F}}$ can be described explicitly, it is the presheaf category $[\mathbf{2},\mathcal{F}]$, where $\mathbf{2}$ is the category $(0 \rTo 1)$, i.e. two objects and one non-identity morphism. From the definition of geometric morphism it is then clear that $\mathbb{S}$-homotopies from $f$ to $g$ are in natural bijection with natural transformations from $f^*$ to $g^*$. Define a \emph{localic $\mathbb{S}$-homotopy from $f$ to $g$} to be a stably Frobenius adjunction $L \dashv R : \mathbf{Loc}_{\mathcal{F}}/ \mathbb{S}_{\mathcal{F}} \pile{ \rTo \\ \lTo} \mathbf{Loc}_{\mathcal{E}}$ such that $(0_{\mathbb{S}_{\mathcal{F}}})^*R = f^*$ and  $(1_{\mathbb{S}_{\mathcal{F}}})^*R = g^*$. 
So,
\begin{proposition}
For any two toposes $\mathcal{F}$ and $\mathcal{E}$, the category of geometric morphisms from $\mathcal{F}$ to $\mathcal{E}$, i.e. $\frak{Top}(\mathcal{F},\mathcal{E})$, is equivalent to the category of stably Frobenius adjunctions $L \dashv R : \mathbf{Loc}_{\mathcal{F}} \pile{ \rTo \\ \lTo} \mathbf{Loc}_{\mathcal{E}}$ such that $R$ preserves the Sierpi\'{n}ski locale. The morphisms of $\frak{Top}(\mathcal{F},\mathcal{E})$ are natural transformations (between the inverse images) and the morphisms of the category of stably Frobenius adjunctions are localic $\mathbb{S}$-homotopies.
\end{proposition}

\section{Bounded geometric morphisms} 
Consider a geometric morphism  $f : \mathcal{F} \rTo \mathcal{E}$. Section 4 of \cite{towgroth} shows that  condition (4') of Proposition \ref{bounded} holds for the adjunction $\Sigma_f \dashv f^*$ if $f$ is bounded. On the other hand suppose that condition (4') of Proposition \ref{bounded} holds. Then, as we have recalled that $\Sigma_f \dashv f^*$ is stably Frobenius, we know that $\mathbf{Loc}_{\mathcal{F}}$ is equivalent to $[ \mathbb{G}, \mathbf{Loc}_{\mathcal{E}}]$ over $\mathbf{Loc}_{\mathcal{E}}$. By restricting to discrete objects we can then conclude that $f$ is equivalent to the canonical geometric morphism $\gamma: B_{\mathcal{E}} \mathbb{G} \rTo \mathcal{E}$ which is well known to be bounded (e.g. B3.4.14(b) of \cite{Elephant} or see the definition of the relevant site in \S6 of \cite{MoerClassTop}) and so $f$ is bounded. In summary:
\begin{lemma}
A geometric morphism  $f : \mathcal{F} \rTo \mathcal{E}$ is bounded if and only if  $\Sigma_f \dashv f^* : \mathbf{Loc}_{\mathcal{F}} \pile{\rTo \\ \lTo} \mathbf{Loc}_{\mathcal{E}}$ satisfies one of the equivalent conditions of either Proposition \ref{bounded} or Proposition \ref{second2}.
\end{lemma}
This lemma captures the essence of the famous Joyal and Tierney result that any topos $\mathcal{F}$ bounded over $\mathcal{E}$ is equivalent to $B_{\mathcal{E}}\mathbb{G}$ for some open localic groupoid $\mathbb{G}$; see \cite{towgroth} for more detail.
We write $\frak{BTop}^i$ for the 2-category whose objects are toposes, morphisms bounded geometric morphisms and whose 2-cells are natural isomorphisms. 
\begin{theorem}\label{i}
Let $\mathcal{S}$ be a topos with a natural numbers object. Sending bounded $p: \mathcal{E} \rTo \mathcal{S}$ to $\Sigma_p \dashv p^*$ determines a pseudo-functor $\frak{i}:\frak{BTop}^i/\mathcal{S} \rTo T_{\mathbf{Loc}_{\mathcal{S}}}$. This pseudo-functor preserves binary pseudo-products and is full and faithful, up to isomorphism.
\end{theorem}
The theorem can be adapted to show that all finite pseudo-limits are preserved; this is omitted in the interests of space.
\begin{proof}
The lemma shows how to construct $\frak{i}$, though of course the construction of a localic groupoid for any bounded $p$ can be taken to be that defined via the well known Joyal-Tierney theorem \cite{JoyT}. In particular we recall (C5.3 of \cite{Elephant}) that given the bounded geometric morphism $\gamma: B_{\mathcal{E}} \mathbb{G} \rTo \mathcal{E}$, the pseudo-functor returns $[\hat{\mathbb{G}},\mathbf{Loc}_{\mathcal{S}}]$ where $\hat{\mathbb{G}}$ is the \'{e}tale completion of $\mathbb{G}$. Therefore for any \'{e}tale complete $\mathbb{G}$, $[\mathbb{G},\mathbf{Loc}_{\mathcal{S}}]$ is equivalent to $\frak{i}(\gamma)$ for some bounded geometric morphism $\gamma$. 

Since every bounded $p$ is equivalent to one of the form $\gamma: B_{\mathcal{E}} \mathbb{G} \rTo \mathcal{E}$ to complete the proof of the assertion about binary pseudo-products, given Proposition \ref{last}, one just needs to check that $\mathbb{H} \times \mathbb{G}$ is \'{e}tale complete if $\mathbb{G}$ and $\mathbb{H}$ are. But $[ \mathbb{H} \times \mathbb{G} , \mathbf{Loc}_{\mathcal{S}}]$ is equivalent to $[ \mathbb{H}^* \mathbb{G} , [\mathbb{H}, \mathbf{Loc}_{\mathcal{S}}]]$ and Lemma 7.3 of \cite{MoerClassTop} shows that $ \mathbb{H}^* \mathbb{G}$ is \'{e}tale complete if $\mathbb{G}$ is.

For the claim of `full and faithful' use the the representation theorem of geometric morphisms as stably Frobenius adjunctions. Note that as $f^*$ preserves the Sierpi\'{n}ski locale for any geometric morphism $f$, $\mathbb{G}^*\mathbb{S}_{\mathcal{S}}$ is the Sierpi\'{n}ski locale of $\mathcal{E}$, where $\mathbb{G}$ is the localic groupoid associated with $\mathcal{E}$; similarly for $\mathcal{F}$. Therefore the right adjoint of $L \dashv R$ must preserve the Sierpi\'{n}ski locale for any adjunction $[\mathbb{H},\mathbf{Loc}_{\mathcal{F}}] \pile{\rTo \\ \lTo} [\mathbb{G},\mathbf{Loc}_{\mathcal{E}}] $ provided it is over $\mathbf{Loc}_{\mathcal{S}}$.

\end{proof}

\begin{remark}
One can expand this result to include natural transformations between geometric morphisms that are not necessarily isomorphisms; the morphisms of $T_{\mathbf{Loc}_{\mathcal{S}}}$ then need to be replaced with $\mathbb{S}$-homotopies. We note that if $\frak{P}_1,\frak{P}_2: \mathbb{H} \rTo \mathbb{G}$ are two Hilsum-Skandalis maps between stably Frobenius localic groupoids then a localic $\mathbb{S}$-homotopy is given by a Hilsum-Skandalis map $\frak{P}: \mathbb{H} \rTo \mathbb{G}$ with the additional data of an open $a$ of $P$ (equivalently a locale map $a: P \rTo \mathbb{S}$) such that $P_1$ is the open sublocale $a^*1_{\mathbb{S}}$ of $P$ and $P_2$ is the closed sublocale $a^*0_{\mathbb{S}}$ of $P$. This provides a localic characterisation of natural transformations between geometric mophisms (for bounded toposes at least). 
\end{remark}

\begin{remark}
It is unfortunate that this result uses Lemma 7.3 of \cite{MoerClassTop} as the proof of Lemma 7.3 requires discussion of sites and so we have not fully achieved a site-free description of the pullback of bounded geometric morphisms. However it is not unreasonable to expect that a site-free characterisation of \'{e}tale completeness is possible and that this can be shown to be pullback stable without reference to sites. Indeed this has been done for localic groups, see Lemma 1.4 of \cite{MoerMorita}.
\end{remark}

\section{Further applications to geometric morphisms}

In this section we prove a number of results about geometric morphisms using their representation as stably Frobenius adjunctions and applying our results about stably Frobenius adjunctions. The results are all essentially known; what is new is that they all follow from our earlier results about Frobenius adjunctions. 

\begin{proposition}
Let $f: \mathcal{F} \rTo \mathcal{E}$ and $g: \mathcal{G} \rTo \mathcal{F}$ be two geometric morphisms. Then (i) $fg$ is localic(bounded) if $f$ and $g$ are, (ii) if $fg$ is localic then $g$ is localic; and, (iii) if $fg$ is bounded then $g$ is bounded.
\end{proposition}
Recall that a geometric morphism $f : \mathcal{F} \rTo \mathcal{E}$ is \emph{localic} if there is a locale $X$ in $\mathcal{E}$ and an equivalence $\mathcal{F} \rTo^{\simeq} Sh_{\mathcal{E}} X$ over $\mathcal{E}$. An equivalent condition on the pullback adjunction $\Sigma_f \dashv f^*$ is that there exists a locale $X$ in $\mathcal{E}$ and an equivalence $\mathbf{Loc}_{\mathcal{F}} \rTo^{\simeq} \mathbf{Loc}_{\mathcal{E}} / X $ over $\mathbf{Loc}_{\mathcal{E}}$; localic geometric morphisms are bounded (consider the trivial localic groupoid $\mathbb{X}$).
\begin{proof}
(i) Certainly if $f$ and $g$ are localic then so is $fg$ as a slice composed with a slice is again a sliced adjunction.
Say $f$ and $g$ are bounded then Proposition \ref{connect} shows that the adjunction $\Sigma_{fg} \dashv (fg)^*$ satisfies the conditions of Proposition \ref{second2} and so $fg$ is bounded.

(ii) Say $fg$ is localic then (2.) of Lemma \ref{localic} shows that $\Sigma_g \dashv g^*$ is a slice and so $g$ is localic. 

(iii) Proposition \ref{connectback}.
\end{proof}

\begin{proposition}
Any geometric morphism $f:\mathcal{F} \rTo \mathcal{E}$ factors, uniquely up to isomorphsm, as a hyperconnected followed by a localic geometric morphism.
\end{proposition}
Recall that $f$ is \emph{hyperconnected} if the canonical map $\Omega_{\mathcal{E}} \rTo f_* \Omega_{\mathcal{F}}$ (i.e. the unique frame homomorphsm) is an isomorphism. An equivalent condition is that $\Sigma_f$ preserves $1$.

\begin{proof}
Repeating the remark just before Lemma \ref{firstlemma} we know that there is a facotization:
\begin{eqnarray*}
\mathbf{Loc}_{\mathcal{F}} \pile{\rTo^{\Sigma_{\eta_1}} \\ \lTo_{\eta_1^*}} \mathbf{Loc}_{\mathcal{F}}/f^* \Sigma_f 1  \pile{ \rTo^{(\Sigma_f)_{\Sigma_f 1}} \\ \lTo_{f^*_{\Sigma_f 1}}}   \mathbf{Loc}_{\mathcal{E}}/\Sigma_f 1 \pile{\rTo^{\Sigma_{\Sigma_f 1}} \\ \lTo_{(\Sigma_f 1)^*}} \mathbf{Loc}_{\mathcal{E}}
\end{eqnarray*}
Because $\Sigma_f \dashv f^*$ is stably Frobenius and pullback adjunctions are stably Frobenius the composite of the first two adjunctions, $(\Sigma_f)_{\Sigma_f 1} \Sigma_{\eta_1} \dashv \eta_1^* f^*_{\Sigma_f 1}$, is stably Frobenius. But the right adjoint of this composite adjunction preserves the Sierpi\'{n}ski locale and so it must be the pullback adjunction of a geometric morphism, $\bar{f}$; and as $(\Sigma_f)_{\Sigma_f 1} \Sigma_{\eta_1} 1 = 1$, $\bar{f}$ is hyperconnected and this completes the proof of the existence of the factorization because pullback adjunctions correspond to the pullback adjunctions of localic geometric morphisms. 

Let us say we have another factorization of $f$ as $\mathcal{F} \rTo^p Sh_{\mathcal{E}}X \rTo \mathcal{E}$ with $p$ hyperconnected. Because $\Sigma_p1 \cong 1$ it follows that $X \cong \Sigma_f1$ (use $\Sigma_X1 \cong X$) and so there is an equivalence $\psi: \mathbf{Loc}_{\mathcal{E}}/X \rTo^{\simeq} \mathbf{Loc}_{\mathcal{E}}/\Sigma_f 1$ over $\mathbf{Loc}_{\mathcal{E}}$. To complete it just needs to be checked that $\psi p$ is isomorphic to $\bar{f}$. But any object $Y_g$ of $\mathbf{Loc}_{\mathcal{E}}/X$ can be written as an equalizer
\begin{eqnarray*}
Y_g \rTo X^*Y  \pile{\rTo^{X^*f} \\ \rTo_{\Delta_X !}} X^*X
\end{eqnarray*}
and so $p^*Y_g$ is determined by the image of this equalizer. Note that $1 \rTo^{p^*\Delta_X} p^*(X^*X)$ is the double adjoint transpose (via $\Sigma_p \dashv p^*$ and $\Sigma_X \dashv X^*$) of the identity $\Sigma_X \Sigma_p 1 \rTo^{Id_X} \Sigma_X \Sigma_p 1$; but this identity is isomorphic to the identity $\Sigma_{\Sigma_f1} \Sigma_{\bar{f}} 1 \rTo^{Id_{\Sigma_f1}} \Sigma_{\Sigma_f1} \Sigma_{\bar{f}} 1 $ and so has isomorphic double adjoint transpose (now via $\Sigma_{\bar{f}} \dashv \bar{f}^*$ and $\Sigma_{\Sigma_f1} \dashv \Sigma_f  1^*$). It follows that $p^*\Delta_X$ is isomorphic to $\bar{f}^* \Delta_{\Sigma_f1}$.
\end{proof}

\begin{proposition}
The hyperconnected-localic factorization of a bounded geometric morphism along a bounded geometric morphism is pullback stable. 
\end{proposition}
\begin{proof}
If $X$ is an object of a category $\mathcal{C}$ and $\mathbb{G}$ an internal groupoid, then there is an isomorphism between $[\mathbb{X} \times \mathbb{G}, \mathcal{C}]$ and $[\mathbb{G},\mathcal{C}]/\mathbb{G}^*X$. So that localic geometric morphisms are pullback stable follows from the description of product given in Proposition \ref{last}. 
For the pullback stability of hyperconnected geometric morphisms, this follows as the diagram (*) in the proof of Proposition \ref{last} satisfies Beck-Chevalley (both ways round). That the diagram satisfies Beck-Chevalley is clear from the explicit description of $\Sigma_{\pi_2}$ given in (i) of Proposition \ref{product}.
\end{proof}

\begin{proposition}
Open and proper bounded geometric morphisms are pullback stable (along bounded geometric morphisms).
\end{proposition}
\begin{proof}
A geometric morphism is open(proper) if and only if the localic part of its hyperconnected-localic factorization is open(proper); see C3.1.9 (preamble to C3.2.12) of \cite{Elephant}. The result then follows from the previous proposition as it is well known that $f^*: \mathbf{Loc}_{\mathcal{E}} \rTo \mathbf{Loc}_{\mathcal{F}}$ preserves both open and proper locales for any geometric morphism $f: \mathcal{F} \rTo \mathcal{E}$. 
\end{proof}

\begin{proposition}
For any two toposes $\mathcal{F}$ and $\mathcal{E}$, bounded over some base topos $\mathcal{S}$, the category of geometric morphisms from $\mathcal{F}$ to $\mathcal{E}$ over $\mathcal{S}$ (with natural isomorphisms as morphisms) is equivalent to the category of Hilsum-Skandalis maps from $\mathbb{H}$ to $\mathbb{G}$ where $\mathbb{H}$($\mathbb{G}$) is the localic groupoid associated with $\mathcal{F}$($\mathcal{E}$).
\end{proposition}
\begin{proof}
Immediate from our main result (Theorem \ref{HS}) and our earlier result, contained in Theorem \ref{i}, that $\frak{i}$ is full and faithful up to isomorphism. 

\end{proof}
\begin{proposition}
For any bounded $p: \mathcal{E} \rTo \mathcal{S}$ there is a localic groupoid $\mathbb{G}=(G_1 \pile{\rTo \\ \rTo} G_0, ...)$ such that for every point $x: 1 \rTo G_0$ there is a geometric morphism $p_x:  \mathcal{S} \rTo \mathcal{E}$ over $\mathcal{S}$ (i.e. a point of $\mathcal{E}$). The full subgroupoid of $\frak{BTop}^i(\mathcal{S},\mathcal{E})$ consisting of all $p_x$s is isomorphic to the groupoid (in $\mathcal{E}$) of points of $\mathbb{G}$.
\end{proposition}
\begin{proof}
Proposition \ref{recoverG}.
\end{proof}

\section{Concluding remarks}
Whilst writing this paper the idea of changing the title to `Everything you always wanted to know about a cartesian category (but were afraid to ask)' occurred to me. At the outset I did not think that generalities on cartesian categories were going to produce any work of substance. The results would either be well known or just not true at that level of generality. The core result on the representation of Hilsum-Skandalis maps was clearly going to require some work, but I was surprised by the depth of the more basic lemmas, notably our earlier ones on slicing adjunctions (e.g. Lemma \ref{firstlemma}) and, as should be apparent from the exposition, the Proposition characterising connected component adjunctions. So, at the very least, this paper should be of use to those who want to see basic results on sliced adjunctions, stably Frobenius adjunctions and groupoids expressed at the greatest possible level of generality (`greatest' because the cartesian structure is needed to express the stably Frobenius condition and in the definition of internal groupoids). 

Open surjections in the category of locales are effective descent morphisms and are stable under all the relevant constructions (basically, pullback). So by replacing `effective descent morphism' with `open surjection' one can return to a topological `terra firma' should that be required. But our essential point is that the relationships persist at a greater level of generality (cartesian categories) and so should be of general interest. The additional generality is not vacuous because proper surjections in the category of locales provide a non-trivial example, different to open surjections.

The essence of the main result on Hilsum-Skandalis maps is already in \cite{towprinc}, from where the ideas of this paper stem. That localic Hilsum-Skandalis maps represent geometric morphisms is implicit in the literature as they are known to correspond to the morphisms of the category of fractions of open groupoids, and geometric morphisms are known to be represented by the morphsms of this category of fractions; see Section 7 of \cite{MoerClassTop}. What is new for this paper is that we have provided a new explicit charcterisation of the morphisms of the relevant category of fractions: they are the stably Frobenius adjunctions.

It has not been possible here to provide a localic characterisation of \'{e}tale completeness of localic groupoids. Therefore some extra work remains to support any assertion that we have `done topos theory without sites'. But I think it is fair to say that most of the heavy lifting has been done above, notably with Proposition \ref{last}. The idea of reasoning about categories of locales to recover known aspects of topos theory is well motivated, but needs to be finished!

However, what I believe to be of greater interest is working out what happens when we forget about \'{e}tale completeness. The paper \cite{towgro} shows that $[\mathbb{G},\mathbf{Loc}]$ behaves like a perfectly good category of locales even if $\mathbb{G}$ is not \'{e}tale complete and so $[\mathbb{G},\mathbf{Loc}]$ is not actually a category of locales over some topos. What we have shown here is that the theory of `geometric morphisms'(=stably Frobenius adjunctions; let's call them \emph{Frobenius morphisms}) works at a greater level of generality. There is an account of hyperconnected-localic factorization, boundedness, pullbacks, open and proper maps and, e.g. \cite{towslice}, aspects of internal lattice theory, all with `categories of locales' taking on the role of toposes and Frobenius morphisms taking on the role of geometric morphisms. And the key point of this paper is that Frobenius mophisms can be represented as Hilsum-Skandalis maps (i.e. generalised principal bundles). There is further work required to see how useful this new context is. As a final idea consider the question `What should the natural numbers object look like in this context?' We propose the following axiom of infinity: there exists an internal groupoid $\mathbb{G}$ and a natural categorical equivalence $\mathbf{Dis}_{\mathcal{D}} \simeq T_{\mathcal{C}}(\mathcal{D},[\mathbb{G},\mathcal{C}])$ for each $\mathcal{D}$ over $\mathcal{C}$. In other words, $\mathbb{G}$ classifies discrete objects. Restricting to toposes we get a characterisation of the natural numbers (B4.2.11 of \cite{Elephant}). What happens if we don't make that restriction?

\section{Acknowledgement}
This paper is dedicated to my son, Nathan.

\end{document}